\documentclass[11pt,a4paper]{article}
\usepackage[T1]{fontenc}
\usepackage{lmodern,amsmath,amssymb,amsthm,mathtools}
\usepackage[margin=27mm]{geometry}
\usepackage{microtype,booktabs,enumitem,tikz,caption,placeins}
\tikzset{
  pst graph/.style={x=1cm,y=1cm,line width=.55pt,line cap=round,
    line join=round,every node/.style={font=\small,inner sep=2pt}},
  pst vertex/.style={circle,fill=black,draw=none,inner sep=0pt,
    minimum size=3.4pt,outer sep=0pt},
  pst terminal/.style={circle,draw=black,fill=white,line width=.7pt,
    inner sep=0pt,minimum size=6.4pt,outer sep=0pt},
  pst region/.style={fill=black!3,draw=black!40,line width=.4pt},
  pst note/.style={font=\footnotesize}
}
\newcommand{\pstvdots}[2]{%
  \foreach \dy in {-.105,0,.105}{\fill (#1,{#2+\dy}) circle (.5pt);}}
\newcommand{\psthdots}[2]{%
  \foreach \dx in {-.105,0,.105}{\fill ({#1+\dx},#2) circle (.5pt);}}

\usepackage[colorlinks=true,linkcolor=blue!45!black,citecolor=blue!45!black,urlcolor=blue!45!black]{hyperref}
\hypersetup{
  pdftitle={Perfect state transfer on unicyclic graphs},
  pdfauthor={Xingkun Song},
  pdfkeywords={perfect state transfer, unicyclic graphs, strong cospectrality, rooted trees, graph spectra}
}
\setlist{topsep=4pt,itemsep=2pt,parsep=0pt}
\newtheorem{theorem}{Theorem}[section]
\newtheoremstyle{citedplain}
  {\topsep}{\topsep}{\itshape}{}{\bfseries}{.}{.5em}
  {\thmname{#1}\thmnumber{ #2}\thmnote{ \normalfont #3}}
\theoremstyle{citedplain}
\newtheorem{lemma}[theorem]{Lemma}
\theoremstyle{plain}
\newtheorem{proposition}[theorem]{Proposition}
\theoremstyle{citedplain}
\newtheorem{corollary}[theorem]{Corollary}
\theoremstyle{plain}
\theoremstyle{definition}
\newtheorem{example}[theorem]{Example}
\newtheorem{question}[theorem]{Question}
\theoremstyle{remark}
\newtheorem{remark}[theorem]{Remark}
\numberwithin{equation}{section}
\newcommand{\R}{\mathbb R}
\newcommand{\Z}{\mathbb Z}
\newcommand{\F}{\mathbb F}
\newcommand{\ee}{\mathbf e}

\newcommand{\diag}{\operatorname{diag}}
\newcommand{\dist}{\operatorname{dist}}
\newcommand{\tr}{\operatorname{tr}}
\makeatletter
\newcommand{\proofheading}[1]{%
  \if@inlabel
    \noindent{\normalfont\bfseries\boldmath #1}\enspace\ignorespaces
  \else
    \par\addvspace{0.35\baselineskip}%
    \noindent{\normalfont\bfseries\boldmath #1}\enspace\ignorespaces
  \fi
}
\makeatother
\title{Perfect state transfer on unicyclic graphs}
\author{Xingkun Song$^{1,2}$\\
{\small $^{1}$ School of Mathematics and Statistics, Qinghai Minzu University,}\\[-2pt]
{\small Xining, Qinghai 810007, P.R. China}\\
{\small $^{2}$ Qinghai Institute of Applied Mathematics,}\\[-2pt]
{\small Xining, Qinghai 810007, P.R. China}\\
{\small \href{mailto:xksong@126.com}{xksong@126.com}}}
\date{}
\begin{document}
\maketitle
\begin{abstract}
We characterize adjacency perfect state transfer in connected unicyclic graphs: the transfer vertices are opposite degree-two vertices of a four-cycle, and the two attached trees satisfy a characteristic-polynomial condition. Except for the four-cycle, the transfer pair is unique and the minimum transfer time is $\pi/\sqrt2$. Outside the four-cycle, we obtain sharp lower bounds of $46$ on the order and $10$ on the maximum degree, and classify all examples of maximum degree at most $20$. A finite inverse construction enumerates all realizations of a prescribed eigenvalue support at the transfer vertices. For supports of size at least seven, we establish the sharp maximum-degree lower bound of $21$. Every seven-element support $\{0,\pm\sqrt2,\pm u\sqrt2,\pm v\sqrt2\}$, with odd integers $3\leq u<v$, is realizable, and every such realization has attachment depth at most four. Explicit branching rules give a minimum-order recurrence. Among all connected unicyclic graphs admitting perfect state transfer with seven-element eigenvalue support, the minimum order is $218$, attained uniquely up to isomorphism.
\end{abstract}



\section{Introduction}\label{sec:introduction}
Quantum state transfer through spin networks was proposed by Bose \cite{bose}. In an XX spin network with uniform couplings, the single-excitation dynamics is described by the adjacency matrix of its interaction graph \cite{christandl}. For a graph $G$ with adjacency matrix $A$, write
\[
 U_G(t)=\exp(-itA).
\]
There is \textbf{perfect state transfer} (PST) between distinct vertices $a,b$ if
\[
 U_G(\tau)\ee_a=\gamma\ee_b
\]
for some $\tau>0$ and $|\gamma|=1$. The scalar $\gamma$ is the transfer phase. Here $\ee_v$ is the standard basis vector indexed by $v$, and $\tau_0$ denotes the minimum positive transfer time for the pair under consideration when transfer occurs. All graphs are finite, simple, undirected, and unweighted. Graphs considered for transfer are connected; deletion may produce disconnected auxiliary graphs. A connected graph is \textbf{unicyclic} if it has exactly one cycle, or equivalently if $|E(G)|=|V(G)|$.
We write $P_m$, $C_m$, and $K_m$ for the path, cycle, and complete graph on $m$ vertices, respectively, and $K_{r,s}$ for the complete bipartite graph with part sizes $r,s$. Write $N(v)$ for the open neighborhood of $v$, $\dist(u,v)$ for graph distance, and $\Delta(G)$ for the maximum degree.

PST requires strong cospectrality and arithmetic restrictions on the supported eigenvalues \cite{godsil,coutinho}. For sparse graphs, the structural problem is to determine which graphs can satisfy these conditions. Coutinho, Juliano, and Spier proved that the only trees admitting adjacency PST are $P_2$ and $P_3$ \cite{trees}. We study the next case, obtained by adding one edge to a tree: connected graphs with one cycle. Unlike trees, this class contains infinitely many examples. We determine the cycle and transfer vertices, characterize the attached trees by a polynomial condition, and use this condition to construct all realizations of a prescribed eigenvalue support. The resulting description also yields sharp order and degree bounds.

For disjoint rooted trees $(T_c,c)$ and $(T_d,d)$, let $G(T_c,T_d)$ be obtained from the cycle $a,c,b,d,a$ by identifying $c,d$ with the respective tree roots. The trees include their roots and may be singletons. Nonadjacent vertices with the same open neighborhood are called \textbf{false twins}; here $N(a)=N(b)=\{c,d\}$, so both have degree two.
The depth of a rooted tree is the maximum distance from its root.

Kirkland, Monterde, and Plosker \cite[Theorems~9 and~10]{twins} characterized transfer between twins spectrally. Here the first step is structural: every transfer pair in a connected unicyclic graph must consist of degree-two false twins.

\begin{theorem}\label{thm:geometry}
Let $G$ be a connected unicyclic graph admitting perfect state transfer between distinct vertices $a$ and $b$. Then:
\begin{enumerate}[label=\textup{(\roman*)}]
\item The unique cycle of $G$ is an induced four-cycle $a,c,b,d,a$, and $a,b$ are degree-two false twins in $G$, so that $G$ has the form $G(T_c,T_d)$ for rooted trees $T_c,T_d$.
\item If $G\not\cong C_4$, then $\{a,b\}$ is the unique unordered pair admitting perfect state transfer.
\item The minimum transfer time between $a$ and $b$ is $\pi/2$ when $G\cong C_4$ and $\pi/\sqrt2$ otherwise, with a transfer phase of $-1$ at every transfer time.
\end{enumerate}
\end{theorem}

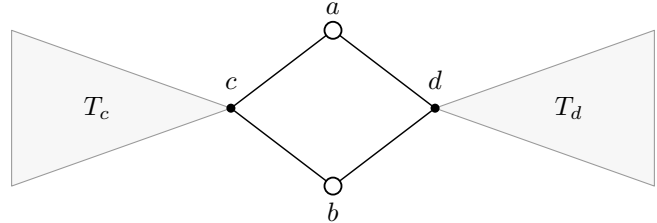
\begin{figure}[htbp]
\centering
\begin{tikzpicture}[pst graph]
  \path[pst region] (-1.35,0)--(-4.25,1.03)--(-4.25,-1.03)--cycle;
  \path[pst region] (1.35,0)--(4.25,1.03)--(4.25,-1.03)--cycle;
  \node at (-3.12,0) {$T_c$};
  \node at (3.12,0) {$T_d$};
  \node[pst vertex] (c) at (-1.35,0) {};
  \node[pst vertex] (d) at (1.35,0) {};
  \node[pst terminal] (a) at (0,1.03) {};
  \node[pst terminal] (b) at (0,-1.03) {};
  \draw (a)--(c)--(b)--(d)--(a);
  \node[above=4pt] at (a) {$a$};
  \node[below=4pt] at (b) {$b$};
  \node[above=5pt] at (c) {$c$};
  \node[above=5pt] at (d) {$d$};
\end{tikzpicture}
\caption{The unicyclic structure required for perfect state transfer. Open circles mark the transfer vertices $a,b$. Shaded regions represent trees rooted at $c,d$, possibly singletons; their boundaries are not graph edges. The attachments must satisfy Theorem~\ref{thm:criterion}.}
\label{fig:structure}
\end{figure}

To state the condition on the attached trees, let $\phi_H(t)=\det(tI-A(H))$ for any graph $H$, with $\phi_{\varnothing}=1$. For $i\in\{c,d\}$ put
\[
 P_i(t)=\phi_{T_i}(t),\qquad Q_i(t)=\phi_{T_i-i}(t).
\]
Define the monic polynomials
\begin{align}
 F(t)&=tP_c(t)P_d(t)-2\bigl(Q_c(t)P_d(t)+P_c(t)Q_d(t)\bigr),\notag\\
 M(t)&=\frac{F(t)}{\gcd(F(t),P_c(t)P_d(t))},\label{eq:Mcriterion}
\end{align}
where the greatest common divisor is monic and is taken in $\mathbb Q[t]$.

\begin{theorem}\label{thm:criterion}
Let $G$ be a connected unicyclic graph, and let $a,b$ be distinct vertices. Then $G$ admits perfect state transfer between $a,b$ if and only if it has the form $G(T_c,T_d)$, with $a,b$ as the designated opposite vertices, for rooted trees $T_c,T_d$ such that the polynomial $M$ in \eqref{eq:Mcriterion} satisfies one of the following conditions:
\begin{enumerate}[label=\textup{(\roman*)}]
\item $M(t)=t^2-4$, which holds exactly when both rooted trees are singletons, or equivalently $G\cong C_4$;
\item $M(t)=\prod_{q\in S}(t^2-2q^2)$ for a finite set $S$ of positive odd integers satisfying $1\in S$ and $|S|\geq2$.
\end{enumerate}
\end{theorem}

Theorems~\ref{thm:geometry} and~\ref{thm:criterion} give a structural
and arithmetic characterization. The roots of $M$ are exactly the
eigenvalues in the support of $(\ee_a+\ee_b)/\sqrt2$; the remaining
eigenvalues need not be integers or quadratic integers. Coutinho and Godsil
\cite[Theorem~2.1]{algorithm} proved that PST can be
decided in polynomial time for general graphs. Here the test uses the
characteristic polynomials of $T_c,T_d$, computed by the recursion in
Section~\ref{sec:trees}.

The criterion also leads to an inverse problem: which trees realize a
prescribed eigenvalue support? Write $\Phi_a$ for the \textbf{eigenvalue support}
of $a$, the set of eigenvalues whose eigenspaces are not orthogonal to
$\ee_a$ \cite[Section~2.1]{trees}. For a transfer vertex $a$ in a graph
other than $C_4$, the criterion
gives $\Phi_a=\{0\}\cup\{\pm q\sqrt2:q\in S\}$, with
$|\Phi_a|=2|S|+1$.
Theorem~\ref{thm:recursive-construction} gives a finite enumeration of all
realizations, with attachment depth at most $2|S|-2$.
For $|S|=3$, the branching equations can be solved explicitly for every
permitted sum of the root degrees, leading to the following result.
\begin{theorem}\label{thm:seven-summary}
The following statements hold.
\begin{enumerate}[label=\textup{(\roman*)}]
\item For every pair of odd integers $3\leq u<v$ and every integer
$R$ with $2u^2-4<R<2v^2-4$, there is a connected unicyclic graph $G(T_c,T_d)$ admitting perfect state transfer between $a,b$ with
\[
 \deg_{T_c}(c)+\deg_{T_d}(d)=R,\qquad
 \Phi_a=\{0,\pm\sqrt2,\pm u\sqrt2,\pm v\sqrt2\}.
\]
\item In every connected unicyclic graph admitting perfect state transfer with seven-element eigenvalue support at the transfer vertices, both attached trees have depth at most four.
\item Every connected unicyclic graph admitting perfect state transfer with seven-element eigenvalue support at the transfer vertices has at least $218$ vertices. Equality holds if and only if the graph is isomorphic to the graph in Example~\ref{ex:218}, which has $(u,v,R)=(3,5,45)$.
\end{enumerate}
\end{theorem}
The branching rules and realization theorem are proved in
Section~\ref{sec:seven}. The minimum-order calculation appears in
Section~\ref{sec:examples}; Appendix~\ref{app:seven-certificate} supplies
the complete arithmetic exclusions, without computer enumeration.
Realizability for every prescribed $S$ with $|S|\geq4$ remains open.

We next state the extremal consequences. The sharp global order and
maximum-degree bounds are attained in a family with five-element eigenvalue support. For nonnegative integers $p,r$, define $H_{p,r}$ by starting with the four-cycle $a,c,b,d,a$, attaching $p$ new vertices to $c$ and $r$ new vertices to $d$, and giving each new neighbor of $c$ exactly $r+2$ leaf neighbors and each new neighbor of $d$ exactly $p+2$ leaf neighbors. Hence
\begin{equation}\label{eq:Horder}
 |V(H_{p,r})|=4+3(p+r)+2pr.
\end{equation}
For $p+r>0$, this graph admits PST between $a,b$ exactly when $p+r+4=2q^2$ for an odd integer $q\geq3$; see Proposition~\ref{prop:Hpr}.

\begin{theorem}\label{thm:sharp-unicyclic-order}
Let $G=G(T_c,T_d)\not\cong C_4$ admit perfect state transfer between $a,b$, and put $p=\deg_{T_c}(c)$ and $r=\deg_{T_d}(d)$. Then:
\begin{enumerate}[label=\textup{(\roman*)}]
\item We have
\[
 p+r\geq14,\qquad |V(G)|\geq4+3(p+r)+2pr.
\]
Equality in the order bound holds if and only if $G\cong H_{p,r}$ under an isomorphism preserving the attachment roots and the transfer pair.
\item We have $|V(G)|\geq46$, with equality if and only if $G\cong H_{0,14}$.
\item If $p,r>0$, then $|V(G)|\geq72$, with equality if and only if $G\cong H_{1,13}$.
\end{enumerate}
\end{theorem}

\begin{theorem}\label{thm:degree20-classification}
Let $G$ be a connected unicyclic graph. Then:
\begin{enumerate}[label=\textup{(\roman*)}]
\item If $\Delta(G)\leq20$, then $G$ admits perfect state transfer if and only if
\[
 G\cong C_4\quad\text{or}\quad G\cong H_{p,14-p}\quad(0\leq p\leq7).
\]
\item If $G\not\cong C_4$ admits perfect state transfer, then $\Delta(G)\geq10$, with equality if and only if $G\cong H_{7,7}$.
\item If $G$ admits perfect state transfer and is isomorphic to neither $C_4$ nor a graph $H_{p,r}$, then $\Delta(G)\geq21$. This bound is attained by the graph in Example~\ref{ex:degree21}.
\end{enumerate}
\end{theorem}

The connected unicyclic graphs admitting PST with five-element eigenvalue support are precisely
the graphs $H_{p,r}$ satisfying $p+r+4=2q^2$ for an odd integer $q\geq3$;
see Lemma~\ref{lem:R14} and Proposition~\ref{prop:Hpr}.
Thus Theorem~\ref{thm:degree20-classification}\textup{(iii)} applies exactly
to supports of size at least seven. Table~\ref{tab:extremal} summarizes
the extremal results. Figure~\ref{fig:extremal} shows $H_{0,14}$,
$H_{1,13}$, and $H_{7,7}$, which attain the two order bounds in
Theorem~\ref{thm:sharp-unicyclic-order}\textup{(ii)--(iii)} and the
maximum-degree bound in Theorem~\ref{thm:degree20-classification}\textup{(ii)}.
For both five- and seven-element
supports, the minimum order and minimum maximum degree are attained by different graphs.

\begin{table}[htbp]
\centering
\caption{Extremal graphs and low-degree classification in
Theorems~\ref{thm:seven-summary}--\ref{thm:degree20-classification}.}
\label{tab:extremal}
\small
\setlength{\tabcolsep}{5pt}
\begin{tabular}{@{}llrr@{}}
\toprule
Graph & Role & $|V(G)|$ & $\Delta(G)$\\
\midrule
$C_4$ & Four-cycle & $4$ & $2$\\
$H_{0,14}$ & Minimum order outside $C_4$ & $46$ & $16$\\
$H_{1,13}$ & Minimum order with $p,r>0$ & $72$ & $16$\\
\midrule
$H_{p,14-p}$, $2\leq p\leq6$ & Remaining graphs with $\Delta(G)\leq20$ & $46+2p(14-p)$ & $17-p$\\
$H_{7,7}$ & Minimum maximum degree outside $C_4$ & $144$ & $10$\\
\midrule
Example~\ref{ex:218} & Minimum order for $|\Phi_a|=7$ & $218$ & $47$\\
Example~\ref{ex:degree21} & Minimum $\Delta(G)$ outside $C_4$ and $H_{p,r}$ & $96760$ & $21$\\
\bottomrule
\end{tabular}
\end{table}

\begin{figure}[htbp]
\centering
\begin{tikzpicture}[pst graph]
\newcommand{\pstextremalfan}[4]{%
  \begin{scope}[xshift=#1cm]
  \ifnum#3=1
    \node[pst vertex] (r) at (0,{#2*1.55}) {};
    \draw (0,{#2*.7})--(r);
    \foreach \dx in {-.8,.8}{
      \node[pst vertex] (l) at (\dx,{#2*2.5}) {};
      \draw (r)--(l);}
    \psthdots{0}{#2*2.5}
    \node[pst note,anchor=west] at (.16,{#2*1.12}) {$1$ branch};
    \node[pst note] at (0,{#2*2.86}) {$#4$ leaves};
  \else\ifnum#3>1
    \foreach \bx in {-1.08,1.08}{
      \node[pst vertex] (r) at (\bx,{#2*1.55}) {};
      \draw (0,{#2*.7})--(r);
      \ifnum#4=2
        \foreach \dx in {-.34,.34}{
          \node[pst vertex] (l) at ({\bx+\dx},{#2*2.5}) {};
          \draw (r)--(l);}
      \else\ifnum#4=3
        \foreach \dx in {-.4,0,.4}{
          \node[pst vertex] (l) at ({\bx+\dx},{#2*2.5}) {};
          \draw (r)--(l);}
      \else
        \foreach \dx in {-.4,.4}{
          \node[pst vertex] (l) at ({\bx+\dx},{#2*2.5}) {};
          \draw (r)--(l);}
        \psthdots{\bx}{#2*2.5}
      \fi\fi
    }
    \psthdots{0}{#2*1.55}
    \node[pst note] at (0,{#2*1.91}) {$#3$ branches};
    \node[pst note] at (0,{#2*2.86}) {$#4$ leaves each};
  \fi\fi
  \end{scope}}
\foreach \x/\panel/\p/\r/\order/\degree in
  {0/a/0/14/46/16,5.15/b/1/13/72/16,10.30/c/7/7/144/10}{
  \begin{scope}[xshift=\x cm]
    \node at (0,3.46) {\textup{(\panel)}\quad $H_{\p,\r}$};
    \node[pst vertex] (c) at (0,.7) {};
    \node[pst vertex] (d) at (0,-.7) {};
    \node[pst terminal] (a) at (-.78,0) {};
    \node[pst terminal] (b) at (.78,0) {};
    \draw (a)--(c)--(b)--(d)--(a);
    \node[left=4pt] at (a) {$a$};
    \node[right=4pt] at (b) {$b$};
    \node[right=5pt] at (c) {$c$};
    \node[right=5pt] at (d) {$d$};
    \node[pst note] at (0,-3.45) {$|V|=\order,\quad\Delta=\degree$};
  \end{scope}
  \pgfmathtruncatemacro{\cl}{\r+2}
  \pgfmathtruncatemacro{\dl}{\p+2}
  \pstextremalfan{\x}{1}{\p}{\cl}
  \pstextremalfan{\x}{-1}{\r}{\dl}
}
\end{tikzpicture}
\caption{Extremal graphs: (a) minimum order outside $C_4$; (b) minimum order with both attachments nontrivial; (c) minimum maximum degree outside $C_4$. Open circles mark $a,b$; dots indicate omitted identical branches or leaves.}
\label{fig:extremal}
\end{figure}
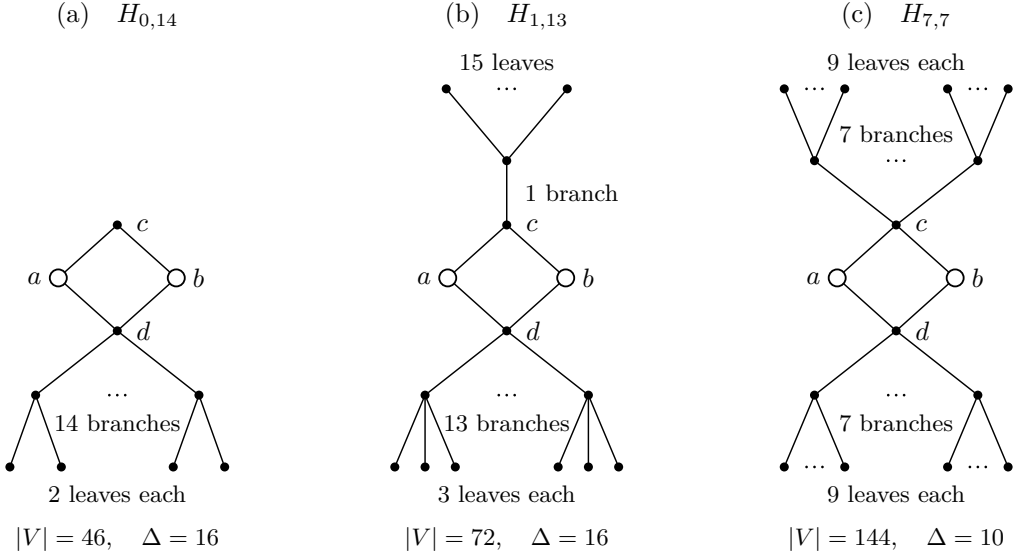

The structural proof combines a bound on the distance between the eigenvalue supports
of $(\ee_a+\ee_b)/\sqrt2$ and $(\ee_a-\ee_b)/\sqrt2$
(Theorem~\ref{thm:gap}), a degree bound for transfer vertices with pendant neighbors
(Theorem~\ref{thm:pendant}), and the parity constraint in
Lemma~\ref{lem:parity}. The gap bound extends the arrowhead comparison
in \cite[Section~5]{trees} and the bridge result in
\cite[Theorem~8]{trees} to arbitrary connected graphs.
Sections~\ref{sec:preliminaries}--\ref{sec:geometry} develop these tools
and prove Theorem~\ref{thm:geometry}. Section~\ref{sec:trees} establishes
the tree criterion, Section~\ref{sec:extremal} proves the sharp global
order bounds, and Section~\ref{sec:seven} develops the inverse construction.
Section~\ref{sec:examples} applies it to the degree and order thresholds
beyond five supported eigenvalues.

\section{Spectral and resolvent tools}\label{sec:preliminaries}
\subsection{Spectral conditions for transfer}
Fix a graph $G$, put $A=A(G)$ and $U(t)=U_G(t)$, and write $A=\sum_\lambda\lambda E_\lambda$ for the spectral decomposition of $A$. For real vectors, $\langle x,y\rangle=x^Ty$ and $\|x\|$ is the Euclidean norm. For a nonzero real vector $v$, define its eigenvalue support $\Phi_v$ and scalar resolvent $g_{A,v}$ by
\[
 \Phi_v=\{\lambda:E_\lambda v\ne0\},\qquad
 g_{A,v}(t)=v^T(tI-A)^{-1}v
 =\sum_\lambda\frac{\|E_\lambda v\|^2}{t-\lambda}.
\]
We also write $g_{G,v}=g_{A(G),v}$. When the second argument is a vertex $r$, the notation $g_{G,r}$ means $g_{A(G),\ee_r}$. For a vertex $a$, write $\Phi_a=\Phi_{\ee_a}$. When $v$ is a unit vector, its spectral measure assigns mass $\|E_\lambda v\|^2$ to $\lambda$; it has total mass one and $j$th moment $v^TA^jv$. For sufficiently large $|t|$, expansion as a geometric series gives
\[
 g_{A,v}(t)=\sum_{j=0}^{\infty}\frac{v^TA^jv}{t^{j+1}}.
\]
This identifies each Laurent coefficient with the corresponding spectral
moment. The poles of $g_{A,v}$ are simple and form the set $\Phi_v$.
Throughout, resolvent identities are understood as identities of rational
functions in $t$, with zeros and poles taken after cancellation of common
factors.

Vertices $a,b$ are \textbf{cospectral} if $(E_\lambda)_{aa}=(E_\lambda)_{bb}$ for every $\lambda$, and \textbf{strongly cospectral} if
\[
 E_\lambda\ee_a=\pm E_\lambda\ee_b
\]
for every $\lambda$. Cospectral vertices have equal degrees. For any two distinct vertices $a,b$, put
\[
 w_\pm=\frac{\ee_a\pm\ee_b}{\sqrt2},\qquad
 \Phi^\pm=\Phi_{w_\pm}.
\]
If $a,b$ are strongly cospectral, the superscripts in $\Phi^\pm$ record the signs in $E_\lambda\ee_a=\pm E_\lambda\ee_b$, not the signs of the eigenvalues. The sets $\Phi^+,\Phi^-$ are then nonempty and disjoint, and $\Phi^+\cup\Phi^-=\Phi_a=\Phi_b$ is their common eigenvalue support. PST implies strong cospectrality. For strongly cospectral vertices, the spectral criterion \cite[Lemma~2.1 and Corollary~2.2]{godsil} states that PST occurs at time $\tau$ with phase $\gamma$ if and only if
\begin{equation}\label{eq:phase}
 e^{-i\tau\lambda}=\gamma\quad(\lambda\in\Phi^+),
 \qquad e^{-i\tau\mu}=-\gamma\quad(\mu\in\Phi^-).
\end{equation}
In particular, $\lambda-\mu$ is an odd multiple of $\pi/\tau$ whenever $\lambda\in\Phi^+$ and $\mu\in\Phi^-$.

Coutinho's spectral arithmetic criterion \cite[Theorem~2.4.4]{coutinho},
also stated in \cite[Theorem~2.2]{algorithm}, gives the following
necessary conditions and minimum transfer time. If PST occurs from $a$, the eigenvalues in $\Phi_a$ are either all integers or all quadratic integers with a common rational part. In the latter case, $\Phi_a$ is contained in $\{(h+j\sqrt{\Delta_0})/2:j\in\Z\}$ for an integer $h$ and a positive square-free integer $\Delta_0$. In either case there is an integer $g>0$ such that
\begin{equation}\label{eq:arithmetic-time}
 \tau_0=\frac{\pi}{g\sqrt{\Delta_0}}.
\end{equation}
In the integral case take $\Delta_0=1$. In either case, $g$ is the greatest common divisor of the differences between eigenvalues in $\Phi_a$ after division by $\sqrt{\Delta_0}$. For $\lambda_0\in\Phi^+$ and any $\lambda\in\Phi_a$, the ratio $(\lambda_0-\lambda)/(g\sqrt{\Delta_0})$ is an integer; it is even if and only if $\lambda\in\Phi^+$, and odd if and only if $\lambda\in\Phi^-$.

Coutinho, Juliano, and Spier proved the universal bound
\begin{equation}\label{eq:universal-time}
 \tau_0\leq\frac{\pi}{\sqrt2}
\end{equation}
in \cite[Theorem~14]{decorated}; it is restated in \cite[Theorem~3]{trees}. For a pair admitting PST, every two distinct supported eigenvalues are therefore separated by at least $\sqrt2$. The minimum distance between $\Phi^+$ and $\Phi^-$
\[
 \delta_{+-}=\min\{|\lambda-\mu|:\lambda\in\Phi^+,\ \mu\in\Phi^-\}
\]
satisfies $\delta_{+-}\geq\pi/\tau_0\geq\sqrt2$.

For later arithmetic arguments, put $m_\pm(t)=\prod_{\lambda\in\Phi^\pm}(t-\lambda)$. Each $m_\pm$ is the monic polynomial of least degree satisfying $m_\pm(A)(\ee_a\pm\ee_b)=0$. The coefficients of this monic polynomial are the unique solution of a linear system obtained from the vectors $A^j(\ee_a\pm\ee_b)$, $0\leq j\leq\deg m_\pm$. These vectors have integer entries, so $m_\pm$ has rational coefficients. It divides the characteristic polynomial and has algebraic-integer roots, so $m_\pm\in\Z[t]$. Thus each of $\Phi^+,\Phi^-$ is stable under algebraic conjugation.

A vertex $a$ is \textbf{periodic} if $U(T)\ee_a=\zeta\ee_a$ for some $T>0$ and $|\zeta|=1$. PST implies periodicity at both endpoints
\cite[Lemma~1.1]{godsil}, since $U(2\tau)\ee_a=\gamma^2\ee_a$. At a periodic vertex of a bipartite graph, nonintegral supported
eigenvalues have zero common rational part \cite[Lemma~7.1.2]{coutinho}.
PST between opposite bipartition classes has integral
support \cite[Lemma~7.1.6]{coutinho}. If the endpoints lie in the same part, the spectral measures of $w_+$ and $w_-$ are symmetric about zero. Applying \eqref{eq:phase} at the minimum time $\tau_0$ forces $\gamma\in\{1,-1\}$, and the sets $\Phi^+$ and $\Phi^-$ are contained in the sets of even and odd multiples of $\pi/\tau_0$, in some order.

\begin{lemma}\label{lem:parity}
Suppose a graph admits PST from a vertex $a$, with minimum transfer time $\tau_0$. Then
\[
 \left(\frac{\pi}{\tau_0}\right)^2\in2\Z_{>0},
 \qquad \min\Phi_a+\max\Phi_a\in2\Z.
\]
\end{lemma}
\begin{proof}
Let $b$ be the other transfer vertex. Put $\kappa=\pi/\tau_0=g\sqrt{\Delta_0}$ as in \eqref{eq:arithmetic-time}. The polynomials $m_\pm(t)=\prod_{\lambda\in\Phi^\pm}(t-\lambda)$ belong to $\Z[t]$ and satisfy $m_\pm(A)(\ee_a\pm\ee_b)=0$.

If the support is integral, choose $s\in\Phi_a$ and write each supported eigenvalue as $s+j\kappa$ with $j\in\Z$. In the nonintegral case, write a supported eigenvalue as $\lambda=(h+j_0\sqrt{\Delta_0})/2$. Its algebraic conjugate $\lambda'=(h-j_0\sqrt{\Delta_0})/2$ belongs to the same one of $\Phi^+$ and $\Phi^-$ as $\lambda$, so the phase criterion gives
\[
 \frac{\lambda-\lambda'}{\kappa}=\frac{j_0}{g}\in2\Z.
\]
Consequently, setting $j=j_0/(2g)\in\Z$ and $s=h/2$, we obtain $\lambda=s+j\kappa$. Since $j\kappa$ is an algebraic integer, the rational number $s=\lambda-j\kappa$ is an integer. In both cases, membership in $\Phi^+$ or $\Phi^-$ is determined by the parity of $j$.

Suppose $\kappa^2$ is odd, and let a bar denote reduction modulo $2$. In the integral case each factor $t-s-j\kappa$ reduces modulo $2$ to $t-\bar s-\bar j$. In the nonintegral case each conjugate pair contributes
\[
 (t-s)^2-j^2\kappa^2,
\]
which reduces to $(t-\bar s-\bar j)^2$; a possible rational root $s$ has $j=0$. Thus $\overline m_+$ and $\overline m_-$ are powers of different linear polynomials and are coprime over $\F_2$. But they both annihilate the same nonzero vector $\overline{\ee_a+\ee_b}=\overline{\ee_a-\ee_b}$. Since the two reduced polynomials are coprime, there are
$u,v\in\F_2[t]$ with $u\overline m_++v\overline m_-=1$.
Evaluating this identity at $\overline A$ and applying it to
$\overline{\ee_a+\ee_b}$ gives $\overline{\ee_a+\ee_b}=0$,
contrary to $a\ne b$. Since $\kappa^2=g^2\Delta_0$ is a positive integer, it is therefore even.

For integral support, $\kappa=g$ is now even, so all supported eigenvalues have the same parity. For nonintegral support, algebraic conjugation makes the support symmetric about the integer $s$. In either case its least and greatest elements have even integer sum.
\end{proof}

\subsection{Eigenvalue displacement and resolvents}
For a symmetric matrix $B$, let $B^\dagger$ denote its Moore--Penrose inverse, let $\lambda_{\min}(B)$ denote its smallest eigenvalue, and let $\|B\|$ denote its operator norm. For a graph $G$, write $\lambda_{\min}(G)=\lambda_{\min}(A(G))$ and let $\rho(G)$ be its adjacency spectral radius. We use the following forms of the Lidskii--Mirsky--Wielandt inequality
and Weyl's norm bound, including the equality case of the latter.

\begin{lemma}[{\cite[Theorem~1.1]{limathias}; \cite[Theorem~6]{trees}}]\label{lem:displacement}
Let $X,Y$ be real symmetric $n\times n$ matrices, with increasingly
ordered eigenvalues $x_j,y_j$. Let $d_1\geq\cdots\geq d_n$ be the
eigenvalues of $D=X-Y$. Then:
\begin{enumerate}[label=\textup{(\roman*)}]
\item For every $J\subseteq\{1,\ldots,n\}$,
\begin{equation}\label{eq:lmw-subset}
 \sum_{j\in J}(x_j-y_j)\leq\sum_{\ell=1}^{|J|}d_\ell.
\end{equation}
\item For each $j$, $|x_j-y_j|\leq\|X-Y\|$. If equality holds, then
$X,Y$ have a common eigenvector with eigenvalues $x_j,y_j$, respectively.
\end{enumerate}
\end{lemma}
For the application below, write $u_+=\max\{u,0\}$ and take
$J=\{j:x_j>y_j\}$ and $r=|J|$ in \eqref{eq:lmw-subset}. This gives
\[
 \sum_j(x_j-y_j)_+
 =\sum_{j\in J}(x_j-y_j)
 \leq\sum_{\ell=1}^r d_\ell
 \leq\sum_{\ell=1}^n(d_\ell)_+.
\]
Applying the same inequality to $Y-X$ gives the other half of the
following direct consequence.
\begin{corollary}[{\cite[Theorem~1.1]{limathias}}]\label{cor:trace-displacement}
Let $X,Y$ be real symmetric matrices of the same order, with increasingly
ordered eigenvalues $x_j,y_j$. If $X-Y=D_+-D_-$ is the decomposition into
positive and negative parts, then
\[
 \sum_j(x_j-y_j)_+\leq\tr D_+,
 \qquad \sum_j(y_j-x_j)_+\leq\tr D_-.
\]
\end{corollary}

We also use the following partial-fraction form of strict interlacing.
\begin{lemma}[{\cite[Lemma~5]{trees}}]\label{lem:resolvent-interlacing}
Let $m\geq0$, let $\theta_1<\cdots<\theta_m$ be real, and let
$c_1,\ldots,c_m>0$ and $b\in\R$. The rational function
\[
 \alpha(t)=t-b-\sum_{j=1}^m\frac{c_j}{t-\theta_j}
\]
has exactly $m+1$ zeros, all real and simple, with one in each component
of $\R\setminus\{\theta_1,\ldots,\theta_m\}$. Thus its zeros strictly
interlace its poles.
\end{lemma}

\begin{lemma}\label{lem:first-positive-zero}
Let $g_1,\ldots,g_m$ be scalar resolvents of unit vectors whose finite spectral measures are symmetric about zero, and let $\beta_1,\ldots,\beta_m>0$. Put
\[
 \alpha(t)=t-\sum_{i=1}^m\beta_i g_i(t).
\]
Suppose $0$ is not a zero of $\alpha$, and let $t_0>0$ be its smallest positive zero. Then no $g_i$ has a pole in $(0,t_0]$. For each $i$, we have $g_i(t_0)\leq1/t_0$, with equality if and only if the spectral measure defining $g_i$ is concentrated at zero. If $g_i(t_0)\geq0$, the same measure has positive mass at zero.
\end{lemma}
\begin{proof}
Each $g_i$ is an odd rational function with positive residues, so $\alpha$ is odd and
its poles are exactly the union of the poles of the $g_i$.
Since $0$ is not a zero, oddness forces it to be a pole of $\alpha$.
Collecting terms with the same pole expresses $\alpha$ in the form of
Lemma~\ref{lem:resolvent-interlacing}, with $b=0$ and positive coefficients.
If $\alpha$ has a positive pole, there is therefore a zero between $0$ and its smallest positive pole.
Since $t_0$ is the smallest positive zero and cannot itself be a pole,
no $g_i$ has a pole in $(0,t_0]$.

Let $\eta_{i,0}$ be the spectral mass at zero and let $\eta_{i,\theta}$ be the total mass at $\{-\theta,\theta\}$. Pairing positive and negative poles gives
\[
 g_i(t_0)=\frac{\eta_{i,0}}{t_0}
 +\sum_{\theta>t_0}\eta_{i,\theta}
      \frac{t_0}{t_0^2-\theta^2},
 \qquad \eta_{i,0}+\sum_{\theta>t_0}\eta_{i,\theta}=1,
\]
where the sums run over the positive poles of $g_i$, all of which exceed $t_0$. Every term in the sum is negative when its weight is positive, so
$g_i(t_0)\leq\eta_{i,0}/t_0\leq1/t_0$.
Equality with $1/t_0$ holds exactly when $\eta_{i,0}=1$.
If $\eta_{i,0}=0$, the remaining weights sum to one and give
$g_i(t_0)<0$; hence $g_i(t_0)\geq0$ requires $\eta_{i,0}>0$.
\end{proof}

We will also use the standard relation between the size of an eigenvalue support and graph distance.
\begin{lemma}[{\cite[Lemma~5 and its proof]{supportdistance}}]\label{lem:support-distance}
For a vertex $v$ of a connected graph $G$,
\[
 \max_{u\in V(G)}\dist(v,u)\leq |\Phi_v|-1.
\]
In particular, a rooted tree whose root has an eigenvalue support of size $m$ has depth
at most $m-1$.
\end{lemma}

\section{Spectral gaps and transfer times}\label{sec:sparse}
We bound the distance between $\Phi^+$ and $\Phi^-$ using the components of $G-\{a,b\}$ that contain neighbors of both endpoints. The argument applies to every connected graph.
\subsection{A bound from endpoint neighborhoods}
Let $G$ be a connected graph, and let $a,b$ be nonadjacent vertices of $G$. Set $H=G-\{a,b\}$. Let $W$ be the union of the vertex sets of those components of $H$ that contain a neighbor of $a$ and a neighbor of $b$. Define
\[
 k_a=|N(a)\cap W|,\qquad k_b=|N(b)\cap W|,\qquad k=\min(k_a,k_b).
\]
Connectivity gives $k\geq1$. Neighbors in components that meet only one endpoint neighborhood are not counted.

\begin{theorem}\label{thm:gap}
Let $G$ be a connected graph, and let $a,b$ be nonadjacent strongly cospectral vertices. Put $k=\min\{|N(a)\cap W|,|N(b)\cap W|\}$, where $W$ is the union of the vertex sets of the components of $G-\{a,b\}$ meeting both $N(a)$ and $N(b)$. Then:
\begin{enumerate}[label=\textup{(\roman*)}]
\item The minimum distance between $\Phi^+$ and $\Phi^-$ satisfies
\[
 \delta_{+-}:=\min_{\lambda\in\Phi^+,\ \mu\in\Phi^-}|\lambda-\mu|
 \leq \sqrt{2k}.
\]
Equality holds if and only if $G\cong K_{2,k}$, with $a,b$ corresponding to the two vertices in a part of size two.
\item If $G$ admits PST between $a,b$, its minimum transfer time satisfies
\[
 \tau_0\geq\frac{\pi}{\sqrt{2k}}.
\]
Equality holds if and only if $G\cong K_{2,k}$, with $a,b$ corresponding to the two vertices in a part of size two.
\end{enumerate}
\end{theorem}

\begin{proof}
\proofheading{Step 1: compare the scalar resolvents for part \textup{(i)}.}
Let $C=A(H)$, and let $x,y$ be the neighbor indicator vectors of $a,b$ in $H$. Indexing rows and columns by $a,b$, and $V(H)$,
\[
 A=\begin{pmatrix}0&0&x^T\\0&0&y^T\\x&y&C\end{pmatrix}.
\]
Put $f_x=x^T(tI-C)^{-1}x$, $f_y=y^T(tI-C)^{-1}y$, and
$h=x^T(tI-C)^{-1}y$. The Schur complement of $tI-C$ in $tI-A$ is
\[
 S(t)=\begin{pmatrix}t-f_x&-h\\-h&t-f_y\end{pmatrix}.
\]
Its inverse is the principal submatrix of $(tI-A)^{-1}$ indexed by $a,b$, so
\[
 g_{G,a}(t)=\frac{t-f_y}{\det S(t)},\qquad
 g_{G,b}(t)=\frac{t-f_x}{\det S(t)}.
\]
Cospectrality gives $g_{G,a}=g_{G,b}$, hence $f_x=f_y$.
For each eigenvalue $\theta$ of $C$, let $F_\theta$ be its spectral projector.
Comparing residues in $f_x=f_y$ gives $\|F_\theta x\|^2=\|F_\theta y\|^2$.
Put
\[
 d_\theta=\|F_\theta x\|^2=\|F_\theta y\|^2,
 \qquad c_\theta=\langle F_\theta x,F_\theta y\rangle.
\]
The inequality $|c_\theta|\leq d_\theta$ follows from the Cauchy--Schwarz inequality.
In the symmetric and antisymmetric endpoint coordinates, $S(t)$ has
diagonal entries $t-f_x-h$ and $t-f_x+h$. Substituting the spectral
expansions of $f_x,h$ therefore gives $g_{A,w_\pm}=1/\alpha_\pm$, where
\begin{equation}\label{eq:alpha}
 \alpha_\pm(t)=t-\sum_\theta\frac{d_\theta\pm c_\theta}{t-\theta}.
\end{equation}
The zeros of $\alpha_\pm$ are exactly $\Phi^\pm$.

Suppose, after switching the endpoints if necessary, that $k_a=k$. Let $x_0,y_0$ be the restrictions of $x,y$ to $W$, extended by zero. Since the coordinate subspaces corresponding to distinct components of $H$ are orthogonal and invariant under $C$,
\[
 c_\theta=\langle F_\theta x_0,F_\theta y_0\rangle.
\]
Set $u_\theta=\|F_\theta x_0\|^2$ and $v_\theta=\|F_\theta y_0\|^2$. We have
\[
 0\leq u_\theta,v_\theta\leq d_\theta,\quad
 c_\theta^2\leq u_\theta v_\theta,\quad
 \sum_\theta u_\theta=k.
\]
For each $\theta$,
\begin{align}
 \bigl(\sqrt{d_\theta+c_\theta}-\sqrt{d_\theta-c_\theta}\bigr)^2
 &=2d_\theta-2\sqrt{d_\theta^2-c_\theta^2}\notag\\
 &\leq 2u_\theta. \label{eq:residue}
\end{align}
Indeed, $c_\theta^2\leq u_\theta v_\theta\leq u_\theta d_\theta$ and $0\leq u_\theta\leq d_\theta$ imply
\[
 d_\theta^2-c_\theta^2
 \geq d_\theta^2-u_\theta d_\theta
 \geq(d_\theta-u_\theta)^2.
\]
Since $d_\theta-u_\theta\geq0$, taking square roots yields \eqref{eq:residue}.
Equality in the second inequality of the preceding display requires
$u_\theta(d_\theta-u_\theta)=0$. If $u_\theta>0$, it follows that
$d_\theta=u_\theta$; equality in the first inequality then gives
$c_\theta^2=u_\theta d_\theta$, and $c_\theta^2\leq u_\theta v_\theta\leq u_\theta d_\theta$
forces $v_\theta=d_\theta$. Thus, when $u_\theta>0$, equality in \eqref{eq:residue} forces
\begin{equation}\label{eq:eqres}
 d_\theta=u_\theta=v_\theta=|c_\theta|.
\end{equation}

\proofheading{Step 2: compare the ordered eigenvalues.}
List the distinct values $\theta$ with $d_\theta>0$ as $\theta_1,\ldots,\theta_q$. Define two $(q+1)\times(q+1)$ real symmetric arrowhead matrices
\[
 M_\pm=
 \begin{pmatrix}
 0&r_\pm^T\\
 r_\pm&\diag(\theta_1,\ldots,\theta_q)
 \end{pmatrix},
 \qquad (r_\pm)_j=\sqrt{d_{\theta_j}\pm c_{\theta_j}}.
\]
Coordinates are indexed by $0,1,\ldots,q$, with $0$ central and $j$ corresponding to $\theta_j$.
Let $q_\pm$ be the number of poles of $\alpha_\pm$, equivalently the
number of values $\theta$ for which $d_\theta\pm c_\theta>0$.
After omitting the zero coefficients, apply
Lemma~\ref{lem:resolvent-interlacing} with $b=0$ and
$c_j=d_{\theta_j}\pm c_{\theta_j}>0$ at the poles with positive coefficients.
Thus $\alpha_\pm$ has $q_\pm+1$ simple zeros strictly interlacing its poles. The identity
\[
 \det(tI-M_\pm)=\alpha_\pm(t)\prod_{i=1}^q(t-\theta_i)
\]
shows that all these zeros occur as eigenvalues of $M_\pm$. For each sign separately, permuting coordinates decomposes $M_\pm$ into an arrowhead block containing the central coordinate and the $q_\pm$ coordinates with positive coupling, together with $q-q_\pm$ uncoupled scalar blocks. This arrowhead block has simple spectrum consisting of the zeros, and the uncoupled diagonal entries are distinct. Thus a zero that coincides with an uncoupled entry has multiplicity exactly two. In the increasingly ordered eigenvalue list of each matrix, select one occurrence of each zero and write $J_\pm\subseteq\{1,\ldots,q+1\}$ for the selected positions. Then $|J_\pm|=q_\pm+1$. The permutations are used only for this spectral description; the matrix comparison below uses the same coordinate ordering $0,1,\ldots,q$ for both matrices.

If $m$ of the values $\theta_1,\ldots,\theta_q$ satisfy both $d_\theta+c_\theta>0$ and $d_\theta-c_\theta>0$, then $q_++q_-=q+m$. Consequently
\begin{equation}\label{eq:common-positions}
 |J_+\cap J_-|\geq(q_++1)+(q_-+1)-(q+1)=m+1.
\end{equation}
Choose a common position $j$, occupied by $\lambda\in\Phi^+$ and $\mu\in\Phi^-$. Lemma~\ref{lem:displacement}\textup{(ii)} gives
\[
 |\lambda-\mu|\leq\|M_+-M_-\|
 =\|r_+-r_-\|
 \leq\sqrt{2k}.
\]
The norm identity follows because a block matrix
$\left(\begin{smallmatrix}0&\xi^T\\ \xi&0\end{smallmatrix}\right)$
has eigenvalues $\pm\|\xi\|$ and zeros. This proves the gap bound.

\proofheading{Step 3: determine the endpoint degrees in the equality case.}
We now determine the equality case, allowing $\alpha_+$ and $\alpha_-$ to share poles. If $\delta_{+-}=\sqrt{2k}$, all inequalities in the preceding display are equalities. Put $\xi=r_+-r_-$ and $D=M_+-M_-$. The equality assertion of Lemma~\ref{lem:displacement}\textup{(ii)} gives a common eigenvector $z$ of $M_+$ and $M_-$ which is an eigenvector of $D$ with eigenvalue $\pm\|\xi\|$. Since
\[
 D=\begin{pmatrix}0&\xi^T\\\xi&0\end{pmatrix}
 \quad\text{and}\quad \|\xi\|=\sqrt{2k}>0,
\]
the noncentral rows of $Dz=\nu z$, with
$\nu=\pm\|\xi\|\ne0$, give $\nu z_\ell=\xi_\ell z_0$.
If $z_0=0$, all coordinates of $z$ would vanish. Thus $z_0\ne0$,
and $z_\ell=0$ whenever $\xi_\ell=0$.
If $\xi_\ell=0$ for some $1\leq\ell\leq q$, then
$z_\ell=0$ and $c_{\theta_\ell}=0$. The $\ell$th noncentral row of $M_+z=\lambda z$ is
\[
 \sqrt{d_{\theta_\ell}+c_{\theta_\ell}}\,z_0+\theta_\ell z_\ell=\lambda z_\ell,
\]
which reduces to $\sqrt{d_{\theta_\ell}}\,z_0=0$, a contradiction. Thus $\xi_\ell\ne0$ for every $\ell\in\{1,\ldots,q\}$.
By \eqref{eq:residue}, this also gives $u_{\theta_\ell}>0$ for every such $\ell$.

Since $\|r_+-r_-\|^2=2k=2\sum_\theta u_\theta$, equality holds in \eqref{eq:residue} for each $\theta$. Thus \eqref{eq:eqres} holds whenever $d_\theta>0$. Therefore
\[
 \deg(a)=\sum_\theta d_\theta
 =\sum_\theta u_\theta=k.
\]
Cospectrality gives $\deg(b)=k$.

\proofheading{Step 4: identify the equality graph.}
Continue to assume that the gap bound is an equality. Use the spectral measures of $w_+$ and $w_-$ to define independent real random variables $X,Y$ with supports $\Phi^+,\Phi^-$. Since the endpoints are nonadjacent,
\[
 \mathbb E X=\mathbb E Y=0,\qquad
 \mathbb E\bigl[(X-Y)^2\bigr]=
 \langle w_+,A^2w_+\rangle+\langle w_-,A^2w_-\rangle
 =2k.
\]
Every possible value of $(X-Y)^2$ is at least $2k$. Every pair in $\Phi^+\times\Phi^-$ has positive probability, so $|\lambda-\mu|=\sqrt{2k}$ for every $\lambda\in\Phi^+$ and $\mu\in\Phi^-$. The set $\Phi^+$ has at least two elements, because the spectral measure of $w_+$ has zero mean and assigns positive mass to the Perron eigenvalue. There is at most one real number at a fixed positive distance from two distinct fixed real numbers. Since $\Phi^-$ is nonempty, it is a singleton. The spectral measure of $w_-$ has zero mean, so $\Phi^-=\{0\}$. It follows that
\[
 \Phi^-=\{0\},\qquad \Phi^+=\{-\sqrt{2k},\sqrt{2k}\}.
\]
Since $\Phi^-=\{0\}$, we have $A(\ee_a-\ee_b)=0$, so $x=y$.
Since every eigenvalue in $\Phi^+$ has square $2k$, spectral
decomposition gives $A^2(\ee_a+\ee_b)=2k(\ee_a+\ee_b)$.
Restricting this equation to $H$ and using $x=y$ gives $2Cx=0$.
As $C$ and $x$ are nonnegative, every common neighbor of $a,b$ is isolated in $H$. Connectivity now forces $G\cong K_{2,k}$. Conversely, $K_{2,k}$ has $\Phi^-=\{0\}$ and $\Phi^+=\{-\sqrt{2k},\sqrt{2k}\}$, so it attains the gap bound. This completes part \textup{(i)}.

\proofheading{Step 5: prove part \textup{(ii)}.}
For any pair admitting PST, \eqref{eq:phase} gives $\delta_{+-}\geq\pi/\tau_0$, proving the asserted time bound. If the time bound is an equality, then $\pi/\tau_0=\sqrt{2k}$, so part \textup{(i)} forces equality in the gap bound and gives $G\cong K_{2,k}$. Conversely, the displayed supports of $K_{2,k}$ give minimum transfer time $\pi/\sqrt{2k}$.
\end{proof}

For the next lemma, put $C=A(G-\{a,b\})$, let $x,y$ be the
neighbor indicator vectors of $a,b$ in $G-\{a,b\}$, and let $F_\theta$
be the spectral projectors of $C$. Write
$\alpha_\pm=1/g_{A,w_\pm}$, as in \eqref{eq:alpha}.
An eigenvalue $\theta$ of $C$ is a pole of both $\alpha_+$ and $\alpha_-$
if and only if both $F_\theta(x+y)$ and $F_\theta(x-y)$ are nonzero.
\begin{lemma}\label{lem:shared}
Let $G$ be a connected graph, and let $a,b$ be nonadjacent strongly cospectral vertices. With $k$ as in Theorem~\ref{thm:gap}, suppose $k\leq3$ and $\delta_{+-}\geq\sqrt2$. Then:
\begin{enumerate}[label=\textup{(\roman*)}]
\item The functions $\alpha_+$ and $\alpha_-$ have at most one common pole.
\item If a component $Q$ of $G-\{a,b\}$ contains a neighbor of $u\in\{a,b\}$ but no neighbor of the other endpoint, then $Q$ is regular and every vertex of $Q$ is adjacent to $u$. Moreover, $Q$ is a singleton if only one edge joins it to $u$ or if $G$ is bipartite.
\item If $G$ is bipartite and $a,b$ belong to the same bipartition class, the functions $\alpha_+$ and $\alpha_-$ have no nonzero common pole; equivalently,
\[
 F_\theta x=\pm F_\theta y\quad(\theta\ne0).
\]
\end{enumerate}
\end{lemma}
\begin{proof}
\textup{(i)} Use the arrowhead matrices $M_\pm$ constructed in the proof of Theorem~\ref{thm:gap}, with $d_\theta=\|F_\theta x\|^2$ and $c_\theta=\langle F_\theta x,F_\theta y\rangle$. Their difference $D$ has potentially nonzero eigenvalues $s,-s$, where $s=\|r_+-r_-\|\leq\sqrt{2k}\leq\sqrt6$. Corollary~\ref{cor:trace-displacement} shows that the total positive displacement of their ordered eigenvalues is at most $s$, and the total negative displacement has magnitude at most $s$.

If there are $m$ common poles of $\alpha_+$ and $\alpha_-$, \eqref{eq:common-positions} gives at least $m+1$ common selected positions. At each such position the eigenvalue difference has magnitude at least $\sqrt2$. The total displacement bounds permit at most one positive and one negative such difference, since $2\sqrt2>\sqrt6$. Hence $m+1\leq2$.

\textup{(ii)} Suppose $Q$ contains a neighbor of $a$ and no neighbor of $b$; the other case follows by interchanging $a,b$. Let $z$ be the indicator vector of $N(a)\cap V(Q)$, and let $F_\theta^Q$ be the spectral projector of
$A(Q)$, taken to be zero if $\theta$ is not an eigenvalue. Since the
restriction of $y$ to $Q$ vanishes,
\[
 d_\theta\pm c_\theta
 =\tfrac12\|F_\theta(x\pm y)\|^2
 \geq\tfrac12\|F_\theta^Qz\|^2.
\]
Every eigenvalue supported at $z$ is therefore a common pole of $\alpha_+$ and $\alpha_-$. There is
at most one such pole, so $A(Q)z=\theta z$ for a single $\theta$. At any vertex $v$ with $z_v=0$, the equation
$\sum_{u\sim v}z_u=\theta z_v=0$ and the nonnegativity of $z$
force $z_u=0$ at every neighbor of $v$.
Since $Q$ is connected and $z$ is a nonzero indicator vector,
$z=\mathbf1$, where $\mathbf1$ is the all-ones vector on $Q$.
Thus $Q$ is regular and every vertex is adjacent to $a$.
If only one edge joins $Q$ to $a$, $z$ has only one
nonzero coordinate, so $Q$ is a singleton.
If $G$ is bipartite, the neighbors of an endpoint form an independent
set, again forcing the connected graph $Q$ to be a singleton.

\textup{(iii)} Suppose $G$ is bipartite and the endpoints lie in the same
bipartition class. Then $x,y$ are both supported
in the opposite class. Conjugating $C$ by the diagonal matrix with signs
$1,-1$ on the two classes sends $C$ to $-C$ and multiplies $x,y$
by the same sign. It therefore preserves their projected norms and
inner products while exchanging $\theta$ with $-\theta$.
Thus $d_{-\theta}=d_\theta$ and $c_{-\theta}=c_\theta$. A nonzero common pole of $\alpha_+$ and $\alpha_-$ would therefore give two common poles, which is impossible. For $\theta\ne0$ with $d_\theta>0$, the absence of a common pole gives $d_\theta=|c_\theta|$, and hence
$F_\theta x=\pm F_\theta y$. Conversely, this identity makes one of
$F_\theta(x+y),F_\theta(x-y)$ zero, so $\theta$ is not a pole of both functions.
If $d_\theta=0$, both projections vanish. This proves the equivalence
for every $\theta\ne0$.
\end{proof}

\subsection{Bridges and cactus graphs}
\begin{theorem}\label{thm:bridge}
Let $G$ be a connected graph, and let $a,b$ be distinct vertices. Suppose an edge $aa'$ is a bridge whose deletion separates $a$ and $b$. Then $G$ admits PST between $a,b$ if and only if $G\cong P_2$ or $G\cong P_3$, with $a,b$ corresponding to the two ends.
\end{theorem}
\begin{proof}
Suppose first that $G$ admits PST between $a,b$. If $a,b$ are adjacent, then $aa'=ab$ is a bridge.
By \cite[Theorem~11]{bridges}, a connected graph admitting PST between the two ends of a bridge must be $P_2$.
Hence $G\cong P_2$ in this case. Otherwise, $aa'$ lies on every $a$--$b$ path. Each neighbor counted in $k_a$ is the first vertex after $a$ on such a path, so $a'$ is the only counted neighbor and $k_a=1$.
In the nonadjacent case, Theorem~\ref{thm:gap}\textup{(ii)} and \eqref{eq:universal-time} give
\[
 \frac{\pi}{\sqrt2}\leq\tau_0\leq\frac{\pi}{\sqrt2}.
\]
Equality in Theorem~\ref{thm:gap}\textup{(ii)} therefore gives $G\cong K_{2,1}\cong P_3$. Conversely, the two ends of $P_2$ and $P_3$ admit PST at times $\pi/2$ and $\pi/\sqrt2$, respectively.
\end{proof}

\begin{corollary}\label{cor:cycles}
Let $G$ be a connected graph not isomorphic to $P_2$ or $P_3$. If $G$ admits PST between distinct vertices $a,b$, then both $a$ and $b$ lie on cycles. In particular, if $G$ is unicyclic, both vertices lie on its unique cycle.
\end{corollary}
\begin{proof}
If a vertex $a$ lies on no cycle, every edge incident to $a$ is a bridge. The first edge of a path from $a$ to any other vertex $b$ is therefore an incident bridge separating the pair. Apply Theorem~\ref{thm:bridge}.
\end{proof}

A \textbf{cactus} is a connected graph whose blocks with at least two vertices are single edges or cycles.

\begin{theorem}\label{thm:cactus}
Let $G$ be a connected cactus admitting PST between
nonadjacent vertices $a,b$. Its minimum transfer time is $\pi/2$ if
$G\cong C_4$, in which case $a,b$ are opposite vertices, and
$\pi/\sqrt2$ otherwise.
\end{theorem}
\begin{proof}
Two $a$--$b$ paths cannot start in different blocks incident with $a$: their union would contain a cycle using edges from both blocks. Thus every $a$--$b$ path starts in the same block $B$, and every neighbor counted in $k_a$ lies in $B$. Since $B$ is an edge or a cycle, $k_a\leq2$, and similarly $k_b\leq2$.

Theorem~\ref{thm:gap}\textup{(ii)} yields $\tau_0\geq\pi/2$. Combining this with Lemma~\ref{lem:parity} gives
\[
 (\pi/\tau_0)^2\in\{2,4\}.
\]
If $\tau_0=\pi/2$, then $\delta_{+-}\geq2$ by \eqref{eq:phase}.
Thus $k=2$, and equality in Theorem~\ref{thm:gap}\textup{(i)} gives
$G\cong K_{2,2}\cong C_4$ with $a,b$ opposite.
Conversely, opposite vertices of $C_4$ have
$\Phi^- =\{0\}$ and $\Phi^+=\{-2,2\}$, so their minimum transfer
time is $\pi/2$. If $G\not\cong C_4$, the only remaining possibility
is $(\pi/\tau_0)^2=2$, giving $\tau_0=\pi/\sqrt2$.
\end{proof}

\begin{corollary}\label{cor:bipartite}
Let $G$ be a connected bipartite cactus admitting PST between nonadjacent vertices $a,b$. Then $a,b$ belong to the same bipartition class. If $G\not\cong C_4$, then $\Phi_a=\Phi_b\subseteq\sqrt2\,\Z$ and the minimum transfer time is $\pi/\sqrt2$.
\end{corollary}
\begin{proof}
PST between opposite bipartition classes has
integral support \cite[Lemma~7.1.6]{coutinho}, so
\eqref{eq:arithmetic-time} makes $\tau_0/\pi$ rational.
Theorem~\ref{thm:cactus} would then force $\tau_0=\pi/2$ and
$G\cong C_4$ with opposite vertices. Those vertices lie in the same
bipartition class, a contradiction.
Thus the endpoints lie in the same class. The spectral measures of $w_+$ and $w_-$ are
symmetric about zero, so the phase is $1$ or $-1$. Equation~\eqref{eq:phase}
therefore gives $\tau_0\lambda\in\pi\Z$ for every $\lambda\in\Phi_a$.
When $G\not\cong C_4$, substitute $\tau_0=\pi/\sqrt2$ to obtain
$\Phi_a\subseteq\sqrt2\,\Z$.
\end{proof}

\section{The location and degree of transfer vertices}\label{sec:geometry}
By Corollary~\ref{cor:cycles}, both transfer vertices of a unicyclic graph lie on its unique cycle. We prove that nonadjacent transfer vertices have degree two and a common neighborhood in the bipartite case, then exclude odd cycles and adjacent transfer vertices.

\subsection{Pendant neighbors and endpoint degrees}

The next inequality applies to two transfer vertices in the same part of any bipartite graph.
\begin{theorem}\label{thm:pendant}
Let $G$ be a connected bipartite graph admitting PST between distinct vertices $a,b$ in the same bipartition class, with minimum transfer time $\tau_0$. Let $p\geq1$ be an integer, and suppose each of $a,b$ has at least $p$ pendant neighbors. Put $D=\deg(a)=\deg(b)$ and $\kappa=\pi/\tau_0$. Then
\[
 D\geq p+\frac32\kappa^2.
\]
Equality holds if and only if $p$ is even and $G$ is obtained from $K_{2,3p/2}$ by attaching $p$ leaves to each vertex in its specified part $\{a,b\}$ of size two.
In the equality case, the minimum transfer time between $a,b$ is $\pi/\sqrt p$.
\end{theorem}
\begin{proof}
\proofheading{Step 1: obtain the moment bounds.}
The spectral measures of $w_+$ and $w_-$ are symmetric about zero. PST therefore has phase $\gamma\in\{1,-1\}$, and the eigenvalues in $\Phi^+$ and $\Phi^-$ are, in some order, even and odd multiples of $\kappa$. Choose $p$ pendant neighbors at each endpoint and let $L_\pm$ be the normalized sums of their standard basis vectors, with coefficient $1/\sqrt{2p}$ at each chosen leaf adjacent to $a$ and coefficient $\pm1/\sqrt{2p}$ at each chosen leaf adjacent to $b$. Then
\[
 AL_\pm=\sqrt p\,w_\pm,\qquad
 A^\dagger w_\pm=p^{-1/2}(A^\dagger A)L_\pm.
\]
Thus $w_\pm\in\operatorname{range}A$, so neither measure has mass at zero. Since $A^\dagger A$ is an orthogonal projection and $\|L_\pm\|=1$, we also have $\|A^\dagger w_\pm\|^2\leq1/p$. Let $X_\pm=\lambda^2/\kappa^2$ under these probability measures, and put $s=p/\kappa^2$. One variable is supported on the squares of positive even integers and the other on the squares of odd integers. Their moments satisfy
\[
 \mathbb E X_++\mathbb E X_-=\frac{2D}{\kappa^2},
 \qquad \mathbb E(X_\pm^{-1})\leq\frac1s.
\]
\proofheading{Step 2: prove the degree bound.}
It is enough to prove that the two means sum to at least $2s+3$.
\proofheading{Case 1: $0<s<1$.}
Since the positive even squares and the odd squares are at least $4$ and $1$, respectively, the two means sum to at least $4+1>2s+3$, giving a strict inequality.

\proofheading{Case 2: $s\geq1$.}
Choose an integer $n\geq1$ such that $n^2\leq s\leq(n+1)^2$. On the positive squares of the same parity as $n$, and on those of the opposite parity, respectively, the inequalities
\[
 \frac{(x-n^2)(x-(n+2)^2)}x\geq0,
 \qquad
 \frac{(x-(n-1)^2)(x-(n+1)^2)}x\geq0
\]
hold. For $n=1$, the second inequality reduces to $x-4\geq0$ on the positive even squares: its $1/x$ coefficient $(n-1)^2(n+1)^2$ vanishes, so the same calculation applies. Expanding the first scalar inequality gives
$x\geq n^2+(n+2)^2-n^2(n+2)^2/x$, and the second has the analogous
form with $n-1,n+1$. Both products $n^2(n+2)^2$ and $(n-1)^2(n+1)^2$ are nonnegative.
Since their inverse-moment terms are subtracted, the upper bounds
$\mathbb E(X_\pm^{-1})\leq1/s$ give the lower bound
\[
 \mathbb E X_++\mathbb E X_-
 \geq4n^2+4n+6-
 \frac{n^2(n+2)^2+(n-1)^2(n+1)^2}{s}.
\]
The numerator of the last fraction is $2n^2(n+1)^2+1$. Subtracting
$2s+3$ and factoring gives
\[
 \mathbb E X_++\mathbb E X_--2s-3
 \geq\frac{2(s-n^2)((n+1)^2-s)+(s-1)}s\geq0.
\]
Both terms in the numerator are nonnegative because
$n^2\leq s\leq(n+1)^2$ and $s\geq1$. This proves the degree bound.

\proofheading{Step 3: determine equality.}
Equality forces the numerator in the preceding display to vanish.
Its term $s-1$ gives $s=1$, and hence the two means sum to $5$.
The variable supported on the positive even squares has mean at least $4$, and the variable supported on the odd squares has mean at least $1$. Thus their distributions are concentrated at $4$ and $1$, respectively. Hence $p=\kappa^2$ and $D=5p/2$. The sets $\Phi^+$ and $\Phi^-$ are therefore $\{\pm2\sqrt p\}$ and
$\{\pm\sqrt p\}$ in some order. The Perron eigenvalue is the largest
supported eigenvalue and belongs to $\Phi^+$, so
\[
 \Phi^+=\{-2\sqrt p,2\sqrt p\},\qquad
 \Phi^-=\{-\sqrt p,\sqrt p\}.
\]
Since $\langle w_\pm,A^2w_\pm\rangle=D\pm|N(a)\cap N(b)|$, the common-neighbor count is $(4p-p)/2=3p/2$. All remaining $p$ neighbors at each endpoint are the selected leaves. If a common neighbor had a neighbor $u$ other than $a,b$, then the
$u$ coordinate of $A^2w_+$ would be positive. This contradicts
$A^2w_+=4p\,w_+$, whose right side is supported only at $a,b$.
Thus the common neighbors have no additional neighbors. Connectivity gives the stated graph, and integrality of $3p/2$ forces $p$ even. Conversely, let $p$ be even and let $G$ be the described graph.
The $p$ leaves at each endpoint contribute $p\,w_\pm$ to
$A^2w_\pm$, while the $3p/2$ common neighbors contribute $3p\,w_+$
in the symmetric case and zero in the antisymmetric case. Hence
\[
 A^2w_+=4p\,w_+,\qquad A^2w_-=p\,w_-.
\]
Symmetry of the spectral measures about zero gives exactly the displayed sets $\Phi^+$ and $\Phi^-$.
The phase condition \eqref{eq:phase} is then satisfied first at
$\pi/\sqrt p$. Thus $\kappa^2=p$ and $D=5p/2=p+3\kappa^2/2$,
verifying both transfer and equality in the degree bound.
\end{proof}

\begin{lemma}\label{thm:noleaves}
Let $G$ be a connected bipartite unicyclic graph admitting PST between nonadjacent vertices $a,b$. Then $\deg(a)=\deg(b)=2$.
\end{lemma}
\begin{proof}
By Corollaries~\ref{cor:cycles} and~\ref{cor:bipartite}, the endpoints lie on the unique cycle and in the same bipartition class. Every neighbor counted in $k_a$ or $k_b$ lies on the unique cycle, so $k\leq2$, while PST gives $\delta_{+-}\geq\sqrt2$. By Lemma~\ref{lem:shared}\textup{(ii)}, every component of $G-\{a,b\}$ containing neighbors of exactly one endpoint is a single vertex, which is a leaf of $G$. The common endpoint degree has the form $D=p+2$, with $p$ leaves at each endpoint. If $p>0$, Theorem~\ref{thm:pendant} and $(\pi/\tau_0)^2\geq2$ give $D\geq p+3$, a contradiction. Thus $p=0$, so both endpoint degrees are two.
\end{proof}

\subsection{Degree-two vertices in bipartite unicyclic graphs}
\begin{proposition}\label{prop:distance}
Let $G$ be a connected bipartite unicyclic graph, and let $a,b$ be nonadjacent vertices of degree two. If $G$ admits PST between $a,b$, then $\dist(a,b)=2$.
\end{proposition}
\begin{proof}
In $C_4$, any two nonadjacent vertices are opposite and have distance
two. Otherwise, the endpoints lie on the unique cycle by
Corollary~\ref{cor:cycles}. They lie in the same bipartition class, and $\tau_0=\pi/\sqrt2$, by Corollary~\ref{cor:bipartite}.

Suppose, for a contradiction, that they have no common neighbor.

\proofheading{Step 1: determine $\Phi^-$.}
The endpoint degrees are two and $N(a)\cap N(b)=\varnothing$, so
$\langle w_\pm,A^2w_\pm\rangle=2\pm|N(a)\cap N(b)|=2$.
One of $\Phi^+,\Phi^-$ consists of odd multiples of $\sqrt2$,
and the other consists of even multiples. Every eigenvalue in the former
set has square at least two. The corresponding spectral measure has
second moment two, so it is concentrated on $\{-\sqrt2,\sqrt2\}$.
This set cannot be $\Phi^+$, because $\rho(G)\in\Phi^+$ and
$\rho(G)\geq2$ for a graph containing a cycle. Hence it is $\Phi^-$, and
\begin{equation}\label{eq:a2difference}
 A^2(\ee_a-\ee_b)=2(\ee_a-\ee_b).
\end{equation}

\proofheading{Step 2: reduce to an eight-cycle.}
Every vertex at distance two from $a$ must therefore be at distance two from $b$, with the same number of two-step walks. If a cycle neighbor $v$ of $a$ had an off-cycle neighbor $u$, the two-step walk from $a$ to $u$ would require a two-step walk from $b$ to $u$. As $b$ is on the unique cycle and $u$ is off the cycle, that walk would also have to pass through $v$, making $v$ a common neighbor, a contradiction. The same argument applies at the cycle neighbors of $b$.

Let $\ell$ be the length of the unique cycle. Identify its vertices with $\Z/\ell\Z$ in cyclic order, with $a=0$ and $b=j$. The cycle length $\ell$ is even because $G$ is bipartite;
the case $\ell=4$ would give two common neighbors, so $\ell\geq6$. Equality of the sets of vertices at distance two from $a$ and $b$ along the cycle gives
\[
 \{-2,2\}=\{j-2,j+2\}\quad\text{in }\Z/\ell\Z.
\]
Matching $-2$ with $j-2$ gives $j=0$, contrary to $a\ne b$. The other matching gives $j-2=2$ and $j+2=-2$, hence $j\equiv4\equiv-4\pmod\ell$ and $\ell\mid8$. Hence $\ell=8$ and $j=4$. Vertices $0,1,3,4,5,7$ have degree two, and trees can be attached only at $2$ and $6$.

\proofheading{Step 3: compare the spectral mass at zero.}
Let $E_0$ be the projector onto $\ker A$. A vector $z\in\ker A$ satisfies, from the rows at vertices $1,3,5,7$,
\[
 z_0=z_4=s,\qquad z_2=z_6=-s.
\]
Thus $\|z\|^2\geq4|s|^2=2|\langle w_+,z\rangle|^2$. For $z=E_0w_+$, orthogonal projection gives
$\langle w_+,z\rangle=\|z\|^2$. Thus
$\|z\|^2\geq2\|z\|^4$, so $\|z\|^2\leq1/2$, including the case
$z=0$. Equivalently,
\begin{equation}\label{eq:zero-upper}
 \langle w_+,E_0w_+\rangle\leq\frac12.
\end{equation}
But $\Phi^+$ consists of even multiples of $\sqrt2$. Its nonzero elements have square at least $8$, whereas the spectral measure of $w_+$ has second moment $2$. Therefore
\[
 2\geq8\bigl(1-\langle w_+,E_0w_+\rangle\bigr),
 \qquad \langle w_+,E_0w_+\rangle\geq\frac34,
\]
contradicting \eqref{eq:zero-upper}. The endpoints must share a neighbor and hence have distance two.
\end{proof}

It remains to exclude exactly one common neighbor. The following tree lemma supplies the required inverse bound. Here $C^\dagger$ is the Moore--Penrose inverse of a real symmetric adjacency matrix.

\begin{lemma}\label{lem:treeinverse}
Let $T$ be a tree, let $C=A(T)$, and let $P_0$ project onto $\ker C$. Suppose distinct vertices $x,y$ lie in the same bipartition class and
\[
 P_0\ee_x=P_0\ee_y\ne0.
\]
Then, with $z_-=(\ee_x-\ee_y)/\sqrt2$,
\[
 \|C^\dagger z_-\|^2\geq1.
\]
The constant is attained by the endpoints of $P_5$.
\end{lemma}
\begin{proof}
Because $C$ is real symmetric and $P_0$ is the orthogonal projector
onto $\ker C$, we have
\[
 \operatorname{range}C=(\ker C)^\perp=\ker P_0.
\]
The assumption $P_0\ee_x=P_0\ee_y$ gives
$P_0z_-=(P_0\ee_x-P_0\ee_y)/\sqrt2=0$, so
$z_-\in\operatorname{range}C$.
Since $P_0\ee_x=P_0\ee_y\ne0$, neither $\ee_x$ nor $\ee_y$ lies in $\operatorname{range}C$.
Set $f=C^\dagger z_-$. The spectral definition of the Moore--Penrose
inverse gives $CC^\dagger=I-P_0$ and $P_0C^\dagger=0$;
hence $Cf=z_-$ and $f\perp\ker C$.
Let $X,Y$ be the bipartition classes, with $x,y\in X$, and decompose
$f=f_X+f_Y$ into its restrictions to the two classes, extended by zero.
Since $C$ maps vectors supported on one bipartition class to vectors supported on the other, and $z_-$ is supported on $X$,
we have $Cf_X=0$ and $Cf_Y=z_-$. Thus $f_X\in\ker C$, and
\[
 0=\langle f,f_X\rangle=\|f_X\|^2.
\]
Consequently $f$ is supported on the class opposite to $x,y$.

Let $v$ be the first neighbor of $x$ along the unique $x$--$y$ path. Remove the edge $xv$, let $R_x$ be the component containing $x$, and let $f_{R_x}$ be the restriction of $f$ to $R_x$, extended by zero elsewhere. Since $y\notin R_x$ and $xv$ is the only edge leaving $R_x$, the only potentially nonzero coordinates of $Cf_{R_x}$ are $x$ and $v$. At $v$ we have $(Cf_{R_x})_v=f_x=0$, whereas at $x$ the restriction removes the summand $f_v$. Hence
\[
 C f_{R_x}=(1/\sqrt2-f_v)\ee_x.
\]
The vector $Cf_{R_x}$ belongs to $\operatorname{range}C$, whereas
$\ee_x$ does not. Therefore $1/\sqrt2-f_v=0$, giving
$f_v=1/\sqrt2$.
Let $w$ be the path neighbor of $y$, and delete $yw$.
If $R_y$ is the component containing $y$ and $f_{R_y}$ is the
restriction extended by zero, the same coordinate calculation gives
$Cf_{R_y}=(-1/\sqrt2-f_w)\ee_y$.
Since $\ee_y\notin\operatorname{range}C$, we obtain $f_w=-1/\sqrt2$.
If $v=w$, the two forced values contradict each other. Otherwise
$\|f\|^2\geq|f_v|^2+|f_w|^2=1$.
For $P_5$ in path order, take $x=1$ and $y=5$.
Its kernel is spanned by $h=(1,0,-1,0,1)^T$, so
$P_0\ee_1=P_0\ee_5=h/3\ne0$, verifying the hypothesis.
The vector $f=(0,1,0,-1,0)^T/\sqrt2$ solves $Cf=z_-$ and is
orthogonal to $\ker C$. It is therefore $C^\dagger z_-$ and has
squared norm one, proving sharpness.
\end{proof}

\begin{theorem}\label{thm:twins}
Let $G$ be a connected bipartite unicyclic graph, and let $a,b$ be nonadjacent vertices of degree two. If $G$ admits PST between $a,b$, then there are distinct vertices $c,d$ such that
\[
 N(a)=N(b)=\{c,d\}.
\]
Thus the unique cycle is $C_4$, $a,b$ are its opposite vertices, and all other vertices belong to trees attached at $c$ or $d$.
\end{theorem}
\begin{proof}
\proofheading{Step 1: compute the two resolvents.}
By Proposition~\ref{prop:distance}, the endpoints share a neighbor. Suppose they share exactly one, denoted $c$, and write
\[
 N(a)=\{c,x\},\qquad N(b)=\{c,y\},\qquad x\ne y.
\]
This assumption rules out $G\cong C_4$. Deleting $a,b$ leaves two trees: $T_c$ containing $c$, and $T$ containing $x,y$. Put $C=A(T)$, $z_\pm=(\ee_x\pm\ee_y)/\sqrt2$, and define
\[
 q_\pm(t)=z_\pm^T(tI-C)^{-1}z_\pm,
 \qquad g_c(t)=\ee_c^T(tI-A(T_c))^{-1}\ee_c.
\]
Since $a,b$ lie in the same bipartition class, $x,y$ lie in the same class of $T$.
In $G-\{a,b\}=T_c\sqcup T$, the endpoint neighbor vectors are
$\ee_c\oplus\ee_x$ and $\ee_c\oplus\ee_y$. Their diagonal resolvent
quadratic forms are $g_c+g_{T,x}$ and $g_c+g_{T,y}$.
Cospectrality of $a,b$, together with the two-by-two Schur-complement calculation in the proof of Theorem~\ref{thm:gap}, gives $g_c+g_{T,x}=g_c+g_{T,y}$. Hence $g_{T,x}=g_{T,y}$, so $x,y$ are cospectral in $T$.
The symmetric neighbor vector is $\sqrt2\,\ee_c\oplus z_+$, whereas
the antisymmetric one is $0\oplus z_-$. Since the two tree blocks are
disjoint, their resolvent quadratic forms are $2g_c+q_+$ and $q_-$.
Schur complementation therefore gives
\begin{equation}\label{eq:onecommonres}
 g_{G,w_-}(t)=\frac1{t-q_-(t)},\qquad
 g_{G,w_+}(t)=\frac1{t-q_+(t)-2g_c(t)}.
\end{equation}
All three functions $q_+,q_-,g_c$ have positive residues summing to one and arise from spectral measures symmetric about zero.

\proofheading{Step 2: bound the spectral mass at zero from below.}
By Corollary~\ref{cor:bipartite}, the two supports $\Phi^+,\Phi^-$ are contained in the sets of even and odd multiples of $\sqrt2$, in some order. The corresponding spectral measures have second moments
\[
 \langle w_-,A^2w_-\rangle=1,\qquad
 \langle w_+,A^2w_+\rangle=3.
\]
Let $E_0^G$ be the orthogonal projector onto $\ker A(G)$. The second moment $1$ rules out odd multiples of $\sqrt2$ in $\Phi^-$. Thus $\Phi^-\subseteq2\sqrt2\,\Z$. Every nonzero element of $\Phi^-$ has square at least $8$, and the spectral mass $p_0=\|E_0^G w_-\|^2$ satisfies
\begin{equation}\label{eq:seven-eighths}
 1\geq8(1-p_0),\qquad p_0\geq7/8.
\end{equation}
The eigenvalues in $\Phi^+$ are odd multiples of $\sqrt2$. The second moment $3$ forces $\sqrt2\in\Phi^+$, since the next possible positive eigenvalue is $3\sqrt2$.

\proofheading{Step 3: compare the projections onto $\ker C$.}
Let $P_0$ be the orthogonal projector onto $\ker C$.
By \eqref{eq:onecommonres}, a pole of $q_-$ at zero would make
$g_{G,w_-}=1/(t-q_-)$ vanish there. This contradicts its positive
spectral mass $p_0\geq7/8$, so $q_-$ has no pole at zero.
The residue of $q_-$ there is
$\|P_0z_-\|^2=\tfrac12\|P_0\ee_x-P_0\ee_y\|^2$.
Its vanishing gives
\begin{equation}\label{eq:equal-kernel}
 P_0\ee_x=P_0\ee_y.
\end{equation}
We next show that $P_0\ee_x\ne0$.

Set $t_0=\sqrt2$ and $\alpha_+(t)=t-q_+(t)-2g_c(t)$. Its zeros are $\Phi^+$, which excludes zero and has smallest positive element $t_0$. Lemma~\ref{lem:first-positive-zero}, applied with coefficients $1,2$, gives $g_c(t_0)\leq1/t_0$. Since $\alpha_+(t_0)=0$,
\[
 q_+(t_0)=t_0-2g_c(t_0)\geq0.
\]
Lemma~\ref{lem:first-positive-zero} now gives
$\|P_0z_+\|^2>0$. By \eqref{eq:equal-kernel},
$P_0z_+=\sqrt2\,P_0\ee_x=\sqrt2\,P_0\ee_y$.
Thus the common projection is nonzero, as required by
Lemma~\ref{lem:treeinverse}.

\proofheading{Step 4: prove $p_0\leq1/2$.}
Let $F_\theta$ denote the spectral projectors of $C=A(T)$. Since
$q_-$ has no pole at zero, pair the poles $\pm\theta$ and use
$t/(t^2-\theta^2)=-t/\theta^2+O(t^3)$. Its expansion at zero is
\[
 q_-(t)=-Lt+O(t^3),\qquad
 L=\sum_{\theta\ne0}\frac{\|F_\theta z_-\|^2}{\theta^2}
   =\|C^\dagger z_-\|^2\geq1.
\]
The spectral mass at zero is the residue of $g_{G,w_-}$. Hence
\[
 p_0=\lim_{t\to0}\frac{t}{t-q_-(t)}
     =\frac1{1+L}\leq\frac12,
\]
contradicting \eqref{eq:seven-eighths}. Therefore the endpoints have two common neighbors, as claimed.
\end{proof}

\subsection{Odd cycles}
\begin{lemma}\label{lem:extremes}
Let $G$ be a connected unicyclic graph, let $a,b$ be distinct strongly cospectral vertices on its unique cycle, and put $H=G-\{a,b\}$. Then
\[
 \lambda_{\min}(G)<\lambda_{\min}(H),
\]
and both $\rho(G)$ and $\lambda_{\min}(G)$ belong to $\Phi_a=\Phi_b$.
\end{lemma}
\begin{proof}
Every component of the forest $H$ contains at most one neighbor of $a$ and at most one neighbor of $b$. Choose a component attaining $\lambda_{\min}(H)$. An eigenvector for its least eigenvalue can be obtained by reversing the signs of a strictly positive Perron eigenvector on one bipartition class, so every coordinate is nonzero. Extend this vector by zero to a vector $z$ on $G$. Its Rayleigh quotient is $\lambda_{\min}(H)$. By connectivity, the chosen component contains a neighbor $r$ of some endpoint $u\in\{a,b\}$, and $r$ is the only neighbor of $u$ in that component. Thus $(Az)_u=z_r\ne0$ while $z_u=0$, so $z$ is not an eigenvector of $G$. If $\lambda_{\min}(G)=\lambda_{\min}(H)$, equality in the Rayleigh principle would force $z$ to be an eigenvector for $\lambda_{\min}(G)$. Hence the inequality is strict.

If $\lambda_{\min}(G)$ were absent from the support at $a$, strong cospectrality would make it absent at $b$ too. Any nonzero eigenvector for $\lambda_{\min}(G)$ would then vanish at both endpoints. Its restriction to $H$ would be a nonzero eigenvector with the same eigenvalue, contradicting $\lambda_{\min}(G)<\lambda_{\min}(H)$. The Perron eigenvalue is supported at every vertex of a connected graph
\cite[Lemma~2.3]{godsil}.
\end{proof}

\begin{lemma}\label{lem:asymmetry}
Let $G$ be a connected unicyclic graph whose unique cycle has odd length. Then
\[
 0<\rho(G)+\lambda_{\min}(G)<2.
\]
\end{lemma}
\begin{proof}
For a connected graph, Perron--Frobenius theory implies that
$-\rho(G)$ is an eigenvalue only if the graph is bipartite.
The odd cycle excludes this case, so $\lambda_{\min}(G)>-\rho(G)$.
Delete a cycle edge $uw$ to obtain a tree $T$, and assign signs
$\varepsilon_r\in\{1,-1\}$ to $r\in V(T)$ so that
$\varepsilon_x\varepsilon_y=-1$ for every edge $xy\in E(T)$.
Since the unique cycle of $G$ has odd length, the $u$--$w$ path in
$T$ has even length, so $\varepsilon_u=\varepsilon_w$. For a positive unit Perron vector
$v$ of $G$, put $z_r=\varepsilon_rv_r$. Then $\|z\|=1$.
Every edge of $T$ changes sign in the quadratic form, whereas $uw$
does not, giving
\[
 z^TA(G)z=-\rho(G)+4v_uv_w.
\]
The Rayleigh principle therefore yields
\[
 \rho(G)+\lambda_{\min}(G)
 \leq4v_uv_w\leq2(v_u^2+v_w^2)<2.
\]
The last inequality is strict because $G$ has at least three vertices
and all coordinates of $v$ are positive.
\end{proof}
\begin{remark}
The stronger bound $0<\rho(G)+\lambda_{\min}(G)\leq1$ was proved by Braga, Moraes, and Santos whenever deleting one edge makes a connected nonbipartite graph $G$ bipartite \cite[Theorem~1.2]{asymmetry}. Theorem~\ref{thm:odd} needs only Lemma~\ref{lem:asymmetry}.
\end{remark}

\begin{theorem}\label{thm:odd}
Let $G$ be a connected unicyclic graph whose unique cycle has odd length. Then $G$ does not admit PST.
\end{theorem}
\begin{proof}
Suppose PST occurs between $a,b$. By Corollary~\ref{cor:cycles}, both endpoints lie on the unique cycle. By Lemma~\ref{lem:extremes}, the least and greatest eigenvalues of $G$ are both supported at $a$. Lemma~\ref{lem:parity} therefore gives
\[
 \rho(G)+\lambda_{\min}(G)\in2\Z,
\]
contrary to Lemma~\ref{lem:asymmetry}.
\end{proof}

\subsection{Adjacent vertices}
\begin{lemma}\label{lem:adjacent}
Let $G$ be a connected unicyclic graph, and let $a,b$ be adjacent vertices. If $G$ admits PST between $a,b$ with minimum transfer time $\tau_0$, then $\deg(a)=\deg(b)=2$ and
\[
 \frac{\pi}{\tau_0}\leq1+\sqrt3<3.
\]
\end{lemma}
\begin{proof}
\proofheading{Step 1: bound the eigenvalue displacements.}
The endpoints are on the unique cycle. In this proof, take $H=G-\{a,b\}$ and $C=A(H)$, and let $x,y$ be the endpoint neighbor indicator vectors in $H$. Define $F_\theta,d_\theta,c_\theta$ as in the proof of Theorem~\ref{thm:gap}. The reciprocal resolvents are
\[
 \alpha_\pm(t)=t\mp1-\sum_\theta\frac{d_\theta\pm c_\theta}{t-\theta}.
\]
Construct the arrowhead matrices $M_\pm$, coupling vectors $r_\pm$,
and selected positions $J_\pm$ as in the proof of Theorem~\ref{thm:gap},
now with central diagonal entries $1,-1$ in $M_+,M_-$, respectively.
The coefficient estimate \eqref{eq:residue} depends only on the
neighbor vectors in $H$, so it remains valid.
The added constants do not change the derivatives of $\alpha_\pm$
or their limits at poles and infinity. Their zeros still strictly
interlace the poles, and the same count gives \eqref{eq:common-positions}. Write $\xi=r_+-r_-$. After deleting $a,b$, the remaining cycle path contains exactly one neighbor of each endpoint; every other component meets at most one endpoint neighborhood. Applying \eqref{eq:residue} with $k=1$ therefore gives
\[
 D=M_+-M_-=
 \begin{pmatrix}2&\xi^T\\\xi&0\end{pmatrix},\qquad \|\xi\|^2\leq2.
\]
Its potentially nonzero eigenvalues are $1\pm\sqrt{1+\|\xi\|^2}$. Applying Corollary~\ref{cor:trace-displacement} to $M_+,M_-$ bounds the total
positive displacement by $1+\sqrt{1+\|\xi\|^2}\leq1+\sqrt3$
and the magnitude of the total negative displacement by
$\sqrt{1+\|\xi\|^2}-1\leq\sqrt3-1$.

\proofheading{Step 2: exclude branches attached to only one endpoint.}
At each position in $J_+\cap J_-$, the eigenvalue displacement has magnitude at least $\pi/\tau_0\geq\sqrt2$. Since $\sqrt3-1<\sqrt2$, none of these displacements can be negative. Since $1+\sqrt3<2\sqrt2$, at most one can be positive. If there are $m$ common poles of $\alpha_+$ and $\alpha_-$, \eqref{eq:common-positions} gives at least $m+1$ common selected positions. Hence $m=0$: for every $\theta$ with $d_\theta>0$, at least one of
$d_\theta+c_\theta$ and $d_\theta-c_\theta$ is zero.
Since these coefficients are $\tfrac12\|F_\theta(x\pm y)\|^2$,
we obtain $F_\theta x=\pm F_\theta y$ for every $\theta$, including zero.
In a component of $G-\{a,b\}$ containing a neighbor of $a$ and no neighbor of $b$, the restriction of $y$ vanishes.
The projector identities would then make every spectral projection
of the restriction of $x$ vanish, although that restriction is
nonzero. This is impossible because the spectral projectors sum
to the identity. Interchanging $a,b$ excludes components containing a neighbor of $b$ and no neighbor of $a$. Thus neither endpoint has an off-cycle branch. The one common position also gives the stated upper bound on $\pi/\tau_0$.
\end{proof}

\begin{theorem}\label{thm:noadjacent}
Let $G$ be a connected unicyclic graph. Then $G$ does not admit PST between adjacent vertices.
\end{theorem}
\begin{proof}
Suppose $a,b$ are adjacent transfer vertices. Theorem~\ref{thm:odd} shows that the unique cycle has even length $\ell\geq4$. By Lemma~\ref{lem:adjacent}, write $N(a)=\{b,x\}$ and $N(b)=\{a,y\}$. By \cite[Lemma~1.1]{godsil}, transfer
from $a$ to $b$ also occurs from $b$ to $a$ at the same time and with the
same phase. Thus $U(\tau_0)\ee_a=\gamma\ee_b$ implies
$U(\tau_0)\ee_b=\gamma\ee_a$.
The neighborhoods $N(a)=\{b,x\}$ and $N(b)=\{a,y\}$ give
$A\ee_a=\ee_b+\ee_x$ and $A\ee_b=\ee_a+\ee_y$.
Commuting $U(\tau_0)$ with $A$ gives
$U(\tau_0)\ee_b+U(\tau_0)\ee_x=\gamma(\ee_a+\ee_y)$;
canceling $U(\tau_0)\ee_b=\gamma\ee_a$ yields
\[
 U(\tau_0)\ee_x=\gamma\ee_y.
\]

\proofheading{Case 1: $\ell\geq6$.}
The vertices $x,y$ are nonadjacent and lie in opposite bipartition classes, contrary to Corollary~\ref{cor:bipartite}.

\proofheading{Case 2: $\ell=4$.}
The vertices $x,y$ are adjacent. Transfer between them and Lemma~\ref{lem:adjacent} give $\deg(x)=\deg(y)=2$. Thus all four cycle vertices have degree two, and connectivity forces $G\cong C_4$. For these adjacent vertices of $C_4$, the zero-eigenspace projector
satisfies
\[
 E_0\ee_a=\tfrac12(\ee_a-\ee_y),\qquad
 E_0\ee_b=\tfrac12(\ee_b-\ee_x).
\]
These nonzero orthogonal vectors cannot be equal up to sign, so $a,b$
are not strongly cospectral, a contradiction.
\end{proof}

\begin{proof}[Proof of Theorem~\ref{thm:geometry}]
\textup{(i)} By Corollary~\ref{cor:cycles}, $a,b$ lie on the unique cycle.
Theorem~\ref{thm:noadjacent} excludes adjacency, and
Theorem~\ref{thm:odd} shows that the cycle has even length, so $G$ is
bipartite. Lemma~\ref{thm:noleaves} and Theorem~\ref{thm:twins} then
give the four-cycle $a,c,b,d,a$ and $N(a)=N(b)=\{c,d\}$.
Thus $a,b$ are degree-two false twins, and the remaining vertices
form trees attached at $c,d$, giving $G=G(T_c,T_d)$.
A chord of this four-cycle would create another cycle, so it is induced.

\textup{(ii)} Suppose $G\not\cong C_4$ has a second unordered transfer pair.
By part \textup{(i)}, any transfer pair consists of opposite vertices
of the unique four-cycle. The second pair must therefore be $\{c,d\}$.
Part \textup{(i)} would also give $\deg(c)=\deg(d)=2$.
All four cycle vertices would then have degree two, so connectivity
would force $G\cong C_4$, a contradiction.

\textup{(iii)} Theorem~\ref{thm:cactus} gives minimum transfer time
$\pi/2$ for $C_4$ and $\pi/\sqrt2$ otherwise.
The equality $N(a)=N(b)$ gives $A\ee_a=A\ee_b$, hence
$Aw_-=0$ and $\Phi^-=\{0\}$.
At any transfer time, substituting $\mu=0$ into
\eqref{eq:phase} gives $1=-\gamma$, so the transfer phase is $-1$.
\end{proof}
\section{The rooted-tree criterion}\label{sec:trees}
Theorem~\ref{thm:geometry} reduces the classification to graphs $G(T_c,T_d)$. We now express the eigenvalue support at the transfer vertices in terms of the two root resolvents and prove Theorem~\ref{thm:criterion}.

For a rooted tree $(T,v)$, write $P_T=\phi_T$ and $Q_T=\phi_{T-v}$. Cramer's rule gives the root resolvent $g_{T,v}=Q_T/P_T$, abbreviated to $g_T$ when the root is understood. Empty products are $1$ and empty sums are $0$. The standard rooted-tree recursion takes the following form in this notation.

\begin{lemma}[{\cite[Corollary~1]{damnjanovic}}]\label{lem:rooted-tree-recursion}
Let $(T,v)$ be a rooted tree, and let $(T_1,v_1),\ldots,(T_k,v_k)$ be
the components of $T-v$, rooted at the respective neighbors of $v$. Then
\begin{equation}\label{eq:treerecursion}
 Q_T=\prod_{i=1}^k P_{T_i},\qquad
 P_T=t\prod_{i=1}^k P_{T_i}
 -\sum_{i=1}^k Q_{T_i}\prod_{j\ne i}P_{T_j}.
\end{equation}
Equivalently,
\begin{equation}\label{eq:rootrecursion}
 g_{T,v}(t)=\frac1{t-\sum_{i=1}^k g_{T_i,v_i}(t)}.
\end{equation}
\end{lemma}
The rational function assigned to the root in the cited formulation is $P_T/Q_T=1/g_T$.
A single vertex has $P_T=t$ and $Q_T=1$. Two consequences of
\eqref{eq:rootrecursion} will be used repeatedly. Let $h_m$ be the
resolvent at the center of $K_{1,m}$, and let $f_{k,m}$ be the root
resolvent of a tree whose root has $k$ children, each with $m$ leaf
children. For nonnegative integers $k,m$,
\begin{equation}\label{eq:elementary-resolvents}
 h_m(t)=\frac1{t-m/t}=\frac{t}{t^2-m},\qquad
 f_{k,m}(t)=\frac1{t-kh_m(t)}
           =\frac{t^2-m}{t(t^2-k-m)}.
\end{equation}
These identities include zero multiplicities after cancellation.

For any graph $G=G(T_c,T_d)$, use the abbreviations $g_c=g_{T_c,c}$, $g_d=g_{T_d,d}$, and $g_+=g_{G,w_+}$, where $w_+=(\ee_a+\ee_b)/\sqrt2$. The identity $Aw_-=0$ gives $(tI-A)^{-1}w_-=t^{-1}w_-$.
Since $w_+\perp w_-$, we have $w_+^T(tI-A)^{-1}w_-=0$.
Substitute $\ee_a=(w_++w_-)/\sqrt2$ and
$\ee_b=(w_+-w_-)/\sqrt2$ to obtain
\begin{equation}\label{eq:endpoint-resolvent}
 g_{G,a}(t)=g_{G,b}(t)=\frac12\left(g_+(t)+\frac1t\right).
\end{equation}

\begin{proof}[Proof of Theorem~\ref{thm:criterion}]
By Theorem~\ref{thm:geometry}, every connected unicyclic graph admitting PST between $a,b$ has the stated rooted-tree form. It therefore suffices to consider $G=G(T_c,T_d)$ with $a,b$ as the designated vertices. 

\proofheading{Step 1: identify the polynomial with roots $\Phi^+$.}
Deleting $a,b$ leaves $T_c\sqcup T_d$, and the two neighbor vectors
are equal: $x=y=\ee_c\oplus\ee_d$. In the Schur-complement calculation
leading to \eqref{eq:alpha}, this gives $f_x=h=g_c+g_d$. Hence
\begin{equation}\label{eq:plusg}
 g_+(t)=w_+^T(tI-A)^{-1}w_+
 =\frac1{t-2Q_c/P_c-2Q_d/P_d}
 =\frac{P_cP_d}{F}.
\end{equation}
After cancellation its denominator is $M$.
Both $F$ and $P_cP_d$ are monic integer polynomials. Their monic greatest
common divisor belongs to $\Z[t]$ by Gauss's lemma, so the reduced
numerator and denominator are also monic integer polynomials.
The resolvent has positive residues and simple poles, so
$M(t)=\prod_{\lambda\in\Phi^+}(t-\lambda)$.
Thus eigenvalues of the attached trees absent from $\Phi^+$ do not enter $M$.

\proofheading{Step 2: prove necessity.}
Since $Aw_-=0$, transfer forces the phase to be $-1$. If $G\cong C_4$, then $g_+(t)=t/(t^2-4)$, so $M(t)=t^2-4$. Otherwise Theorem~\ref{thm:cactus} gives $\tau_0=\pi/\sqrt2$, and the phase condition implies $\lambda/\sqrt2\in2\Z+1$ for every $\lambda\in\Phi^+$. The support $\Phi^+$ is symmetric about zero because $G$ is bipartite. Thus $M(t)=\prod_{q\in S}(t^2-2q^2)$ for a finite set $S$ of positive odd integers. In this case,
\[
 \langle w_+,A^2w_+\rangle=4.
\]
If $1\notin S$, the square of every eigenvalue in $\Phi^+$ is at least $18$, a contradiction. If $S=\{1\}$, the second moment is $2$, also a contradiction.

\proofheading{Step 3: prove sufficiency.}
Conversely, take $\tau=\pi/2$ in case \textup{(i)} and $\tau=\pi/\sqrt2$ in case \textup{(ii)}. Every root $\lambda$ of $M$ then satisfies
$e^{-i\tau\lambda}=-1$: in case \textup{(i)}, $\lambda=\pm2$, and in
case \textup{(ii)}, $\lambda=\pm q\sqrt2$ with $q$ odd.
Since $M$ has root set $\Phi^+$, spectral decomposition gives
$U(\tau)w_+=-w_+$. Also, $Aw_-=0$ gives $U(\tau)w_-=w_-$. Thus $U(\tau)\ee_a=-\ee_b$. If $M=t^2-4$, then $A^2w_+=4w_+$. An off-cycle neighbor of $c$ or $d$ would have coordinate $\sqrt2$ in $A^2w_+$ and coordinate zero in $4w_+$. Hence no such neighbor exists, and $G\cong C_4$. In the second case $\sqrt2$ is supported, forcing every positive transfer time to be an odd multiple of $\pi/\sqrt2$, so the minimum is $\pi/\sqrt2$. 
\end{proof}

\section{Sharp order bounds}\label{sec:extremal}

Throughout this section, let $G=G(T_c,T_d)\not\cong C_4$ admit
PST between $a,b$, and put
\[
 p=\deg_{T_c}(c),\qquad r=\deg_{T_d}(d),\qquad R=p+r>0,
 \qquad w=\frac{\ee_a+\ee_b}{\sqrt2}.
\]
By Theorem~\ref{thm:criterion}, the support of $w$ is
$\Phi^+=\{\pm q\sqrt2:q\in S\}$, where $S$ is a finite set of positive
odd integers with $1\in S$ and $|S|\geq2$. Its smallest positive
element is denoted by $t_0=\sqrt2$.

\subsection{Restrictions on the attached trees}
The least positive supported eigenvalue constrains each branch at $c,d$. A fourth-moment identity then bounds the total number of branches and identifies the graphs with $|\Phi^+|=4$.

\begin{lemma}\label{lem:extremal-root-branches}
Let $G=G(T_c,T_d)\not\cong C_4$ admit PST between $a,b$. For each $v\in\{c,d\}$, root every component $B$ of $T_v-v$ at its unique neighbor of $v$. Then:
\begin{enumerate}[label=\textup{(\roman*)}]
\item Both attachment roots have positive spectral mass at zero in their respective trees.
\item Every eigenvalue in the support of the root of a branch $B$ has absolute value at least $\sqrt2$. Consequently, the root of $B$ has degree at least two within $B$, and $B$ has at least three vertices.
\item If neither $T_c$ nor $T_d$ is a singleton, every eigenvalue in the support of the root of each branch $B$ has absolute value strictly greater than $\sqrt2$, and the root of every branch $B$ has degree at least three within $B$.
\end{enumerate}
\end{lemma}
\begin{proof}
\proofheading{Step 1: prove part \textup{(i)}.}
By \eqref{eq:plusg}, $g_{G,w}=1/\alpha$, where
$\alpha(t)=t-2g_c(t)-2g_d(t)$. Its zeros form $\Phi^+$, which excludes
zero and has smallest positive element $t_0$. Apply Lemma~\ref{lem:first-positive-zero} with coefficients $2,2$. Neither $g_c$ nor $g_d$ has a pole in $(0,t_0]$, and
\[
 g_c(t_0)\leq\frac1{t_0},\qquad g_d(t_0)\leq\frac1{t_0}.
\]
Equality in either bound means that the spectral measure at the corresponding root is concentrated at zero. The second moment of this measure is the degree of that root, so
concentration at zero is equivalent to degree zero.
For a connected tree this means that the tree is a singleton. Since $\alpha(t_0)=0$ and $t_0^2=2$,
\[
 g_c(t_0)+g_d(t_0)=\frac1{t_0}.
\]
Thus $g_c(t_0)$ and $g_d(t_0)$ are nonnegative. Lemma~\ref{lem:first-positive-zero} therefore gives positive spectral mass at zero at both roots. If $p,r>0$, neither tree is a singleton, so both upper bounds are strict. Together with $g_c(t_0)+g_d(t_0)=1/t_0$, this gives $g_c(t_0)>0$ and $g_d(t_0)>0$.

\proofheading{Step 2: prove parts \textup{(ii)} and \textup{(iii)}.}
Choose $v\in\{c,d\}$, put $T=T_v$, and consider its child branches, with root resolvents
$h_1,\ldots,h_m$. The tree recursion is
\[
 g_{T,v}(t)=\frac1{t-\sum_i h_i(t)}.
\]
No $h_i$ can have a pole at zero. Such a pole has positive residue and
cannot cancel against the other summands; it would make $g_{T,v}$ vanish
at zero, contradicting the positive spectral mass at zero just established.
Nor can an $h_i$ have a positive pole $\theta<t_0$: it would give a zero
of $g_{T,v}$ at $\theta$. But $g_{T,v}$ has no pole in $(0,t_0]$ and is
strictly decreasing there. Indeed, if $E_\lambda$ denotes a spectral projector of $A(T)$, then
\[
 g_{T,v}^{\prime}(t)
 =-\sum_{\lambda:\,E_\lambda\ee_v\ne0}
   \frac{\|E_\lambda\ee_v\|^2}{(t-\lambda)^2}<0
 \qquad(0<t\leq t_0).
\]
A zero at $\theta<t_0$ would therefore force
$g_{T,v}(t_0)<0$. If $g_{T,v}(t_0)>0$, a pole of any $h_i$ at $t_0$ is
also excluded. Symmetry of the spectral measures at the branch roots now shows that zero is absent from their supports and that every supported eigenvalue satisfies $|\lambda|\geq\sqrt2$, with strict inequality when $p,r>0$.

Let $k_i$ be the degree of a branch root within its branch. Its second spectral moment equals $k_i$.
Since every supported eigenvalue has square at least two, $k_i\geq2$.
When $p,r>0$, all these squares exceed two, so the integer $k_i$
satisfies $k_i\geq3$. Its branch has at least $1+k_i$ vertices,
proving parts \textup{(ii)} and \textup{(iii)}.
\end{proof}

\begin{lemma}\label{lem:R14}
Let $G=G(T_c,T_d)\not\cong C_4$ admit PST between
$a,b$, and put $p=\deg_{T_c}(c)$, $r=\deg_{T_d}(d)$, and $R=p+r$.
Let $S$ be as in Theorem~\ref{thm:criterion}, and put
$q=\min(S\setminus\{1\})$ and $Q=\max S$. Then
\[
 2q^2-4\leq R\leq2Q^2-4.
\]
Equality in either inequality holds if and only if $S=\{1,q\}$,
in which case $q=Q$ and $G\cong H_{p,r}$ under an isomorphism
preserving the attachment roots and the transfer pair.
In particular, $R\geq14$, and $R=14$ implies $G\cong H_{p,r}$.
\end{lemma}
\begin{proof}
\proofheading{Step 1: apply the fourth-moment bounds.}
Put $w=(\ee_a+\ee_b)/\sqrt2$, and let $z$ be the sum of the standard
basis vectors at the $R$ branch roots.
From $N(a)=N(b)=\{c,d\}$ we obtain
$Aw=\sqrt2(\ee_c+\ee_d)$.
The neighbors of $c,d$ are $a,b$ and the branch roots, so
$A(\ee_c+\ee_d)=2(\ee_a+\ee_b)+z$.
Moreover $w\perp z$ and $\|z\|^2=R$. Consequently
\[
 A^2w=4w+\sqrt2\,z,\qquad
 \langle w,A^2w\rangle=4,\qquad
 \langle w,A^4w\rangle=16+2R.
\]
For any real $h$, these moments give
\begin{equation}\label{eq:R-fourth-bound}
\begin{split}
 \langle w,(A^2-2I)(A^2-2h^2I)w\rangle
 &=(16+2R)-4(2+2h^2)+4h^2\\
 &=2R-4h^2+8.
\end{split}
\end{equation}
Every eigenvalue in $\Phi^+$ has square $\mu=2j^2$ with $j\in S$.
For $h=q$, the polynomial $(\mu-2)(\mu-2q^2)$ is nonnegative
on this set, giving $R\geq2q^2-4$.
For $h=Q$, the polynomial $(\mu-2)(\mu-2Q^2)$ is nonpositive,
giving $R\leq2Q^2-4$.
Every element of $\Phi^+$ has positive spectral mass. Equality in
either bound therefore requires the square of every eigenvalue in $\Phi^+$
to be a zero of the corresponding polynomial, which is equivalent to
$S=\{1,q\}$ and $q=Q$.
Since $q\geq3$, we also have $R\geq14$.

\proofheading{Step 2: identify the branches when $S=\{1,q\}$.}
Suppose $S=\{1,q\}$, so $R=2q^2-4$ and
$(A^2-2I)(A^2-2q^2I)w=0$.
Since $A^4w=16w+4\sqrt2\,z+\sqrt2 A^2z$, substitution in the
annihilating identity gives
$\sqrt2 A^2z=2Rw+\sqrt2(R+2)z$, or equivalently
\[
 A^2z=R(\ee_a+\ee_b)+(R+2)z.
\]
If a branch root had a descendant at distance two within its branch,
the coordinate of $A^2z$ at that descendant would be positive,
whereas the corresponding coordinate on the right would be zero.
Thus every branch is a star. At a branch root attached to $c$ and having $k$ children,
there are $k+1$ two-step walks returning to that root and one
two-step walk to each of the other $p-1$ branch roots at $c$.
Thus its coordinate in $A^2z$ is $(k+1)+(p-1)=k+p$.
Comparing with $(R+2)z$ gives $k+p=R+2$, hence $k=r+2$.
Similarly, every branch root attached to $d$ has $p+2$ children.
This is precisely $H_{p,r}$.
Finally, $R=14$ forces $q=3$ and equality in the lower bound.
\end{proof}

\subsection{The extremal family and the order bounds}
The next calculation verifies the transfer condition for $H_{p,r}$. Together with Lemma~\ref{lem:R14}, it also characterizes the graphs other than $C_4$ whose transfer vertices have eigenvalue support of size five.

\begin{proposition}\label{prop:Hpr}
Let $p,r$ be nonnegative integers with $R=p+r>0$. The eigenvalue support of each of the vertices $a,b$ in $H_{p,r}$ is
\[
 \{0,\pm\sqrt2,\pm\sqrt{R+4}\}.
\]
The graph $H_{p,r}$ admits PST between $a,b$ if and only if $R+4=2q^2$ for an odd integer $q\geq3$.
\end{proposition}
\begin{proof}
\proofheading{Step 1: compute the attached-tree resolvents.}
The attachments are precisely the trees in \eqref{eq:elementary-resolvents}:
\[
 g_c(t)=f_{p,r+2}(t)=\frac{t^2-r-2}{t(t^2-R-2)},\qquad
 g_d(t)=f_{r,p+2}(t)=\frac{t^2-p-2}{t(t^2-R-2)}.
\]
If $p=0$ or $r=0$, cancellation of common factors reduces the corresponding fraction to $1/t$.

\proofheading{Step 2: compute the supported eigenvalues.}
Adding the two fractions gives
\[
 g_c(t)+g_d(t)=\frac{2t^2-R-4}{t(t^2-R-2)}.
\]
Substitute this into $g_{H_{p,r},w}=1/(t-2g_c-2g_d)$ from
\eqref{eq:plusg}. The denominator after clearing fractions is
\[
 \begin{split}
 t^2(t^2-R-2)-2(2t^2-R-4)
 &=t^4-(R+6)t^2+2R+8\\
 &=(t^2-2)(t^2-R-4).
 \end{split}
\]
Therefore
\begin{equation}\label{eq:Hpr-resolvent}
 g_{H_{p,r},w}(t)=
 \frac{t(t^2-R-2)}{(t^2-2)(t^2-R-4)}.
\end{equation}
For $R>0$, the numerator is nonzero at $t^2=2$ and $t^2=R+4$:
its factor $t^2-R-2$ takes the values $-R$ and $2$, respectively.
Thus there is no cancellation, and
$\Phi^+=\{\pm\sqrt2,\pm\sqrt{R+4}\}$.
By \eqref{eq:endpoint-resolvent}, $\Phi_a=\Phi_b=\{0\}\cup\Phi^+$. The transfer assertion follows from Theorem~\ref{thm:criterion}.
\end{proof}

\begin{proof}[Proof of Theorem~\ref{thm:sharp-unicyclic-order}]
\proofheading{Step 1: prove part \textup{(i)}.}
Lemma~\ref{lem:R14} gives $R\geq14$, and
Lemma~\ref{lem:extremal-root-branches}\textup{(ii)} gives
$|V(G)|\geq4+3R$. It remains to establish the additional term $2pr$
and the equality characterization.

\proofheading{Case 1: $pr=0$.}
Here $4+3R=4+3R+2pr$. Equality in $|V(G)|\geq4+3R$ forces every branch to consist of
its root and exactly two leaf children, which is precisely $H_{p,r}$
when one parameter is zero.

\proofheading{Case 2: $p,r>0$.}
Let $(B_i,u_i)$, $1\leq i\leq p$, be the branches of $T_c-c$, rooted at the neighbors of $c$, and let $(D_j,v_j)$, $1\leq j\leq r$, be the analogous branches of $T_d-d$. Put $k_i=\deg_{B_i}(u_i)$, $\ell_j=\deg_{D_j}(v_j)$, and $h_i=g_{B_i,u_i}$. The degrees $k_i,\ell_j$ are at least three. Define the positive integers
\[
 K=\sum_i(k_i-2),\qquad L=\sum_j(\ell_j-2).
\]
Let $\lambda_i$ be a random variable with the spectral distribution at $u_i$ in $B_i$.
By Lemma~\ref{lem:extremal-root-branches}\textup{(iii)}, $\lambda_i^2-2>0$ on each
support. Each measure has total mass one and second moment $k_i$.
Apply the Cauchy--Schwarz inequality to the functions
$\sqrt{\lambda_i^2-2}$ and $1/\sqrt{\lambda_i^2-2}$ on the disjoint
union of their supports, equipped with the corresponding spectral measures. Their product is one and the
total mass is $p$, so
\[
 p^2\leq
 \left(\sum_i\mathbb E(\lambda_i^2-2)\right)
 \left(\sum_i\mathbb E\frac1{\lambda_i^2-2}\right)
 =K\sum_i\mathbb E\frac1{\lambda_i^2-2}.
\]
Symmetry of the measures and $t_0^2=2$ therefore yield
\[
 \sum_i h_i(t_0)
 =-t_0\sum_i\mathbb E\frac1{\lambda_i^2-2}
 \leq-\frac{t_0p^2}{K}.
\]
Applying root recursion at $c$ and the same argument at $d$ gives
\[
 g_c(t_0)\leq\frac{K}{t_0(K+p^2)},\qquad
 g_d(t_0)\leq\frac{L}{t_0(L+r^2)}.
\]
Equation~\eqref{eq:plusg} at $t_0=\sqrt2$ gives $g_c(t_0)+g_d(t_0)=1/t_0$, so
\[
 0\leq\frac{K}{K+p^2}+\frac{L}{L+r^2}-1
   =\frac{KL-p^2r^2}{(K+p^2)(L+r^2)}.
\]
Hence $KL\geq p^2r^2$, and the arithmetic--geometric mean inequality
gives $K+L\geq2\sqrt{KL}\geq2pr$. Counting each branch root
and its children now yields
\[
 |V(G)|\geq4+R+\sum_i k_i+\sum_j\ell_j
          =4+3R+K+L\geq4+3R+2pr.
\]
If equality holds, every branch is a star and $K+L=2pr$.
Together with $KL\geq p^2r^2$, this gives $K=L=pr$.
The upper bounds for $g_c(t_0)$ and $g_d(t_0)$ now sum to $1/t_0$, so both
must be equalities, as must the corresponding Cauchy--Schwarz
inequalities. At $c$, equality in the Cauchy--Schwarz inequality implies that
$\lambda_i^2-2$ has the same constant value on the eigenvalue support at every branch root.
Taking second moments gives $k_i-2=K/p=r$ for every $i$.
Likewise $\ell_j-2=L/r=p$ for every $j$.
Thus these stars form exactly $H_{p,r}$.
Conversely, \eqref{eq:Horder} gives equality for $H_{p,r}$
in both cases. This proves \textup{(i)}.

\proofheading{Step 2: prove part \textup{(ii)}.}
Part \textup{(i)} gives $|V(G)|\geq4+3R+2pr\geq46$.
Equality forces $R=14$ and $pr=0$, giving $H_{0,14}$ up to
interchange of $c,d$. This graph has PST by
Proposition~\ref{prop:Hpr} and order $46$ by \eqref{eq:Horder}.

\proofheading{Step 3: prove part \textup{(iii)}.}
For $p,r>0$, the inequality $(p-1)(r-1)\geq0$ gives
$pr\geq p+r-1=R-1$, and therefore
\[
 |V(G)|\geq4+3R+2(R-1)=5R+2\geq72.
\]
Equality forces $R=14$ and $\{p,r\}=\{1,13\}$, and the preceding
equality characterization gives $H_{1,13}$. Conversely,
Proposition~\ref{prop:Hpr} and \eqref{eq:Horder} show that $H_{1,13}$
has PST and order $72$.
\end{proof}

\section{Inverse construction and prescribed supports}\label{sec:seven}
\subsection{Finite inverse construction}\label{subsec:recursive-construction}
The criterion in Theorem~\ref{thm:criterion} tests a given pair of
attached trees. We now prescribe the eigenvalue support and recover all
possible pairs by inverting the rooted-tree recursion
\eqref{eq:rootrecursion}.
We then solve the local branching constraints explicitly for seven-element supports.
For a monic polynomial $P\in\Z[t]$ with simple real roots, let
$\mathcal R(P)$ consist of odd rational functions $a/b$
(that is, $a(-t)/b(-t)=-a(t)/b(t)$), where
$a,b\in\Z[t]$ are monic and coprime, $b\mid P$ is nonconstant,
$\deg a=\deg b-1$, and the roots of $a$ strictly interlace those of $b$.
Interlacing is vacuous for $a=1$; set $\mathcal R(1)=\varnothing$.
The positive-residue characterization of strict interlacing
\cite[Lemma~4]{trees} states that such a quotient $a/b$ has positive
residue $a(\theta)/b'(\theta)$ at every root $\theta$ of $b$. The residues sum to $1$ because $a/b=t^{-1}+O(t^{-3})$ at infinity.
The set $\mathcal R(P)$ is finite and can be enumerated exactly.
There are finitely many nonconstant monic divisors $b$ of $P$.
For each such $b$, put $m=\deg b$ and choose $L$ so that all roots of
$P$ lie in $[-L,L]$. Interlacing then places all $m-1$ roots of $a$
in that interval, so Vieta's formulas give
$|[t^{m-1-j}]a|\leq\binom{m-1}{j}L^j$ for $0\leq j\leq m-1$,
where $[t^k]a$ denotes the coefficient of $t^k$ in $a(t)$.
These bounds leave finitely many integer numerators $a$; the resulting
fractions are tested for coprimality, oddness, and interlacing.

A function in $\mathcal R(P)$ need not be the root resolvent of a tree.
For each such fraction $g=a/b$, the following recursion defines the
collection $\mathcal T(g)$ of its rooted-tree realizations.
Let $\mathcal T(1/t)$ contain only the one-vertex rooted tree.
Otherwise choose nonnegative integers $n_f$, for
$f=A_f/B_f\in\mathcal R(a)$, satisfying
\begin{equation}\label{eq:recursive-splitting}
 \sum_f n_f A_f(t)\frac{a(t)}{B_f(t)}=t a(t)-b(t).
\end{equation}
Join a new root to $n_f$ trees chosen from $\mathcal T(f)$ for each $f$,
allowing repetition. Retain the resulting trees up to rooted isomorphism;
if no choice succeeds, set $\mathcal T(g)=\varnothing$.

\begin{lemma}\label{lem:inverse-root-recursion}
For every monic $P\in\Z[t]$ with simple real roots and every
$g=a/b\in\mathcal R(P)$, the recursion \eqref{eq:recursive-splitting}
terminates and produces a finite collection consisting, up to rooted
isomorphism, of exactly the trees with root resolvent $g$.
Every such tree has depth at most $\deg b-1$.
\end{lemma}
\begin{proof}
Write $g=t^{-1}+d(g)t^{-3}+O(t^{-5})$. Its poles and residues define a probability measure: each pole has mass equal to its residue. This measure is concentrated at zero exactly when $g=1/t$; otherwise its second moment $d(g)$ is positive. If $m=\deg b$, expansion at
infinity gives $ta-b=d(g)t^{m-2}+O(t^{m-4})$, so $d(g)$ is an integer.
Each monic summand $A_fa/B_f$ has degree $m-2$; comparing these leading
coefficients in \eqref{eq:recursive-splitting} gives $\sum_f n_f=d(g)$,
so every step has finitely many choices. The reduced denominator $B_f$ of each child resolvent divides $a$, so it has smaller degree than $b$. Recursion therefore terminates.
Dividing \eqref{eq:recursive-splitting} by $a$ gives
$t-\sum_f n_f f=b/a$.
Once the child trees have their prescribed resolvents, root recursion
\eqref{eq:rootrecursion} gives $(t-\sum_f n_f f)^{-1}=a/b=g$
at the new root.
Conversely, the child resolvents of a tree with resolvent $a/b$ sum to
$t-b/a$. Their positive residues prevent cancellation of poles in their sum, so the reduced denominator of every child resolvent divides $a$. Each child resolvent is odd by bipartiteness. Cramer's rule and Gauss's lemma show that its reduced numerator and denominator are monic polynomials with integer coefficients; positivity of the residues implies that their roots strictly interlace. Thus it belongs to $\mathcal R(a)$ and satisfies \eqref{eq:recursive-splitting}.
Induction recovers every tree; the case $1/t$ forces root degree zero.
Finally, the eigenvalue support at the root consists of the $m=\deg b$ poles of $g=a/b$,
so Lemma~\ref{lem:support-distance} gives the asserted depth bound.
\end{proof}

Fix a finite set $S$ of positive odd integers, with $1\in S$ and
$s=|S|\geq2$, and put $M_S(t)=\prod_{q\in S}(t^2-2q^2)$.
Enumerate monic odd polynomials $\widetilde N\in\Z[t]$ of degree $2s-1$
whose roots strictly interlace those of $M_S$, with
$[t^{2s-3}]\widetilde N=4-2\sum_{q\in S}q^2$.
For each $\widetilde N$, choose $g_i=a_i/b_i\in\mathcal R(\widetilde N)$, $i=c,d$, satisfying
\begin{equation}\label{eq:initial-root-splitting}
 2a_c\frac{\widetilde N}{b_c}+2a_d\frac{\widetilde N}{b_d}=t\widetilde N-M_S,
\end{equation}
and form $G(T_c,T_d)$ for every $T_i\in\mathcal T(g_i)$.

\begin{theorem}\label{thm:recursive-construction}
Let $S$ be a finite set of positive odd integers with $1\in S$ and $|S|\geq2$. Consider the construction \eqref{eq:initial-root-splitting}, followed by \eqref{eq:recursive-splitting}. Then:
\begin{enumerate}[label=\textup{(\roman*)}]
\item Up to isomorphism preserving the unordered transfer pair, the construction produces exactly the connected unicyclic graphs admitting PST between $a,b$ with eigenvalue support
\[
 \{0\}\cup\{\pm q\sqrt2:q\in S\}.
\]
\item The construction terminates without requiring a prior order bound. In every graph in \textup{(i)}, both attached trees have depth at most $2|S|-2$.
\end{enumerate}
\end{theorem}
\begin{proof}
\textup{(i)} Let $G$ be a connected unicyclic graph admitting PST between
$a,b$ with the stated eigenvalue support. Theorem~\ref{thm:geometry} gives
the two trees and \eqref{eq:plusg} gives $g_+=\widetilde N/M_S$ with integral
monic $\widetilde N$. Bipartiteness implies that $g_+$ is odd, so $\widetilde N$ is odd. Positivity of the residues of $g_+$ implies that the roots of $\widetilde N$ strictly interlace those of $M_S$. Let $\nu$ be the coefficient of $t^{2s-3}$ in $\widetilde N$.
Since $M_S=t^{2s}-2(\sum_{q\in S}q^2)t^{2s-2}+O(t^{2s-4})$,
\[
 \frac{\widetilde N(t)}{M_S(t)}
 =\frac1t+\frac{\nu+2\sum_{q\in S}q^2}{t^3}+O(t^{-5}).
\]
The second moment $\langle w_+,A^2w_+\rangle=4$ gives
$\nu=4-2\sum_{q\in S}q^2$.
Since $g_c+g_d=(t-M_S/\widetilde N)/2$ and the residues of $g_c,g_d$ are positive, their reduced denominators both divide $\widetilde N$. Thus \eqref{eq:initial-root-splitting} holds,
and Lemma~\ref{lem:inverse-root-recursion} recovers both trees.
Conversely, \eqref{eq:initial-root-splitting} gives $g_+=\widetilde N/M_S$ for every
constructed pair. Strict interlacing prevents cancellation, and
Theorem~\ref{thm:criterion} gives the required transfer and support.

\textup{(ii)} The coefficient bound used to enumerate $\mathcal R(M_S)$
gives finitely many choices for $\widetilde N$. For each choice,
$\mathcal R(\widetilde N)$ is finite, so there are finitely many pairs
$g_c,g_d$ to test in \eqref{eq:initial-root-splitting}.
Lemma~\ref{lem:inverse-root-recursion} then gives finitely many rooted
trees for each pair, without a prior order bound.
The reduced denominator $b_i$ of $g_i$ divides $\widetilde N$, whose degree is
$2s-1$. The same lemma gives depth at most $\deg b_i-1\leq2s-2$
for each attached tree. By \textup{(i)}, this applies to every graph
with the prescribed transfer and support.
\end{proof}

The procedure may return the empty set. The next subsection gives
explicit branching rules and proves nonemptiness for every $S$ with $|S|=3$.

\subsection{Explicit rules for seven supported eigenvalues}\label{subsec:seven-explicit}
For seven-element eigenvalue supports, the inverse construction reduces to the following integer branching rules.

Fix odd integers $3\leq u<v$, and put
\[
 x=2u^2,\qquad y=2v^2,\qquad \sigma=x+y-2.
\]
Choose nonnegative integers $p_c,p_d$, write $R=p_c+p_d$, and set
\[
 \beta=(x-2)(y-2)+4+2R,\qquad s_i=\sigma-p_i
 \quad(i\in\{c,d\}).
\]
Choose positive integers $\eta_c,\eta_d$ subject to
\begin{equation}\label{eq:seven-global}
 x-4<R<y-4,\qquad \eta_c+\eta_d=xy.
\end{equation}
For each $i\in\{c,d\}$, choose integers $k_{ij}\geq2$, with
$1\leq j\leq p_i$, satisfying
\begin{equation}\label{eq:seven-budget}
 \sum_{j=1}^{p_i}k_{ij}=p_i s_i+\eta_i-\beta.
\end{equation}
For every $j$, choose $k_{ij}$ nonnegative integers $t_{ij\ell}$ such that
\begin{equation}\label{eq:seven-local}
 0\leq t_{ij\ell}\leq s_i-k_{ij},\qquad
 \sum_{\ell=1}^{k_{ij}}t_{ij\ell}
       =k_{ij}(s_i-k_{ij})-\eta_i.
\end{equation}
Empty sums in \eqref{eq:seven-budget} are zero.

Construct a rooted tree $T_i$ as follows. Its root $i$ has $p_i$ children.
The $j$th child is the root of a branch $B_{ij}$ and has $k_{ij}$ children.
The $\ell$th of these has $t_{ij\ell}$ children, each of which has
$s_i-k_{ij}-t_{ij\ell}$ leaf children. There are no other vertices or
edges; Figure~\ref{fig:branch} shows the branching multiplicities.
In particular, $p_i=0$ gives a singleton. Denote the graph obtained
from these two rooted trees by $G(T_c,T_d)$, as in Section~\ref{sec:introduction}.

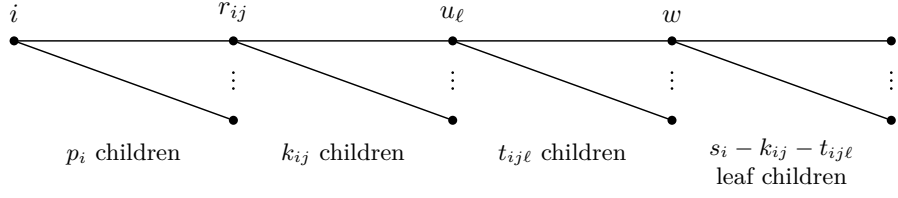
\begin{figure}[htbp]
\centering
\begin{tikzpicture}[pst graph]
  \node[pst vertex] (i) at (0,0) {};
  \node[pst vertex] (r) at (2.9,0) {};
  \node[pst vertex] (u) at (5.8,0) {};
  \node[pst vertex] (w) at (8.7,0) {};
  \node[pst vertex] (l) at (11.6,0) {};
  \foreach \x/\name in {2.9/r2,5.8/u2,8.7/w2,11.6/l2}{
    \node[pst vertex] (\name) at (\x,-1.05) {};
    \pstvdots{\x}{-.525}}
  \draw (i)--(r) (i)--(r2) (r)--(u) (r)--(u2)
    (u)--(w) (u)--(w2) (w)--(l) (w)--(l2);
  \node[above=5pt] at (i) {$i$};
  \node[above=5pt] at (r) {$r_{ij}$};
  \node[above=5pt] at (u) {$u_\ell$};
  \node[above=5pt] at (w) {$w$};
  \node[pst note] at (1.45,-1.5) {$p_i$ children};
  \node[pst note] at (4.35,-1.5) {$k_{ij}$ children};
  \node[pst note] at (7.25,-1.5) {$t_{ij\ell}$ children};
  \node[pst note,align=center] at (10.15,-1.59)
    {$s_i-k_{ij}-t_{ij\ell}$\\leaf children};
\end{tikzpicture}
\caption{Branching in $T_i$, $i\in\{c,d\}$, with $B_{ij}$ rooted at $r_{ij}$. Only one child's descendants are shown at each level; other children of $i$ and $r_{ij}$ may have different descendants. A zero multiplicity means that the corresponding children are absent.}
\label{fig:branch}
\end{figure}

\begin{theorem}\label{thm:seven-construction}
Let $3\leq u<v$ be odd integers, let $G$ be a connected unicyclic graph, and let $a,b$ be distinct vertices. Then:
\begin{enumerate}[label=\textup{(\roman*)}]
\item The graph $G$ admits PST between $a,b$ with eigenvalue support
\[
 \Phi_a=\{0,\pm\sqrt2,\pm u\sqrt2,\pm v\sqrt2\}
\]
if and only if, with $a,b$ as its designated opposite vertices, it is obtained by the construction \eqref{eq:seven-global}--\eqref{eq:seven-local}.
\item For each fixed pair $u,v$, the construction has finitely many choices. Every attached tree produced by the construction has depth at most four.
\end{enumerate}
\end{theorem}

\begin{proof}
\proofheading{Step 1: prove necessity in part \textup{(i)}.}
Suppose first that $G$ has the stated transfer and support. By
Theorem~\ref{thm:geometry}, write $G=G(T_c,T_d)$ and
$p_i=\deg_{T_i}(i)$. Put $w=(\ee_a+\ee_b)/\sqrt2$ and
$\Psi(t)=(t^2-2)(t^2-x)(t^2-y)$. The reduced denominator of $g_+$ is $\Psi$.
Since the spectral measure of $w$ has second and fourth moments
\[
 \langle w,A^2w\rangle=4,\qquad
 \langle w,A^4w\rangle=16+2R,
\]
its resolvent has the expansion
$g_+(t)=t^{-1}+4t^{-3}+(16+2R)t^{-5}+O(t^{-7})$.
Write the odd numerator as $t(t^4+\nu_1t^2+\nu_0)$.
Since $\Psi(t)=t^6-(x+y+2)t^4+(xy+2x+2y)t^2-2xy$,
multiplication by the expansion of $g_+$ gives
\[
 \nu_1=4-(x+y+2)=-\sigma,\qquad
 \nu_0=16+2R-4(x+y+2)+xy+2x+2y=\beta.
\]
Hence
\begin{equation}\label{eq:seven-plus}
 g_+(t)=\frac{tN(t)}{\Psi(t)},\qquad
 N(t)=t^4-\sigma t^2+\beta.
\end{equation}
Thus $\widetilde N(t)=tN(t)$ in the notation of Section~\ref{subsec:recursive-construction}.
Here $S=\{1,u,v\}$ has three elements, so Lemma~\ref{lem:R14}
gives the strict bounds $x-4<R<y-4$ in \eqref{eq:seven-global}.

The relation $g_c+g_d=(t-1/g_+)/2$ shows that every pole of $g_c$ or
$g_d$ is a zero of $tN(t)$: positive residues prevent cancellation
between the two root resolvents. Each reduced root resolvent is a
quotient of monic integer polynomials. Its denominator divides $tN(t)$,
so its numerator over this common denominator is also monic and integral.
The root has degree $p_i$, so
$g_i(t)=t^{-1}+p_it^{-3}+O(t^{-5})$. Multiplying by $tN(t)$ gives
a monic even polynomial with leading terms
$t^4+(p_i-\sigma)t^2$; denote its constant term by $\eta_i$.
This yields
\begin{equation}\label{eq:seven-root}
 g_i(t)=\frac{N_i(t)}{tN(t)},\qquad
 N_i(t)=t^4-s_i t^2+\eta_i.
\end{equation}
Lemma~\ref{lem:extremal-root-branches}\textup{(i)} gives positive
spectral mass at zero at both roots. This mass is the residue $\eta_i/\beta$; since
$\beta>0$, we have $\eta_i>0$. Moreover,
$g_c+g_d=(t-1/g_+)/2$ gives
\[
 N_c(t)+N_d(t)=\frac{t^2N(t)-\Psi(t)}2
             =2t^4-(2\sigma-R)t^2+xy.
\]
Comparing constant terms yields $\eta_c+\eta_d=xy$.

\proofheading{Step 2: recover each branch.}
Let $h_{ij}$ be the root resolvent of the $j$th branch at $i$.
Applying \eqref{eq:rootrecursion} at the attachment root $i$ gives
\[
 \sum_{j=1}^{p_i}h_{ij}(t)
 =t-\frac1{g_i(t)}
 =\frac{t\bigl(p_i t^2+\eta_i-\beta\bigr)}{N_i(t)}.
\]
Again, positivity of the residues of the functions $h_{ij}$ implies that the eigenvalue support at each branch root is contained in the zero set of $N_i$. Consequently $N_i h_{ij}$ is a
polynomial. Let $k_{ij}$ be the degree of the root of $B_{ij}$ within
that branch. The expansion
$h_{ij}=t^{-1}+k_{ij}t^{-3}+O(t^{-5})$ gives
$N_i h_{ij}=t^3+(k_{ij}-s_i)t+O(t^{-1})$.
Since the left side is a polynomial, its negative-power terms vanish,
which proves
\begin{equation}\label{eq:seven-branch}
 h_{ij}(t)=
 \frac{t(t^2-s_i+k_{ij})}{N_i(t)},
\end{equation}
where cancellation is allowed. Summing proves
\eqref{eq:seven-budget}, and
Lemma~\ref{lem:extremal-root-branches}\textup{(ii)} gives $k_{ij}\geq2$.

Put $C=A(B_{ij})$ and let $r_{ij}$ be the branch root.
Every eigenvalue supported at $r_{ij}$ is a zero of $N_i$, so
spectral decomposition gives
\begin{equation}\label{eq:branch-annihilator}
 (C^4-s_iC^2+\eta_iI)\ee_{r_{ij}}=0.
\end{equation}
A vertex at distance four from $r_{ij}$ would have a positive
coordinate in $C^4\ee_{r_{ij}}$ and zero coordinates in the other
two terms, contradicting this equation. Thus the branch has depth
at most three.
Let $t_{ij\ell}$ be the number of children of the $\ell$th child
of $r_{ij}$. The vector $C^2\ee_{r_{ij}}$
has value $k_{ij}$ at the root, value $1$ at every grandchild, and
value $0$ elsewhere. Thus
$\|C^2\ee_{r_{ij}}\|^2=k_{ij}^2+\sum_\ell t_{ij\ell}$, or
$(C^4\ee_{r_{ij}})_{r_{ij}}=k_{ij}^2+\sum_{\ell=1}^{k_{ij}}t_{ij\ell}$, giving
\[
 \sum_{\ell=1}^{k_{ij}}t_{ij\ell}=k_{ij}(s_i-k_{ij})-\eta_i.
\]
At a grandchild $w$ reached through the $\ell$th child, let $m$ be its
number of children. In $C^4\ee_{r_{ij}}=C^2(C^2\ee_{r_{ij}})$, the nonzero contributions to
the coordinate indexed by $w$ come from the root, from $w$, and from the other
$t_{ij\ell}-1$ grandchildren with the same parent as $w$.
They contribute $k_{ij}$, $m+1$, and $t_{ij\ell}-1$, respectively.
Hence
\[
 (C^2\ee_{r_{ij}})_w=1,\qquad
 (C^4\ee_{r_{ij}})_w=k_{ij}+(t_{ij\ell}-1)+(m+1).
\]
Thus $k_{ij}+t_{ij\ell}+m=s_i$, and the $m$ children are leaves.
At the root, the equation
$k_{ij}(s_i-k_{ij})=\eta_i+\sum_\ell t_{ij\ell}>0$
gives $k_{ij}<s_i$.
If $t_{ij\ell}>0$, $k_{ij}+t_{ij\ell}+m=s_i$ and $m\geq0$ give
$t_{ij\ell}\leq s_i-k_{ij}$; if $t_{ij\ell}=0$, the same inequality
follows from $s_i-k_{ij}>0$.
This proves every condition in \eqref{eq:seven-local}.

\proofheading{Step 3: prove sufficiency in part \textup{(i)}.}
Conversely, construct the branches from the specified integers.
Fix $B_{ij}$ and abbreviate
$s=s_i$, $k=k_{ij}$, $\eta=\eta_i$, and $L=s-k$.
Write $t_\ell=t_{ij\ell}$, so
$\sum_{\ell=1}^k t_\ell=kL-\eta$.
Since $kL=\eta+\sum_\ell t_\ell>0$ and $k\geq2$, we have $L>0$.
The $\ell$th child's subtree has the form in
\eqref{eq:elementary-resolvents}, so its root resolvent is
\[
 f_\ell(t)=f_{t_\ell,L-t_\ell}(t)
 =\frac{t^2-L+t_\ell}{t(t^2-L)}.
\]
For $t_\ell=0$, cancellation of the common factor $t^2-L$ gives $f_\ell(t)=1/t$.
Summing over the $k$ children gives
\[
 \sum_{\ell=1}^k f_\ell(t)
 =\frac{k(t^2-L)+\sum_\ell t_\ell}{t(t^2-L)}
 =\frac{kt^2-\eta}{t(t^2-L)}.
\]
A final recursion at the branch root therefore yields
\[
 h_{ij}(t)=\frac1{t-\sum_\ell f_\ell(t)}
 =\frac{t(t^2-L)}{t^4-(L+k)t^2+\eta}
 =\frac{t(t^2-s_i+k_{ij})}{N_i(t)},
\]
which is \eqref{eq:seven-branch}.
Summing the functions $h_{ij}$ and using \eqref{eq:seven-budget} gives
\[
 \sum_{j=1}^{p_i}h_{ij}(t)
 =\frac{t(p_it^2+\eta_i-\beta)}{N_i(t)},\qquad
 t-\sum_{j=1}^{p_i}h_{ij}(t)=\frac{tN(t)}{N_i(t)},
\]
where the last equality uses $s_i+p_i=\sigma$.
Root recursion gives \eqref{eq:seven-root}. If $p_i=0$,
\eqref{eq:seven-budget} forces $\eta_i=\beta$ and $N_i=N$,
so this also gives $g_i=1/t$.
Finally, $s_c+s_d=2\sigma-R$ and $\eta_c+\eta_d=xy$ imply
\[
 t^2N(t)-2\bigl(N_c(t)+N_d(t)\bigr)
 =(t^2-2)(t^2-x)(t^2-y)=\Psi(t).
\]
Substitution in $g_+=1/(t-2g_c-2g_d)$ proves \eqref{eq:seven-plus}.

The inequalities in \eqref{eq:seven-global} give
\[
\begin{aligned}
 N(\sqrt{x})&=2(R-y+4)<0,\\
 N(\sqrt{y})&=2(R-x+4)>0,\\
 N(\sqrt2)&=(x-4)(y-4)+2R>0.
\end{aligned}
\]
There is therefore no cancellation at any of the six roots of $\Psi$,
and Theorem~\ref{thm:criterion} gives the prescribed support and
PST. This completes \textup{(i)}.

For \textup{(ii)}, \eqref{eq:seven-global} bounds the integers
$R,p_c,p_d,\eta_c,\eta_d$ for fixed $u,v$.
Since $\eta_i>0$, \eqref{eq:seven-local} gives
$k_{ij}(s_i-k_{ij})=\eta_i+\sum_\ell t_{ij\ell}>0$, hence
$2\leq k_{ij}<s_i$. Each $t_{ij\ell}$ is an integer in $[0,s_i-k_{ij}]$, so there are only finitely many choices.
The construction has at most four levels below each attachment root.
\end{proof}

\begin{corollary}\label{cor:three-frequency-realization}
Let $u<v$ be odd integers with $u\geq3$, and put $x=2u^2$ and $y=2v^2$.
For every integer $R$ satisfying
\[
 x-4<R<y-4,
\]
there is a connected unicyclic graph $G(T_c,T_d)$ in which $T_c$
is a singleton, $\deg_{T_d}(d)=R$, and $T_d$ has depth four.
This graph admits PST between $a,b$ at minimum time
$\pi/\sqrt2$, with eigenvalue support
\[
 \{0,\pm\sqrt2,\pm u\sqrt2,\pm v\sqrt2\},
\]
and has order
\[
 4+3R+(3x-5)(R-x+4)(y-4-R).
\]
\end{corollary}
\begin{proof}
Put $\alpha=R-x+4$, $\delta=y-4-R$, and $K=x+\delta$.
Both $\alpha$ and $\delta$ are positive. Start with the four-cycle
$a,c,b,d,a$. At $d$, attach $x-4$ copies of the three-vertex star rooted at its center,
and $\alpha$ copies of the following rooted tree $B$.
The root of $B$ has two leaf children and $\delta$ further children;
each further child has $x-2$ children, each of which has two leaf children.
Thus $B$ has $3+\delta(3x-5)$ vertices, and there are
$x-4+\alpha=R$ branches at $d$.

We verify the construction using Theorem~\ref{thm:seven-construction}.
With $\sigma=x+y-2$ and $\beta=(x-2)(y-2)+4+2R$, take
\[
 p_c=0,\qquad p_d=R,\qquad
 \eta_c=\beta,\qquad \eta_d=2K.
\]
The identities $K=x+y-4-R$ and $\beta+2K=xy$ give
$s_d=\sigma-R=K+2$ and the required sum
$\eta_c+\eta_d=xy$. Both $\eta_c$ and $\eta_d$ are positive; the assumed bounds on $R$ now verify \eqref{eq:seven-global}.
At $c$, \eqref{eq:seven-budget} has an empty sum and holds because
$\eta_c=\beta$.
At $d$, the sum of the branch-root degrees is
\[
 2(x-4)+\alpha(\delta+2)=2R+\alpha\delta.
\]
To verify \eqref{eq:seven-budget} at $d$, use the identities $R+4=x+\alpha$, $K=x+\delta$,
and $y=x+\alpha+\delta$ to obtain
\[
 \begin{split}
 R s_d+\eta_d-\beta
 &=2R+K(R+4)-xy\\
 &=2R+(x+\delta)(x+\alpha)-x(x+\alpha+\delta)\\
 &=2R+\alpha\delta.
 \end{split}
\]

For a three-vertex star, $k=2$ and both child parameters are zero;
\eqref{eq:seven-local} holds because $k(s_d-k)=2K=\eta_d$.
For $B$, take $k=\delta+2$, so $s_d-k=x$.
Its child parameters are two zeros and $\delta$ copies of $x-2$.
They lie in $[0,x]$ and have sum
\[
 \delta(x-2)=(\delta+2)x-2K=k(s_d-k)-\eta_d.
\]
Every grandchild then has $x-(x-2)=2$ leaf children, as specified.
Thus \eqref{eq:seven-local} holds for every branch, and
Theorem~\ref{thm:seven-construction} gives the stated support and transfer.
The minimum time follows from Theorem~\ref{thm:geometry}\textup{(iii)}.
Since $\alpha,\delta>0$ and $x-2>0$, the attachment has depth four.
Finally, its branch orders give
\[
 \begin{split}
 |V(G)|&=4+3(x-4)+\alpha\bigl(3+\delta(3x-5)\bigr)\\
       &=4+3R+(3x-5)(R-x+4)(y-4-R).\qedhere
 \end{split}
\]
\end{proof}

\section{Degree and order thresholds beyond five eigenvalues}\label{sec:examples}
For a prescribed eigenvalue support, the construction in
Section~\ref{sec:seven} reduces order and degree minimization to a finite
set of integer parameter choices. Bounds on the maximum degree or order
also restrict the possible supports, as shown below. We first determine
the least possible maximum degree outside $C_4$ and the family $H_{p,r}$,
then the minimum order for seven-element eigenvalue support.

\subsection{The maximum-degree threshold}
The spectral radius bounds the possible support sets under a degree bound.
The following bound is due to Hu \cite{hu}.
\begin{lemma}[{\cite[Theorem~4.3, $\alpha=0$]{guozhou}}]\label{lem:unicyclic-rho-bound}
Let $G$ be a connected unicyclic graph, and put $\Delta=\Delta(G)$. Then
\[
 \rho(G)\leq2\sqrt{\Delta-1}.
\]
\end{lemma}
\begin{remark}\label{rem:bounded-degree-finiteness}
Godsil proved finiteness, up to isomorphism, of connected graphs of
bounded maximum degree admitting PST
\cite[Corollary~10.2]{godsil}. A spectral-radius bound also implies this
conclusion, since $\deg(v)=\|A\ee_v\|^2\leq\rho(G)^2$.
Here the inverse construction gives explicit support restrictions and
an enumeration. For a connected unicyclic graph $G\not\cong C_4$ admitting PST,
the positive Perron eigenvector is not orthogonal to
$w_+=(\ee_a+\ee_b)/\sqrt2$. Thus
\[
 \rho(G)=\sqrt2\max S.
\]
For $B>0$, the bound $\rho(G)\leq B$ restricts $S$ to the positive odd integers
at most $B/\sqrt2$. If $\Delta(G)\leq D$, where $D\geq2$ is an integer,
Lemma~\ref{lem:unicyclic-rho-bound} instead gives
$\max S\leq\sqrt{2(D-1)}$. Counting the odd integers up to this bound
and using Theorem~\ref{thm:recursive-construction}\textup{(ii)} bounds
each attachment depth by
\[
 2\left\lfloor\frac{\sqrt{2(D-1)}+1}{2}\right\rfloor-2.
\]
For either bound, enumerate the eligible sets $S$ containing $1$ and
at least one other odd integer, and apply
Theorem~\ref{thm:recursive-construction}. Retain the graphs
satisfying the bound and include $C_4$ when allowed. This gives exactly
the required finite class.
\end{remark}

We now classify the connected unicyclic graphs of maximum degree at most $20$ that admit PST.

\begin{proof}[Proof of Theorem~\ref{thm:degree20-classification}]
\textup{(i)} Suppose first that $G$ admits PST and $\Delta(G)\leq20$. The graph $C_4$ has maximum degree two and satisfies
Theorem~\ref{thm:criterion}\textup{(i)}. For $G\not\cong C_4$, write $G=G(T_c,T_d)$, put $p=\deg_{T_c}(c)$ and $r=\deg_{T_d}(d)$, and let $S$ be as in Theorem~\ref{thm:criterion}.
Lemma~\ref{lem:unicyclic-rho-bound} gives
$\rho(G)\leq2\sqrt{19}<7\sqrt2$, so $S\subseteq\{1,3,5\}$.
If $|S|=2$, Lemma~\ref{lem:R14} gives $G\cong H_{p,r}$.
The possibility $S=\{1,5\}$ would require $p+r=46$, whereas
$p,r\leq18$. Hence $p+r=14$, giving exactly the listed graphs.
It remains to exclude $S=\{1,3,5\}$.

We apply Theorem~\ref{thm:seven-construction} with $u=3$ and $v=5$. Put $R=p+r$ and
$\beta=772+2R$. Then $15\leq R\leq45$, and the two attachment roots have parameters
$\eta_c,\eta_d>0$ satisfying $\eta_c+\eta_d=900$.
For a nonsingleton tree $T_i$ with $z$ branches, write $s=66-z$ and $\eta=\eta_i$, and use $X=t^2$.
For a branch of $T_i-i$, let $k$ be its root degree within the branch; the root then has degree $k+1$ in $G$. The squares of the eigenvalues supported at this root are roots of $X^2-sX+\eta$. The children of its root have $t_1,\ldots,t_k$ children, respectively, where
\[
 \sum_jt_j=k(s-k)-\eta.
\]
Whenever $t_j>0$, each of these $t_j$ vertices has
$m_j=s-k-t_j$ leaf children. Since $\Delta(G)\leq20$,
\[
 z\leq18,\qquad k\leq19,\qquad t_j\leq19,\qquad m_j\leq19
 \quad(t_j>0).
\]
Each root of $X^2-sX+\eta$ is the square of a pole of at least one branch resolvent.
Indeed, one quarter of the discriminant of $N_0(X)=X^2-66X+\beta$ is
$33^2-\beta=317-2R$. For $15\leq R\leq45$, this is an odd integer
strictly between $15^2$ and $17^2$. The only square in that open
interval is $16^2$, which is even. Thus the discriminant is not a square,
and $N_0$ is irreducible over $\mathbb Q$. For $z>0$, the numerator
$X^2-(66-z)X+\eta$ is coprime to $N_0$, since equality of these two monic
quadratics would give $z=0$. The numerator $t^4-st^2+\eta$ also has
nonzero constant term, so it is coprime to $t$. Thus the root resolvent of $T_i$ has all five poles.
Its positive residues give four simple zeros strictly interlacing these
poles. By root recursion, those zeros are precisely the poles of the sum
of the branch resolvents; positive residues prevent cancellation
between branches.

A branch whose root has an eigenvalue support of size two must be a star
rooted at its center. Indeed, if $C$ is its adjacency matrix and $v$ its root of degree $k$,
symmetry and the second moment identify those eigenvalues as
$\pm\sqrt{k}$. Hence $C^2\ee_v=k\ee_v$, which excludes a vertex
at distance two from $v$. Conversely, a star with at least one edge,
rooted at its center, has exactly two eigenvalues in the support of its root. The two roots of
$X^2-sX+\eta$ cannot both be realized only by star branches: their sum is $s=66-z\geq48>38$, whereas a star
branch has at most $19$ leaves. Thus each nonsingleton $T_i$ has a branch whose root has an eigenvalue support of size four. Such a branch has some $t_j>0$, and hence
$s-k=t_j+m_j\leq38$. Combining this with $k\leq19$ gives
$66-z=s\leq57$, so $z\geq9$.

\proofheading{Case 1: one attached tree is a singleton.}
Take $p=0$ and $r>0$. The parameters of $T_d$ are
\[
 s=66-r,\qquad \eta_d=900-\beta=128-2r.
\]
Choose a branch of $T_d$ whose root has an eigenvalue support of size four, and let $k$ be its root degree.
We have proved $9\leq r\leq18$.
The inequality $s-k\leq38$, with $s=66-r$, gives $k\geq28-r$;
together with $k\leq19$, this yields $28-r\leq k\leq19$. Since every $t_j\leq19$,
\[
 \eta_d=k(s-k)-\sum_jt_j\geq k(47-r-k).
\]
The concave quadratic on the right takes the same value
$19(28-r)$ at both ends of the interval $[28-r,19]$. Therefore
\[
 128-2r=\eta_d\geq19(28-r),
\]
which is impossible because $404-17r>0$ for $r\leq18$.

\proofheading{Case 2: both attached trees are nontrivial.}
We have $9\leq p,r\leq18$. The identities \eqref{eq:seven-budget} give
\[
 \eta_c=\beta-p(66-p)+\sum_{j=1}^{p}k_{cj}
       \leq\beta-p(47-p),
\]
and the analogous bound at $d$. Adding yields
\[
 900\leq1544-43(p+r)+p^2+r^2.
\]
The function $z\mapsto z^2-43z$ is strictly decreasing for $9\leq z\leq18$. Consequently, after interchanging $c,d$ if necessary, the only possible pairs are
$(p,r)=(9,9)$ and $(9,10)$: the right side is $884$ at
$(10,10)$ and $886$ at $(9,11)$.

Thus $p=9$ and $s=57$. In any branch of $T_c$ whose root has an eigenvalue support of size four,
$k\leq19$ and $s-k\leq38$ force $k=19$. Every positive $t_j$ then
satisfies $t_j+m_j=38$ with $t_j,m_j\leq19$, so $t_j=m_j=19$.
Thus $\eta_c$ is a multiple of $19$. Moreover,
\[
 \eta_c\leq\beta-342\leq468.
\]
A star branch with $2\leq j\leq19$ leaves would require
$\eta_c=j(57-j)$. Since $19\mid\eta_c$ and $57\equiv0\pmod{19}$,
we have $j^2\equiv0\pmod{19}$. The prime $19$ therefore divides $j$,
so $j=19$ and hence
$\eta_c=722$, a contradiction. All nine branches therefore have root
degree $19$, and \eqref{eq:seven-budget} gives
\[
 \eta_c=\beta-342\in\{466,468\},
\]
neither value being divisible by $19$. This completes the exclusion.

Conversely, every listed $H_{p,14-p}$ has PST by
Proposition~\ref{prop:Hpr} and maximum degree at most $16$.
Together with the $C_4$ case, this proves \textup{(i)}.

\textup{(ii)} For the listed graphs, the maximum degree is $16$ when $p=0$
and $17-p$ when $1\leq p\leq7$. Hence the minimum outside $C_4$ is $10$, attained only by $H_{7,7}$. Graphs of maximum degree greater than $20$ cannot affect this minimum. 

\textup{(iii)} By part \textup{(i)}, no connected unicyclic graph admitting PST outside $C_4$ and the family $H_{p,r}$ has maximum degree at most $20$.
Hence each such graph has maximum degree at least $21$. Sharpness follows from
Example~\ref{ex:degree21}.
\end{proof}

\begin{example}\label{ex:degree21}
There is a connected unicyclic graph of maximum degree $21$ outside
the family $H_{p,r}$ that admits PST.
Define rooted trees $B_c,B_d$ as follows.
\begin{itemize}
\item The root of $B_c$ has six leaf children and $14$ further
children. Each of these $14$ children has $18$ children, each having
$20$ leaf children.
\item The root of $B_d$ has five leaf children and $14$ further
children. Seven of these $14$ children have $19$ children, each having
$16$ leaf children; the other seven have $20$ children, each
having $15$ leaf children.
\end{itemize}

The branching multiplicities are shown in Figure~\ref{fig:degree21-branches}.

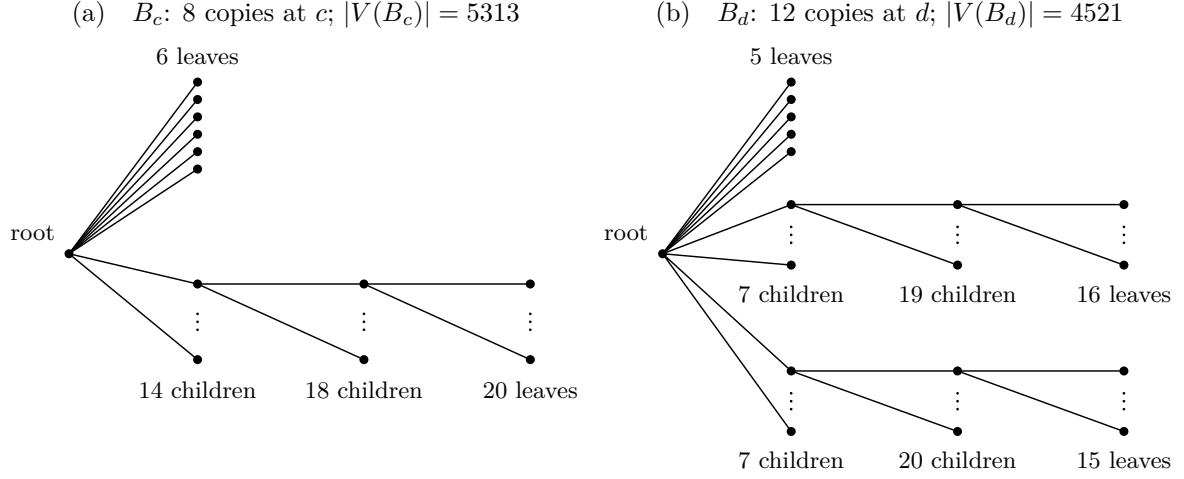
\begin{figure}[htbp]
\centering
\begin{tikzpicture}[pst graph]
\begin{scope}
  \node at (3,3.16) {\textup{(a)}\quad $B_c$: $8$ copies at $c$; $|V(B_c)|=5313$};
  \node[pst vertex] (cr) at (0,0) {};
  \node[above left=3pt,pst note] at (cr) {root};
  \foreach \y in {1.12,1.35,1.58,1.81,2.04,2.27}{
    \node[pst vertex] (cl) at (1.7,\y) {};
    \draw (cr)--(cl);}
  \node[pst note] at (1.7,2.62) {$6$ leaves};
  \foreach \x/\name/\parent/\count/\kind in
    {1.7/cu/cr/14/children,3.9/cw/cu/18/children,6.1/cz/cw/20/leaves}{
    \node[pst vertex] (\name) at (\x,-.4) {};
    \node[pst vertex] (\name b) at (\x,-1.4) {};
    \draw (\parent)--(\name) (\parent)--(\name b);
    \pstvdots{\x}{-.9}
    \node[pst note] at (\x,-1.8) {$\count$ \kind};
  }
\end{scope}
\begin{scope}[xshift=7.85cm]
  \node at (3,3.16) {\textup{(b)}\quad $B_d$: $12$ copies at $d$; $|V(B_d)|=4521$};
  \node[pst vertex] (dr) at (0,0) {};
  \node[above left=3pt,pst note] at (dr) {root};
  \foreach \y in {1.35,1.58,1.81,2.04,2.27}{
    \node[pst vertex] (dl) at (1.7,\y) {};
    \draw (dr)--(dl);}
  \node[pst note] at (1.7,2.62) {$5$ leaves};
  \foreach \x/\name/\parent/\count/\kind in
    {1.7/du/dr/7/children,3.9/dw/du/19/children,6.1/dy/dw/16/leaves}{
    \node[pst vertex] (\name) at (\x,.65) {};
    \node[pst vertex] (\name b) at (\x,-.15) {};
    \draw (\parent)--(\name) (\parent)--(\name b);
    \pstvdots{\x}{.25}
    \node[pst note] at (\x,-.53) {$\count$ \kind};
  }
  \foreach \x/\name/\parent/\count/\kind in
    {1.7/dv/dr/7/children,3.9/dx/dv/20/children,6.1/dz/dx/15/leaves}{
    \node[pst vertex] (\name) at (\x,-1.55) {};
    \node[pst vertex] (\name b) at (\x,-2.35) {};
    \draw (\parent)--(\name) (\parent)--(\name b);
    \pstvdots{\x}{-1.95}
    \node[pst note] at (\x,-2.73) {$\count$ \kind};
  }
\end{scope}
\end{tikzpicture}
\caption{The branches $B_c,B_d$ in Example~\ref{ex:degree21}. Labels within each branch give the child counts; children in the same group have isomorphic descendant subtrees. Dots indicate omitted children. Each root gains one edge when attached at $c$ or $d$.}
\label{fig:degree21-branches}
\end{figure}

Form $G(T_c,T_d)$ by attaching eight copies of $B_c$ to $c$ and
$12$ copies of $B_d$ to $d$. To apply
Theorem~\ref{thm:seven-construction}, take
\[
 (u,v)=(3,5),\quad (p_c,p_d)=(8,12),\quad R=20,\quad
 \sigma=66,\quad\beta=812,
\]
with $(s_c,\eta_c)=(58,508)$ and $(s_d,\eta_d)=(54,392)$.
Equation~\eqref{eq:seven-global} holds because $14<20<46$ and $508+392=900$.
The roots of the branches at $c,d$ have degrees $20,19$, respectively,
within their branches. Their degree sums satisfy \eqref{eq:seven-budget}:
\[
 8\cdot20=8\cdot58+508-812,\qquad
 12\cdot19=12\cdot54+392-812.
\]
For $B_c$, the child parameters are six zeros and $14$ copies of
$18$, with $L=58-20=38$. For $B_d$, they are five zeros, seven copies
of $19$, and seven copies of $20$, with $L=54-19=35$.
All parameters lie in $[0,L]$, and their sums satisfy \eqref{eq:seven-local}:
\[
 14\cdot18=20\cdot38-508,\qquad
 7\cdot19+7\cdot20=19\cdot35-392.
\]
The corresponding grandchildren of the branch roots have $38-18=20$,
$35-19=16$, and $35-20=15$ leaf children, respectively, as specified above. Thus Theorem~\ref{thm:seven-construction} gives
eigenvalue support $\{0,\pm\sqrt2,\pm3\sqrt2,\pm5\sqrt2\}$ and
PST. In particular, \eqref{eq:seven-plus} becomes
\[
 g_+(t)=\frac{t(t^4-66t^2+812)}{(t^2-2)(t^2-18)(t^2-50)}.
\]
When $t^2=2,18,50$, the factor $t^4-66t^2+812$ in the numerator takes the nonzero values $684,-52,12$, respectively. Thus all six denominator roots are poles of $g_+$, confirming the stated support.
The minimum transfer time is $\pi/\sqrt2$ by
Theorem~\ref{thm:geometry}\textup{(iii)}.
Counting each level gives
\[
 \begin{split}
 |V(B_c)|&=1+6+14\bigl(1+18(1+20)\bigr)=5313,\\
 |V(B_d)|&=1+5+7\bigl(1+19(1+16)\bigr)
                  +7\bigl(1+20(1+15)\bigr)=4521.
 \end{split}
\]
Therefore $|V(G)|=4+8\cdot5313+12\cdot4521=96760$.
Every off-cycle vertex has one parent and at most $20$ children;
the cycle roots have degrees $10,14$. Vertices with $20$ children
have degree $21$, so $\Delta(G)=21$.
By Proposition~\ref{prop:Hpr}, the seven-element support shows that $G$ is not isomorphic to any $H_{p,r}$.
\end{example}

\subsection{Minimum order for seven-element supports}
We now minimize the order for prescribed odd integers $3\leq u<v$.
The branching rules reduce this problem to minimizing each branch order and then choosing the branch-root degrees by an integer recurrence.
For positive integers $s,\eta$ and an integer $k$ with $2\leq k<s$, put
$L=s-k$ and $T=kL-\eta$. If $T<0$, set $c_{s,\eta}(k)=\infty$;
otherwise, with $\rho$ the remainder of $T$ modulo $L$, set
\begin{equation}\label{eq:branch-cost}
 c_{s,\eta}(k)=1+k+T+\rho(L-\rho).
\end{equation}
All indices in the following recurrence are integers. Define $F_{s,\eta}(0,0)=0$, with $F_{s,\eta}(0,K)=\infty$ for $K\ne0$,
and, for $j\geq1$, define
\begin{equation}\label{eq:order-recursion}
 F_{s,\eta}(j,K)=\min_{2\leq k<s}
 \bigl(c_{s,\eta}(k)+F_{s,\eta}(j-1,K-k)\bigr).
\end{equation}
Set $F_{s,\eta}(j,K)=\infty$ for $K<0$, with the convention that the minimum over an empty set is $\infty$. Here $c_{s,\eta}(k)$ is the minimum order of a branch with root degree $k$; $F_{s,\eta}(j,K)$ is the minimum total order of $j$ such branches whose root degrees sum to $K$.

\begin{proposition}\label{prop:minimum-seven}
Let $3\leq u<v$ be odd integers, let $F_{s,\eta}$ be defined by
\eqref{eq:order-recursion}, and use the parameters of
Section~\ref{subsec:seven-explicit}. Then:
\begin{enumerate}[label=\textup{(\roman*)}]
\item For positive integers $s,\eta$ and an integer $k$ with $2\leq k<s$,
a branch with the local branching pattern in \eqref{eq:seven-local} exists
if and only if $T=k(s-k)-\eta\geq0$. In that case, its minimum
order is $c_{s,\eta}(k)$, and
the minimizing branch is unique up to rooted isomorphism.
More precisely, put $L=s-k$ and write $T=mL+\rho$, where
$0\leq\rho<L$. The minimizing child parameters consist of $m$
entries equal to $L$, one entry equal to $\rho$ if $\rho>0$, and
zeros in the remaining positions.
\item The minimum order of a connected unicyclic graph admitting PST
with eigenvalue support $\{0,\pm\sqrt2,\pm u\sqrt2,\pm v\sqrt2\}$ is
\begin{equation}\label{eq:minimum-seven}
 4+\min\bigl\{F_{s_c,\eta_c}(p_c,K_c)
             +F_{s_d,\eta_d}(p_d,K_d)\bigr\},
\end{equation}
where $K_i=p_i s_i+\eta_i-\beta$ for $i\in\{c,d\}$ and the minimum
ranges over nonnegative integers $p_c,p_d$ and positive integers
$\eta_c,\eta_d$ satisfying $R=p_c+p_d$ and \eqref{eq:seven-global}.
Every minimizing graph, up to isomorphism, is recovered by backtracking
all minimizing choices in \eqref{eq:order-recursion} and choosing
the branches specified in \textup{(i)}.
\end{enumerate}
\end{proposition}
\begin{proof}
\textup{(i)} A branch with root degree $k$ has $k$ child parameters
$0\leq t_\ell\leq L$ with $\sum t_\ell=T$, so $T\geq0$ is necessary.
Conversely, suppose $T\geq0$ and write $T=mL+\rho$ with
$0\leq\rho<L$. Since $\eta>0$, we have $T<kL$.
The vector with $m$ entries equal to $L$, one entry equal to $\rho$
if $\rho>0$, and zeros elsewhere therefore fits in $k$ positions
and is feasible. This proves existence.

The order of a branch is $1+k+\sum t_\ell(1+L-t_\ell)$.
For fixed sum, minimizing the order is equivalent to maximizing
$\sum t_\ell^2$. If two entries satisfy $0<a\leq b<L$, replacing
them by $a-1,b+1$ preserves the constraints and increases the sum
of their squares by
\[
 (a-1)^2+(b+1)^2-a^2-b^2=2(b-a+1)>0.
\]
Thus a maximizing vector has at most one entry strictly between $0$
and $L$. The fixed sum $T=mL+\rho$ then makes every maximizing
vector a permutation of the feasible vector just constructed.
Its branch order is
\[
 1+k+mL+\rho(1+L-\rho)=1+k+T+\rho(L-\rho),
\]
which proves \eqref{eq:branch-cost}.
Moreover, its child parameters are unique up to permutation, and
\eqref{eq:seven-local} fixes all their descendants. Hence the minimizing
branch for a feasible root degree $k$ is unique up to rooted isomorphism.
This proves \textup{(i)}.

\textup{(ii)} For fixed $s,\eta$, the recurrence \eqref{eq:order-recursion}
minimizes over the root degree $k$ of one branch and the orders of the
remaining $j-1$ branches, whose root degrees sum to $K-k$. Induction on $j$ shows that
$F_{s,\eta}(j,K)$ is the minimum total order of $j$ branches whose root degrees sum to $K$.
Theorem~\ref{thm:seven-construction}\textup{(i)} shows that these choices
exhaust all possible realizations, and its part \textup{(ii)} bounds the parameter choices
in the outer minimum. Corollary~\ref{cor:three-frequency-realization}
ensures that, for each $u,v$, at least one admissible choice gives a finite value in \eqref{eq:minimum-seven}.
When $p_i=0$, a finite recurrence value requires $K_i=0$, hence
$\eta_i=\beta$; the attachment is then a singleton.
Adding the four cycle vertices proves \eqref{eq:minimum-seven}.

Every minimizing graph must use a minimum-order branch at each
chosen root degree: otherwise replacing that branch by the one in
\textup{(i)} preserves all parameter constraints and decreases the
order. Its root-degree choices must likewise attain each minimum in the
recurrence and the outer minimum. Conversely, any such sequence of
minimizing choices gives a graph of the displayed order.
This proves that backtracking recovers all minimizing graphs.
\end{proof}

\begin{theorem}\label{thm:minimum-seven}
Let $G$ be a connected unicyclic graph admitting PST between $a,b$,
and suppose $|\Phi_a|=7$. Then $|V(G)|\geq218$,
with equality if and only if $G$ is isomorphic to the graph in
Example~\ref{ex:218}.
\end{theorem}
\begin{proof}
By Theorem~\ref{thm:criterion},
$\Phi_a=\{0,\pm\sqrt2,\pm u\sqrt2,\pm v\sqrt2\}$ for odd integers $3\leq u<v$.
\proofheading{Step 1: bound all parameters.}
It suffices to test $n=|V(G)|\leq218$. A unicyclic graph has $n$ edges,
so $\tr A^2=2n$. The six nonzero supported eigenvalues alone contribute
$4(1+u^2+v^2)$ to this trace. Hence
$n\geq2(1+u^2+v^2)$, or $u^2+v^2\leq108$.
If $u\geq7$, then $v\geq9$ and $u^2+v^2\geq130$; for $u=3$ the
possible $v$ are $5,7,9$, and for $u=5$ they are $7,9$.
These are exactly the five pairs in Table~\ref{tab:minimum-seven}.
Each branch has at least three vertices, so $4+3R\leq218$ gives
$R\leq71$. After interchanging $c,d$ if necessary, we therefore restrict to
\[
 x-3\leq R\leq\min\{y-5,71\},\quad
 0\leq p_c\leq\lfloor R/2\rfloor,\quad 1\leq\eta_c<xy,
\]
with $p_d=R-p_c$ and $\eta_d=xy-\eta_c$.
Since $c_{s,\eta}(k)\geq1+k$, every such graph has at least
$4+R+K_c+K_d$ vertices, counting the cycle, branch roots, and their
children. To expand this bound, use
\[
 K_c+K_d=R\sigma-(p_c^2+p_d^2)+xy-2\beta,
 \qquad p_c^2+p_d^2=R^2-2p_cp_d.
\]
Substitution of $\beta=(x-2)(y-2)+4+2R$ gives the expression below.
\proofheading{Step 2: perform the finite exclusion.}
We discard choices for which
\begin{equation}\label{eq:seven-preliminary}
 4+R+K_c+K_d
 =R(\sigma-3)-R^2+2p_cp_d-xy+4x+4y-12>218.
\end{equation}
Also, $p_c=0$ forces $K_c=0$ by \eqref{eq:seven-budget}, hence
$\eta_c=\beta$. Appendix~\ref{app:seven-certificate} excludes two
nonsingleton attachments by order bounds and tests for square discriminants, then checks
the possible branch-root degrees for a singleton attachment. These leave only
$(u,v,R,p_c,\eta_c)=(3,5,45,0,862)$, as recorded in
Table~\ref{tab:minimum-seven}. The same exclusion fixes the branch-root
degrees: $43$ are $2$, one is $19$, and one is $16$.
Their minimum orders are $3,20,65$, respectively, and sum to
$43\cdot3+20+65=214$. This equals the upper bound on the total order of the branches,
so $n=218$ and every branch attains its individual minimum.

\proofheading{Step 3: determine the equality graph.}
The branches with root degrees $2$ and $19$ are stars. For the
branch with root degree $16$, $s=21$, $\eta=38$, and hence $L=5$ and
$T=42=8\cdot5+2$. Proposition~\ref{prop:minimum-seven}\textup{(i)}
gives, up to permutation,
eight child parameters equal to $5$, one equal to $2$, and seven
equal to $0$. Theorem~\ref{thm:seven-construction} then fixes every
descendant: the first eight children each have five leaf children,
the next child has two children each with three leaf children,
and the remaining seven are leaves.
Thus both attachments, and hence $G$, are determined up to isomorphism.
They are exactly those in Example~\ref{ex:218}, which verifies the
prescribed support and attains order $218$.
\end{proof}

\begin{table}[htbp]
\centering
\caption{Under $|V(G)|\leq218$, the second column lists the values of $R$ not excluded by the preliminary bound \eqref{eq:seven-preliminary}; the last two columns give the outcome of Appendix~\ref{app:seven-certificate}, up to interchanging $c,d$.}
\label{tab:minimum-seven}
\begin{tabular}{@{}cccl@{}}
\toprule
$(u,v)$ & Surviving $R$ & $(p_c,\eta_c)$ retained & Order retained\\
\midrule
$(3,5)$ & $15$--$19$, $44$, $45$ & $(0,862)$ at $R=45$ & $218$\\
$(3,7)$ & $15$, $16$ & none & none\\
$(3,9)$ & $15$ & none & none\\
$(5,7)$ & $47$ & none & none\\
$(5,9)$ & none & none & none\\
\bottomrule
\end{tabular}
\end{table}

\begin{proof}[Proof of Theorem~\ref{thm:seven-summary}]
Part \textup{(i)} is Corollary~\ref{cor:three-frequency-realization},
and part \textup{(ii)} follows from
Theorem~\ref{thm:seven-construction}\textup{(ii)}.
Theorem~\ref{thm:minimum-seven} proves \textup{(iii)}, with
$(u,v,R)=(3,5,45)$ identified in Step 2 of its proof.
\end{proof}

\begin{example}\label{ex:218}
There is a $218$-vertex connected unicyclic graph admitting PST between
$a,b$, with $\Phi_a=\{0,\pm\sqrt2,\pm3\sqrt2,\pm5\sqrt2\}$ and an
attachment of depth four.
Start with the cycle $a,c,b,d,a$ and take $T_c$ to be a singleton. At $d$, attach
$43$ copies of $K_{1,2}$ and one copy of $K_{1,19}$ through their centers,
together with the following rooted tree $B$. Its root has $16$ children:
eight have five leaf children each, one has two children each having
three leaf children, and seven are leaves. Thus $B$ has $65$ vertices,
and the resulting graph has maximum degree $47$ and order
\[
 4+43\cdot3+20+65=218.
\]

Figure~\ref{fig:minimum-seven} shows the graph and the branch $B$.

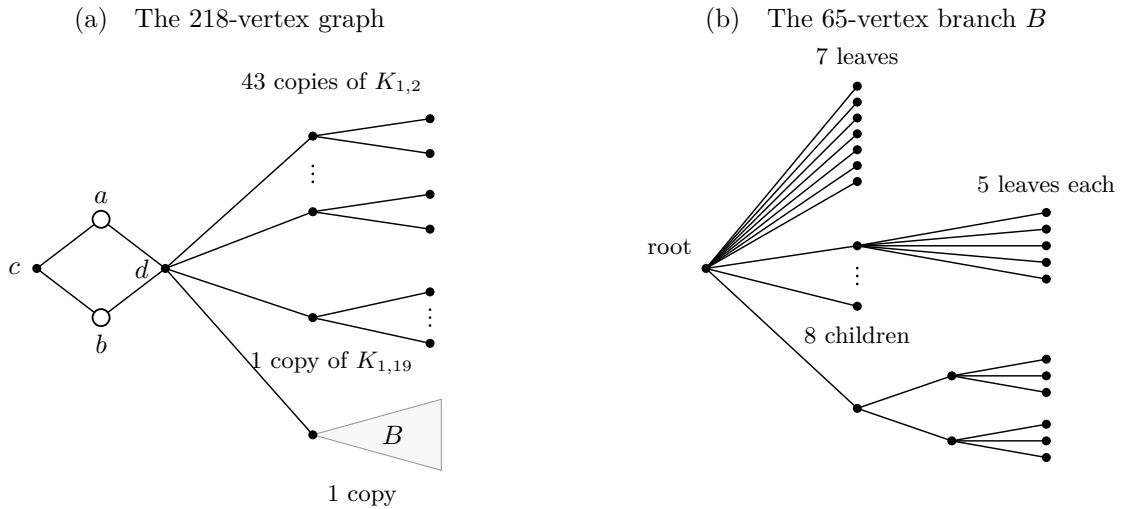
\begin{figure}[htbp]
\centering
\begin{tikzpicture}[pst graph]
\begin{scope}
  \node at (1.7,3.28) {\textup{(a)}\quad The $218$-vertex graph};
  \node[pst terminal] (a) at (0,.65) {};
  \node[pst terminal] (b) at (0,-.65) {};
  \node[pst vertex] (c) at (-.85,0) {};
  \node[pst vertex] (d) at (.85,0) {};
  \draw (a)--(c)--(b)--(d)--(a);
  \node[above=4pt] at (a) {$a$};
  \node[below=4pt] at (b) {$b$};
  \node[left=4pt] at (c) {$c$};
  \node at (.54,0) {$d$};
  \foreach \y/\name in {1.75/su,.75/sl}{
    \node[pst vertex] (\name) at (2.8,\y) {};
    \draw (d)--(\name);
    \foreach \dy in {-.23,.23}{
      \node[pst vertex] (leaf) at (4.35,{\y+\dy}) {};
      \draw (\name)--(leaf);}}
  \pstvdots{2.8}{1.25}
  \node[pst note] at (3.05,2.45) {$43$ copies of $K_{1,2}$};
  \node[pst vertex] (s) at (2.8,-.65) {};
  \draw (d)--(s);
  \foreach \dy in {-.34,.34}{
    \node[pst vertex] (leaf) at (4.35,{-.65+\dy}) {};
    \draw (s)--(leaf);}
  \pstvdots{4.35}{-.65}
  \node[pst note] at (3.05,-1.25) {$1$ copy of $K_{1,19}$};
  \path[pst region] (2.8,-2.2)--(4.5,-1.73)--(4.5,-2.67)--cycle;
  \node[pst vertex] (br) at (2.8,-2.2) {};
  \draw (d)--(br);
  \node at (3.86,-2.2) {$B$};
  \node[pst note] at (3.45,-3.01) {$1$ copy};
\end{scope}
\begin{scope}[xshift=8cm]
  \node at (2.25,3.28) {\textup{(b)}\quad The $65$-vertex branch $B$};
  \node[pst vertex] (r) at (0,0) {};
  \node[above left=3pt,pst note] at (r) {root};
  \foreach \y in {1.15,1.36,1.57,1.78,1.99,2.20,2.41}{
    \node[pst vertex] (leaf) at (2,\y) {};
    \draw (r)--(leaf);}
  \node[pst note] at (2,2.81) {$7$ leaves};
  \node[pst vertex] (u) at (2,.3) {};
  \node[pst vertex] (u2) at (2,-.5) {};
  \draw (r)--(u) (r)--(u2);
  \pstvdots{2}{-.1}
  \node[pst note] at (2,-.88) {$8$ children};
  \foreach \y in {-.14,.08,.3,.52,.74}{
    \node[pst vertex] (leaf) at (4.5,\y) {};
    \draw (u)--(leaf);}
  \node[pst note] at (4.5,1.12) {$5$ leaves each};
  \node[pst vertex] (v) at (2,-1.85) {};
  \draw (r)--(v);
  \foreach \y/\name in {-1.42/w1,-2.28/w2}{
    \node[pst vertex] (\name) at (3.25,\y) {};
    \draw (v)--(\name);
    \foreach \dy in {-.22,0,.22}{
      \node[pst vertex] (leaf) at (4.5,{\y+\dy}) {};
      \draw (\name)--(leaf);}}
\end{scope}
\end{tikzpicture}
\caption{The graph in Example~\ref{ex:218} and its branch $B$. Open circles mark $a,b$; dots indicate omitted identical branches or children. The bottom child of the root in (b) and its descendants are drawn in full. The shaded region represents $B$; its boundary is not a graph edge.}
\label{fig:minimum-seven}
\end{figure}

For a direct verification using Theorem~\ref{thm:seven-construction}, take
\[
 (u,v)=(3,5),\quad (p_c,p_d)=(0,45),\quad
 \beta=862,\quad (\eta_c,\eta_d)=(862,38).
\]
Here $14<45<46$, $\eta_c+\eta_d=900$, and $s_d=66-45=21$.
For $p_c=0$, \eqref{eq:seven-budget} reduces to $\eta_c=\beta$; at $d$, it holds because
\[
 43\cdot2+19+16=121=45\cdot21+38-862.
\]
For the star branches, $k=2$ or $19$ gives $k(21-k)=38$, so all
child parameters are zero. For $B$, $k=16$ and $L=5$;
its child parameters are eight copies of $5$, one $2$, and seven zeros.
They lie in $[0,5]$ and sum to $42=16\cdot5-38$.
A child with parameter $5$ has five leaf children; a child with parameter $2$ has two children, each with three leaf children; a child with parameter $0$ is a leaf.
Thus all construction conditions hold. Formula~\eqref{eq:seven-plus} gives
\[
 g_+(t)=\frac{t(t^4-66t^2+862)}{(t^2-2)(t^2-18)(t^2-50)}.
\]
The factor $t^4-66t^2+862$ in the numerator takes the nonzero values $734,-2,62$ when $t^2=2,18,50$, respectively. Thus all six roots of the denominator are poles of $g_+$, and the stated support occurs,
with PST at minimum time $\pi/\sqrt2$.
\end{example}

\begin{remark}\label{rem:local-optimization}
For a branch in Theorem~\ref{thm:seven-construction}, put
$k=k_{ij}$, $L=s_i-k$, $t_\ell=t_{ij\ell}$, and
$T=\sum_\ell t_\ell=kL-\eta_i$. Its order is
$1+k+(L+1)T-\sum_\ell t_\ell^2$, so minimizing order maximizes the
sum of squares. Proposition~\ref{prop:minimum-seven}\textup{(i)} gives
the minimizing parameters: all are $0$ or $L$, except possibly one equal to the remainder of $T$ modulo $L$.

The degree constraint has a different form. In $G$, the branch root has degree $k+1$. A positive parameter $t_\ell$ gives
degrees $1+t_\ell$ and $1+L-t_\ell$ in the next two levels; a zero
parameter gives only a leaf. For $T>0$ and an integer $D\geq k+1$,
there is a choice of $t_1,\ldots,t_k$ with sum $T$ for which every
vertex of the branch has degree at most $D$ in $G$ if and only if
the number $m$ of positive parameters, with $1\leq m\leq k$,
can be chosen so that
\begin{equation}\label{eq:local-degree-feasibility}
 m\max\{1,L-D+1\}\leq T\leq m\min\{L,D-1\}.
\end{equation}
Indeed, each of the $m$ positive parameters must lie between
$\max\{1,L-D+1\}$ and $\min\{L,D-1\}$.
Conversely, if the inequalities hold, take
these parameters to be $\lfloor T/m\rfloor$ or $\lceil T/m\rceil$,
with the required sum, and set the other $k-m$ parameters to zero.
For fixed $m$, every such partition has largest positive parameter
at least $\lceil T/m\rceil$ and smallest positive parameter at most
$\lfloor T/m\rfloor$. Thus the balanced choice minimizes both the
largest positive parameter and $L$ minus the smallest one. The number
$m$ must also be chosen, and the degrees $k+1$ and $p_i+2$ of the
branch and attachment roots still constrain the maximum degree of
$G$. If $T=0$, the branch is a star and its vertices have
maximum degree $k+1$ in $G$. Thus the order and degree
optimizations in Examples~\ref{ex:218} and~\ref{ex:degree21} impose
different conditions on the same integers $t_1,\ldots,t_k$ satisfying \eqref{eq:seven-local}.
\end{remark}

The construction in Theorem~\ref{thm:recursive-construction} determines whether each prescribed eigenvalue support is realizable. Proposition~\ref{prop:Hpr} and
Corollary~\ref{cor:three-frequency-realization} give realizations for
every $S$ with $|S|\in\{2,3\}$. It remains open whether every larger set $S$ is realizable.

\begin{question}
Does every finite set $S$ of positive odd integers with $1\in S$ and
$|S|\geq4$ occur in a connected unicyclic graph admitting PST between
vertices $a,b$ with
\[
 \Phi_a=\{0\}\cup\{\pm q\sqrt2:q\in S\}?
\]
\end{question}

Finite termination alone does not guarantee a nonempty output for every
$S$. For prescribed $S$ and $R=\deg_{T_c}(c)+\deg_{T_d}(d)$, write
$M_S(t)=t^{2s}-e_1t^{2s-2}+e_2t^{2s-4}+\cdots$, where $s=|S|$.
Multiplying by the resolvent expansion
$g_+(t)=t^{-1}+4t^{-3}+(16+2R)t^{-5}+O(t^{-7})$ gives
\[
 \widetilde N(t)=t^{2s-1}+(4-e_1)t^{2s-3}
       +(16+2R-4e_1+e_2)t^{2s-5}+\cdots.
\]
When $s=3$, these are all the coefficients of the odd numerator,
giving the quartic $N(t)=\widetilde N(t)/t$ in \eqref{eq:seven-plus}.
For $s\geq4$, the two moment identities leave $s-3$ further integer
coefficients to determine; the higher moments
$\langle w_+,A^{2j}w_+\rangle$, $3\leq j\leq s-1$, determine them
successively. These coefficients cannot be chosen independently:
the roots of $\widetilde N$ must strictly interlace those of $M_S$, and
the polynomial identities \eqref{eq:initial-root-splitting}
and \eqref{eq:recursive-splitting} must admit choices of root resolvents
with positive residues and the required integer coefficients. To establish existence for every $S$, one must find choices satisfying all these conditions.

\appendix
\numberwithin{table}{section}
\section{Details of the finite minimum-order calculation}\label{app:seven-certificate}
This appendix supplies the parameter exclusions and intermediate arithmetic for
Theorem~\ref{thm:minimum-seven}. Throughout, the total order of the
branches is at most $214$, and all indices and parameters are integers.
We use the branch orders $c_{s,\eta}(k)$ in \eqref{eq:branch-cost}; this value is $\infty$ when $k(s-k)<\eta$. Throughout this appendix, a branch-root degree is measured within the branch, excluding its edge to $c$ or $d$.

\subsection{Excluding two nonsingleton attachments}
Assume $1\leq p_c\leq p_d$, and write $p=p_c$, $r=p_d$.
Let $C_i$ be the total order of the branches at $i\in\{c,d\}$.
Thus $C_c+C_d\leq214$.

\proofheading{Step 1: restrict $R$ and $p$.}
Write the preliminary lower bound in \eqref{eq:seven-preliminary} as
\[
 B(R,p)=R(\sigma-3)-R^2+2p(R-p)-xy+4x+4y-12.
\]
For $1\leq p\leq\lfloor R/2\rfloor$, the product $p(R-p)$ is
increasing in $p$, so
\[
 B(R,p)\geq B(R,1)
 =-R^2+(\sigma-1)R-xy+4x+4y-14.
\]
The last expression is concave in $R$; on an integer interval it is
therefore bounded below by the smaller of its endpoint values.
For $(u,v)=(3,9),(5,7),(5,9)$, its values at the endpoints of
$x-3\leq R\leq\min\{y-5,71\}$ are, respectively,
\[
 (220,5316),\qquad (284,932),\qquad (348,2532).
\]
All exceed $218$, excluding these three pairs. For $(3,5)$, its
endpoint values on $19\leq R\leq45$ are $232,258$; for $(3,7)$,
its endpoint values on $16\leq R\leq71$ are $238,1668$.
These intervals are excluded as well. At each remaining $R$,
monotonicity in $p$ gives exactly the ranges in
Table~\ref{tab:positive-parameters}.

\begin{table}[htbp]
\centering
\caption{Complete parameter ranges with $p,r>0$ allowed by
$B(R,p)\leq218$. In every row, $r=R-p$.}
\label{tab:positive-parameters}
\begin{tabular}{@{}crcl@{}}
\toprule
$(u,v)$ & $R$ & Allowed $p$ & $B(R,p)$\\
\midrule
$(3,5)$ & $15$ & $1\leq p\leq7$ & $80+2p(15-p)$\\
 & $16$ & $1\leq p\leq4$ & $112+2p(16-p)$\\
 & $17$ & $1\leq p\leq2$ & $142+2p(17-p)$\\
 & $18$ & $p=1$ & $170+2p(18-p)$\\
$(3,7)$ & $15$ & $1\leq p\leq4$ & $128+2p(15-p)$\\
\bottomrule
\end{tabular}
\end{table}

\proofheading{Step 2: force a star branch at $d$.}
Lemma~\ref{lem:extremal-root-branches}\textup{(iii)} gives root degree
at least three in every branch, so every branch has at least four vertices.
A stronger bound holds for a branch that is not a star.
Put $s=s_i$, let $k$ be its root degree, and set $L=s-k$.
Its order, by \eqref{eq:seven-local}, is
\[
 1+k+\sum_{\ell=1}^k t_\ell(1+L-t_\ell).
\]
If the branch is not a star, some $t_\ell$ satisfies $1\leq t_\ell\leq L$.
For this term,
\[
 t_\ell(1+L-t_\ell)-L=(t_\ell-1)(L-t_\ell)\geq0.
\]
The other terms are nonnegative. Every such branch therefore has
at least $1+k+L=s_i+1$ vertices. The ranges in
Table~\ref{tab:positive-parameters} give $r\geq8$ and $s_d\geq49$.
If no branch at $d$ were a star, then
$C_d\geq r(s_d+1)\geq400>214$. Thus a star branch occurs at $d$.

We next show that both roots of $X^2-s_iX+\eta_i$ are distinct
and greater than two. Put
$N_0(X)=X^2-\sigma X+\beta$, so $N(t)=N_0(t^2)$.
For $(u,v)=(3,5)$ and $15\leq R\leq18$, one quarter of the
discriminant of $N_0$ is $317-2R$, strictly between $16^2$ and $17^2$.
For $(3,7)$ and $R=15$, it is $1679$, strictly between $40^2$ and
$41^2$. Thus $N_0$ is irreducible over $\mathbb Q$ in every remaining
case. Since $p_i>0$, the monic quadratic $X^2-s_iX+\eta_i$ differs from the irreducible polynomial $N_0$ and is therefore coprime to it. Thus the corresponding factors in \eqref{eq:seven-root} do not cancel; the positive constant term $\eta_i$ also prevents cancellation of the denominator factor $t$. Consequently $g_i$ has all five poles and
four simple zeros. Those zeros are exactly the poles of
$\sum_j h_{ij}=t-1/g_i$. Positive residues prevent cancellation
between the branch resolvents. By Lemma~\ref{lem:extremal-root-branches}\textup{(iii)}, every eigenvalue in the support at a branch root lies outside $[-\sqrt2,\sqrt2]$.
Hence both roots of $X^2-s_iX+\eta_i$ are distinct and greater than two.

A star with $k$ leaves, rooted at its center, has eigenvalue support $\{\pm\sqrt{k}\}$ at the root.
Thus $k$ is a root of $X^2-s_dX+\eta_d$ for the star branch at $d$;
the integral sum of the two roots makes the other root an integer too. Denote the smaller root by $h$. Then
\begin{equation}\label{eq:appendix-star-parameter}
 \eta_d=h(s_d-h),\qquad
 \eta_c=xy-h(s_d-h),\qquad 3\leq h<s_d/2.
\end{equation}
Each of the $p$ branch roots at $c$ has degree at least three, so
$K_c\geq3p$. Substituting
$K_c=p s_c+\eta_c-\beta=p s_c+xy-h(s_d-h)-\beta$
and rearranging gives
\begin{equation}\label{eq:appendix-star-bound}
 h(s_d-h)\leq U,
 \qquad U=xy-\beta+p(s_c-3).
\end{equation}
The left side is strictly increasing for $h<s_d/2$.
This inequality restricts $h$ to the short ranges used below,
without testing every possible value of $\eta_c$.

\proofheading{Step 3: exclude $p\geq3$, and $p\geq2$ for $(u,v)=(3,7)$.}
If no branch at $c$ were a star, then
\[
 C_c+C_d\geq p(s_c+1)+4r=p(\sigma-p-3)+4R.
\]
The expression is increasing in both $p$ and $R$ over the ranges
in Table~\ref{tab:positive-parameters}. For $(3,5)$ and $p\geq3$,
its minimum is $3(66-3-3)+4\cdot15=240$; for $(3,7)$ and
$p\geq2$, its minimum is $2(114-2-3)+4\cdot15=278$.
Thus a star must also occur at
$c$ in these cases. Its number of leaves is an integer root of
$X^2-s_cX+\eta_c$, so
\[
 D_c=s_c^2-4\eta_c=s_c^2-4xy+4h(s_d-h)
\]
must be a square. In particular, its residue modulo $16$ must
belong to $\{0,1,4,9\}$.

Table~\ref{tab:positive-certificate} gives every remaining possibility.
Its column $h_{\max}$ means that
\[
 h_{\max}(s_d-h_{\max})\leq U
 <(h_{\max}+1)(s_d-h_{\max}-1),
\]
so \eqref{eq:appendix-star-bound} gives $3\leq h\leq h_{\max}$.
The next column retains exactly the values satisfying the necessary
modulo-$16$ condition. For odd $s_c$, the last column gives $D_c$;
for even $s_c$, it gives $D_c/4$, which must also be a square.
Each entry is written as $m^2+\varepsilon$ with $0<\varepsilon<2m+1$ and is therefore
strictly between consecutive squares. This excludes every row.
For example, in the first row $(u,v,R,p)=(3,5,15,3)$,
$s_c=63$, $s_d=54$, and $U=278$.
The inequalities $5\cdot49\leq278<6\cdot48$ give $h_{\max}=5$.
For $h=3,4,5$, the values of $D_c$ are $981,1169,1349$,
with residues $5,1,5$ modulo $16$. Only $h=4$ passes that test,
and $1169=34^2+13$ is not a square.

\begin{table}[htbp]
\centering
\caption{Complete square tests when a star branch is forced at both
attachment roots. Values in the last column follow the order of $h$
in the preceding column.}
\label{tab:positive-certificate}
\small
\setlength{\tabcolsep}{4pt}
\begin{tabular}{@{}rrrll@{}}
\toprule
$R$ & $p$ & $h_{\max}$ & $h$ after the modulo-$16$ test & $D_c$ or $D_c/4$\\
\midrule
\multicolumn{5}{@{}l}{$(u,v)=(3,5)$}\\
$15$ & $3$ & $5$ & $4$ & $34^2+13$\\
 & $4$ & $6$ & $3,4$ & $14^2+21,\ 16^2+9$\\
 & $5$ & $8$ & $4,6,8$ & $30^2+53,\ 36^2+25,\ 40^2+57$\\
 & $6$ & $9$ & $4,5,8,9$ & $14^2+16,\ 16^2+4,\ 19^2+31,\ 20^2+32$\\
 & $7$ & $10$ & $4,6,8,10$ & $27^2+16,\ 33^2+40,\ 38^2+37,\ 42^2+37$\\
$16$ & $3$ & $5$ & $3,4,5$ & $31^2+8,\ 33^2+64,\ 36^2+33$\\
 & $4$ & $7$ & $4,6$ & $16^2+5,\ 18^2+25$\\
\midrule
\multicolumn{5}{@{}l}{$(u,v)=(3,7)$}\\
$15$ & $2$ & $4$ & $4$ & $41^2+79$\\
 & $3$ & $5$ & $4$ & $82^2+109$\\
 & $4$ & $6$ & $3,4$ & $39^2+40,\ 40^2+57$\\
\bottomrule
\end{tabular}
\end{table}
\FloatBarrier

\proofheading{Step 4: exclude the small values of $p$.}
The remaining cases are $p=1$ for $(3,5)$ with $15\leq R\leq18$
and for $(3,7)$ with $R=15$, together with $p=2$ for $(3,5)$ with
$15\leq R\leq17$.
For $p=1$, \eqref{eq:appendix-star-bound} forces $h=3$:
for $(3,5)$, $U\leq160$ whereas $4(s_d-4)\geq180$;
for $(3,7)$, $U=304<4(100-4)=384$.
For $p=2$, $s_c=64$, $s_d=68-R$, and $U=250-2R$.
Here $h=3,4$ are the only possibilities, since
$5(s_d-5)=5(63-R)>250-2R$ for $15\leq R\leq17$.

\proofheading{Case 1: $p=1$.}
There is one branch at $c$, of root degree $K_c=s_c+\eta_c-\beta$.
Equations~\eqref{eq:appendix-star-parameter} and~\eqref{eq:branch-cost}
give the following values. In the first row, $s_c=65$, $K_c=16$,
and $\eta_c=753$, so $L=49$ and $T=\rho=31$.
Thus $c_{65,753}(16)=17+31+31\cdot18=606$; the other entries
follow by the same substitution.
\par\nopagebreak[4]\smallskip
{\centering
\begin{tabular}{@{}crrrr@{}}
\toprule
$(u,v)$ & $R$ & $\eta_c$ & $K_c$ & $c_{s_c,\eta_c}(K_c)$\\
\midrule
$(3,5)$ & $15$ & $753$ & $16$ & $606$\\
 & $16$ & $756$ & $17$ & $510$\\
 & $17$ & $759$ & $18$ & $386$\\
 & $18$ & $762$ & $19$ & $652$\\
$(3,7)$ & $15$ & $1473$ & $16$ & $1518$\\
\bottomrule
\end{tabular}\par}
\smallskip
Every minimum branch order exceeds $214$, excluding these five choices.

\proofheading{Case 2: $p=2$ and $h=3$.}
Now $s_c=64$, $\eta_c=705+3R$, and $K_c=61+R$.
For each of the two branch-root degrees $k$, feasibility requires
$k(64-k)\geq\eta_c$. Since $750\leq\eta_c\leq756$, this gives
$16\leq k\leq48$.
Write $K=K_c$ and let the two degrees be $k_1,k_2$, with $k_1+k_2=K$.
The feasible interval for $k_1$ is $K-48\leq k_1\leq48$.
The quadratic $k_1^2+k_2^2$ is maximized there at
$\{k_1,k_2\}=\{48,K-48\}$. Dropping the nonnegative remainder
term in \eqref{eq:branch-cost} therefore gives
\[
\begin{split}
 C_c&\geq2+65K-(k_1^2+k_2^2)-2\eta_c\\
 &\geq2+65K-48^2-(K-48)^2-2\eta_c\\
 &=-R^2+33R+84\geq354>214
 \qquad(15\leq R\leq17).
\end{split}
\]

\begin{samepage}
\proofheading{Case 3: $p=2$ and $h=4$.}
Here $s_c=64$, $\eta_c=644+4R$, and $K_c=2R$.
Feasibility gives $k\geq15$, since
$14(64-14)=700<\eta_c$. The smaller of the two branch-root degrees
is at most $K_c/2=R$. Thus it lies in $15\leq k\leq R$.
The corresponding values of $c_{64,644+4R}(k)$ are displayed below:
\begin{center}
\begin{tabular}{@{}crrr@{}}
\toprule
$R\mathbin{\backslash}k$ & $15$ & $16$ & $17$\\
\midrule
$15$ & $605$ & & \\
$16$ & $637$ & $509$ & \\
$17$ & $637$ & $393$ & $385$\\
\bottomrule
\end{tabular}
\end{center}
\end{samepage}
Each entry is $c_{64,644+4R}(k)>214$, excluding the last three
choices. This completes the exclusion of two nonsingleton attachments.

\subsection{A singleton attachment}
Now $p_c=0$, $p_d=R$, and $\eta_c=\beta$.
Put $s=s_d=\sigma-R$ and $\eta=\eta_d=xy-\beta=2(s-2)$.
For root degree $k=2$, \eqref{eq:seven-local} forces every child
parameter to be zero. Thus every such branch is $K_{1,2}$ rooted at
its center, and $c_{s,\eta}(2)=3$. Put
\[
 D=K_d-2R=(R-x+4)(y-4-R),\qquad E=214-3R.
\]
The factorization follows by substituting
$K_d=Rs+\eta-\beta$, $s=x+y-2-R$, and $\eta=xy-\beta$.
Here $D$ is the sum of the branch-root degrees minus $2R$, and $E$ is the number
of vertices allowed in addition to the $3R$ vertices in $R$ branches
with root degree two.
Since the other $R-1$ branches each have order at least three, a branch
with root degree $k$ in any candidate must satisfy
$c_{s,\eta}(k)\leq214-3(R-1)=E+3$.
We first derive the complete parameter and degree ranges in
Table~\ref{tab:singleton-certificate}.

\proofheading{Step 1: restrict $R$.}
For $p_c=0$, the preliminary lower bound is
\[
 B_0(R)=-R^2+(\sigma-3)R-xy+4x+4y-12.
\]
This is concave in $R$. The following intervals have both endpoint
values greater than $218$ and are therefore excluded:
\par\nopagebreak[4]\smallskip
{\centering
\begin{tabular}{@{}ccc@{}}
\toprule
$(u,v)$ & Excluded interval & Endpoint values of $B_0$\\
\midrule
$(3,5)$ & $20\leq R\leq43$ & $(220,220)$\\
$(3,7)$ & $17\leq R\leq71$ & $(286,1528)$\\
$(3,9)$ & $16\leq R\leq71$ & $(336,5176)$\\
$(5,7)$ & $48\leq R\leq71$ & $(240,792)$\\
$(5,9)$ & $47\leq R\leq71$ & $(256,2392)$\\
\bottomrule
\end{tabular}\par}
\smallskip
Within $x-3\leq R\leq\min\{y-5,71\}$, the remaining values are
exactly the $11$ rows in Table~\ref{tab:singleton-certificate}.

\proofheading{Step 2: restrict the branch-root degrees.}
Here $\eta=2(s-2)$, so
\[
 T=k(s-k)-\eta=(k-2)(s-k-2).
\]
Thus $2\leq k\leq s-2$. The two possible stars have root degrees $2,s-2$ and orders $3,s-1$, respectively. For a nonstar branch, $3\leq k\leq s-3$.
For $k=3$, we have $L=s-3$ and $T=\rho=s-5$, so
\eqref{eq:branch-cost} gives $c_{s,2(s-2)}(3)=3s-11$. For $4\leq k\leq s-3$, put
$j=k-2$. Dropping the nonnegative remainder term in
\eqref{eq:branch-cost} gives
\begin{equation}\label{eq:singleton-basic-cost}
 c_{s,2(s-2)}(k)\geq3+j(s-3-j)\geq2s-7.
\end{equation}
The last inequality follows by concavity on $2\leq j\leq s-5$;
its endpoint values are $2s-7,2s-7$. Since $3s-11\geq2s-7$,
this bound holds for every nonstar branch. For the rows with
$(u,v)\ne(3,5)$, it exceeds $E+3$, so only stars with root degrees $2$ or $s-2$ are possible. For $(u,v)=(5,7)$, root degree $s-2$ is also excluded because $c_{s,\eta}(s-2)-3=s-4>E$.

For $(3,5)$ and $15\leq R\leq19$, we have $47\leq s\leq51$
and $E+3=3s+19$. We exclude the whole interval $4\leq k\leq s-5$.
If $2\leq j\leq14$, then $2j<L=s-2-j$, and the remainder of
$T=j(L-2)$ modulo $L$ is $L-2j$. Consequently
\begin{equation}\label{eq:singleton-exact-cost}
 c_{s,2(s-2)}(k)=3+j(3s-7-7j).
\end{equation}
This identity holds whenever $2j<L$, including the small values of $s$
considered below. The quadratic $3+j(3s-7-7j)$ is concave in $j$, so its
minimum on $2\leq j\leq14$ is the smaller endpoint value,
$\min\{6s-39,42s-1467\}=6s-39>3s+19$.
For $15\leq j\leq s-7$, \eqref{eq:singleton-basic-cost} gives
the lower bound
$\min\{15s-267,4s-25\}=4s-25>3s+19$.
Only $k=2,3,s-4,s-3,s-2$ remain. The corresponding minimum branch orders are, respectively,
\[
 3,\quad 3s-11,\quad 3s-15+4(s\bmod2),\quad
 2s-7+2\delta_s,\quad s-1,
\]
where $\delta_s=0$ if $s\equiv2\pmod3$ and $\delta_s=1$ otherwise.
For $k=s-4$, the substitution has $L=4$ and $\rho\equiv2s\pmod4$,
giving the term $4(s\bmod2)$.
For $k=s-3$, it has $L=3$ and $\rho\equiv s-5\pmod3$,
giving $2\delta_s$.
Subtracting $3$ from these minimum orders gives the last column in the first five rows of Table~\ref{tab:singleton-certificate}.

For $R=44$, $s=22$ and $E+3=85$.
Equation~\eqref{eq:singleton-exact-cost} gives at least $93$ for
$4\leq k\leq8$, and \eqref{eq:singleton-basic-cost} gives at least
$87$ for $9\leq k\leq14$; both follow by taking endpoint minima
of the corresponding concave quadratics. For $k=15$, the minimum order is $c_{22,40}(15)=91$.
Thus only $k=2,3,16,17,18,19,20$ remain.
For $R=45$, $s=21$ and $E+3=82$.
Equation~\eqref{eq:singleton-exact-cost} gives at least $87$ for
$4\leq k\leq8$. For $9\leq k\leq14$, we have $7\leq L\leq12$.
The values $\rho=0,1,L-1$ would require $L$ to divide
$38,39,37$, respectively. No integer $7\leq L\leq12$ divides any
of these numbers, so $2\leq\rho\leq L-2$.
Hence
\[
 c_{21,38}(k)\geq1+k+T+2(L-2)
 =37+16j-j^2\geq85\qquad(7\leq j\leq12).
\]
Only $k=2,3,15,16,17,18,19$ remain. Substitution in
\eqref{eq:branch-cost} gives the minimum branch orders; subtracting $3$ gives the entries in the last two $(3,5)$ rows.
Thus Table~\ref{tab:singleton-certificate} is complete.

\begin{table}[htbp]
\centering
\caption{Root degrees $k$ satisfying the individual branch-order bound $c_{s,\eta}(k)\leq E+3$, with the values $c_{s,\eta}(k)-3$, when $T_c$ is a singleton. The degree-sum and total-order constraints are checked below.}
\label{tab:singleton-certificate}
\footnotesize
\setlength{\tabcolsep}{4pt}
\begin{tabular}{@{}crrrl@{}}
\toprule
$(u,v)$ & $R$ & $D$ & $E$ & Allowed pairs $(k,c_{s,\eta}(k)-3)$\\
\midrule
$(3,5)$ & $15$ & $31$ & $169$ & $(2,0),(3,139),(47,139),(48,94),(49,47)$\\
 & $16$ & $60$ & $166$ & $(2,0),(3,136),(46,132),(47,90),(48,46)$\\
 & $17$ & $87$ & $163$ & $(2,0),(3,133),(45,133),(46,90),(47,45)$\\
 & $18$ & $112$ & $160$ & $(2,0),(3,130),(44,126),(45,88),(46,44)$\\
 & $19$ & $135$ & $157$ & $(2,0),(3,127),(43,127),(44,84),(45,43)$\\
 & $44$ & $60$ & $82$ & $(2,0),(3,52),(16,78),(17,60),(18,48),(19,36),(20,18)$\\
 & $45$ & $31$ & $79$ & $(2,0),(3,49),(15,73),(16,62),(17,49),(18,34),(19,17)$\\
\midrule
$(3,7)$ & $15$ & $79$ & $169$ & $(2,0),(97,95)$\\
 & $16$ & $156$ & $166$ & $(2,0),(96,94)$\\
$(3,9)$ & $15$ & $143$ & $169$ & $(2,0),(161,159)$\\
$(5,7)$ & $47$ & $47$ & $73$ & $(2,0)$\\
\bottomrule
\end{tabular}
\end{table}

For an allowed root degree $k>2$, let $n_k$ be the number of branches with that root degree. Complete the collection to $R$ branches with copies of $K_{1,2}$ rooted at their centers. Each branch with root degree $k>2$ contributes $k-2$ to $D$ and requires at least $c_{s,\eta}(k)-3$ vertices beyond the three in a copy of $K_{1,2}$. Choosing every branch with minimum order therefore reduces feasibility to
\begin{equation}\label{eq:appendix-excess}
 \sum_{k>2}(k-2)n_k=D,\qquad
 \sum_{k>2}\bigl(c_{s,\eta}(k)-3\bigr)n_k\leq E,\qquad
 \sum_{k>2}n_k\leq R.
\end{equation}
We now examine the rows of Table~\ref{tab:singleton-certificate} in turn.
For $(u,v)=(3,7)$, the only positive value of $k-2$ is $95$ when
$R=15$ and $94$ when $R=16$. The required equations
$95n=79$ and $94n=156$ have no nonnegative integer solutions.
For $(3,9)$, the least positive value of $k-2$ is $159>143=D$.
For $(5,7)$, only $k=2$ is allowed, so the sum in the first equation is zero, whereas $D=47$.
It remains to consider $(u,v)=(3,5)$:
\begin{enumerate}[label=\textup{(\roman*)}]
\item If $R=15$, every allowed root degree other than $2,3$ has
$k-2\geq45>D$. Using only root degrees $2,3$ would require $31>15$
branches with root degree three.
\item If $R=16$, the second inequality in \eqref{eq:appendix-excess}
allows at most one branch with root degree three. Two branches with root degree at least $46$ contribute at least $88>D$ to the sum of $k-2$, so at most one is possible. That sum is then at most $1+46=47<60=D$.
\item If $R=17$, two branches with root degree three would require at least
$2\cdot133>163=E$ additional vertices, so at most one can occur.
At most two branches can have root degree at least $45$, since three
would contribute at least $3\cdot43=129>D$ to the sum of $k-2$.
The only sums of positive contributions $k-2$ equal to $D=87$ are
$43+44$ and $1+43+43$. These choices require at least $223$ and
$399$ additional vertices, respectively, both greater than $E=163$.
\item If $R=18$, two branches with root degree three would require at least
$2\cdot130>160=E$ additional vertices, so at most one can occur.
At most two branches can have root degree at least $44$, since three
would contribute at least $3\cdot42=126>D$ to the sum of $k-2$.
The sum of $k-2$ is then at most $1+2\cdot44=89<112=D$.
\item If $R=19$, two branches with root degree three would require at least
$2\cdot127>157=E$ additional vertices, so at most one can occur.
At most three branches can have root degree at least $43$, since four
would contribute at least $4\cdot41=164>135=D$ to the sum of $k-2$.
That sum is therefore at most $1+3\cdot43=130<135=D$.
\item If $R=44$ and a branch with root degree three occurs, at most
$82-52=30$ additional vertices remain for the other branches.
Thus at most one branch can have root degree greater than three,
and its root degree must be $20$.
The sum of $k-2$ is then at most $1+18=19<60=D$.
Without a branch with root degree three, three or fewer branches with root degree
greater than two contribute at most $54$ to this sum.
Five such branches would require at least $5\cdot18=90>82=E$ additional vertices.
Thus exactly four would be needed. Four branches with root degree $20$ require at least
$72$ additional vertices, leaving at most $82-72=10$.
Replacing any of their root degrees by a smaller allowed value greater than two
increases the minimum order by at least $18$, so all four must have root degree $20$.
They would then contribute $72$ to the sum of $k-2$, whereas $D=60$.
\item If $R=45$ and a branch with root degree three occurs, at most
$79-49=30$ additional vertices remain for the other branches.
Table~\ref{tab:singleton-certificate} then permits at most one further
branch with root degree greater than two, and that degree must be $19$.
Thus the sum of $k-2$ is at most $1+17=18<D$.
If no branch with root degree three occurs, at most two branches can have root degree greater than two,
since three such branches contribute at least $39>D$ to this sum.
One branch contributes at most $17$, so exactly two are needed.
The only pairs of contributions $k-2$ summing to $D=31$ are $14+17$ and
$15+16$. They require at least $79$ and $83$ additional vertices, respectively.
Only the pair $14+17$ satisfies the bound $E=79$. Thus there is one branch with root degree
$16$, one with root degree $19$, and $43$ with root degree $2$.
\end{enumerate}

For the surviving row, $s=21$ and $\eta=38$, and the total branch
order is $43\cdot3+65+20=214$. The branches with root degrees $2$ and $19$
have $T=0$ in \eqref{eq:branch-cost}, and hence are stars with two
and $19$ leaves, respectively. For the branch with root degree $16$,
$L=5$ and $T=42=8\cdot5+2$.
Proposition~\ref{prop:minimum-seven}\textup{(i)} therefore gives eight child
parameters equal to $5$, one equal to
$2$, and seven equal to $0$. These give eight children each with five leaf children, one child with two children each having three leaf children, and seven leaf children, exactly as in the branch $B$ in Example~\ref{ex:218}. Since these parameters
determine each branch up to rooted isomorphism, every order-$218$
realization is isomorphic to the graph in that example.


\begingroup\small
\begin{thebibliography}{99}
\bibitem{bose} S.~Bose, \href{https://doi.org/10.1103/PhysRevLett.91.207901}{Quantum communication through an unmodulated spin chain}, \emph{Physical Review Letters} \textbf{91} (2003), 207901.
\bibitem{asymmetry} R.~O.~Braga, J.~C.~Moraes, and M.~C.~Santos, \href{https://arxiv.org/abs/2610.22773v1}{Spectral asymmetry and odd cycles in integral graphs}, arXiv:2610.22773v1 (2026).
\bibitem{christandl} M.~Christandl, N.~Datta, A.~Ekert, and A.~J.~Landahl, \href{https://doi.org/10.1103/PhysRevLett.92.187902}{Perfect state transfer in quantum spin networks}, \emph{Physical Review Letters} \textbf{92} (2004), 187902.
\bibitem{coutinho} G.~Coutinho, \href{https://uwspace.uwaterloo.ca/items/d47f967a-fadf-4eeb-8cd8-d7852a53ff95}{\emph{Quantum State Transfer in Graphs}}, Ph.D. thesis, University of Waterloo, 2014.
\bibitem{supportdistance} G.~Coutinho, \href{https://doi.org/10.1016/j.disc.2018.04.011}{Quantum walks and the size of the graph}, \emph{Discrete Mathematics} \textbf{342} (2019), 2765--2769.
\bibitem{algorithm} G.~Coutinho and C.~Godsil, \href{https://doi.org/10.26421/QIC17.5-6-8}{Perfect state transfer is poly-time}, \emph{Quantum Information and Computation} \textbf{17} (2017), 495--502.
\bibitem{bridges} G.~Coutinho, C.~Godsil, E.~Juliano, and C.~M.~van Bommel, \href{https://doi.org/10.1016/j.laa.2022.07.009}{Quantum walks do not like bridges}, \emph{Linear Algebra and its Applications} \textbf{652} (2022), 155--172.
\bibitem{trees} G.~Coutinho, E.~Juliano, and T.~J.~Spier, \href{https://doi.org/10.1016/j.jctb.2024.04.004}{No perfect state transfer in trees with more than 3 vertices}, \emph{Journal of Combinatorial Theory, Series B} \textbf{168} (2024), 68--85.
\bibitem{decorated} G.~Coutinho, E.~Juliano, and T.~J.~Spier, \href{https://doi.org/10.1016/j.laa.2025.02.011}{The spectrum of symmetric decorated paths}, \emph{Linear Algebra and its Applications} \textbf{711} (2025), 17--39.
\bibitem{damnjanovic} I.~Damnjanovi\'c, \href{http://operator.pmf.ni.ac.rs/amcs/volumes/2023/amcs-2023-7-1-1.pdf}{Computing the characteristic polynomials of rooted trees and the energies of Bethe trees}, \emph{Applied Mathematics and Computer Science} \textbf{7} (2023), 1--16.
\bibitem{godsil} C.~Godsil, \href{https://doi.org/10.1016/j.disc.2011.06.032}{State transfer on graphs}, \emph{Discrete Mathematics} \textbf{312} (2012), 129--147.
\bibitem{guozhou} H.~Guo and B.~Zhou, \href{https://doi.org/10.2298/AADM180210022G}{On the $\alpha$-spectral radius of graphs}, \emph{Applicable Analysis and Discrete Mathematics} \textbf{14} (2020), 431--458.
\bibitem{hu} S.~Hu, \href{https://doi.org/10.1016/j.disc.2006.06.023}{The largest eigenvalue of unicyclic graphs}, \emph{Discrete Mathematics} \textbf{307} (2007), 280--284.
\bibitem{twins} S.~Kirkland, H.~Monterde, and S.~Plosker, \href{https://doi.org/10.1007/s10801-023-01261-3}{Quantum state transfer between twins in weighted graphs}, \emph{Journal of Algebraic Combinatorics} \textbf{58} (2023), 623--649.
\bibitem{limathias} C.-K.~Li and R.~Mathias, \href{https://doi.org/10.1007/s002110050397}{The Lidskii--Mirsky--Wielandt theorem---additive and multiplicative versions}, \emph{Numerische Mathematik} \textbf{81} (1999), 377--413.
\end{thebibliography}
\endgroup
\end{document}